\documentclass[11pt,dvipsnames,reqno,hypertexnames=false]{amsart}
\usepackage[T1]{fontenc}
\usepackage{a4wide}
\usepackage{yhmath,mathrsfs,amsthm,amsmath,amssymb,amsfonts,enumerate,lipsum, appendix,mathtools, bm}
\usepackage{bbold}
\usepackage{pdfpages}
\usepackage[mathscr]{euscript}
\usepackage{thmtools}
\usepackage{graphicx}
\usepackage{pgf,tikz}
\usepackage{tkz-euclide}
\usepackage{cancel}
\usepackage{soul}
\usepackage{pgfplots}
\pgfplotsset{compat=1.17}
\usetikzlibrary{positioning, arrows.meta, calc, shapes.geometric, decorations.pathreplacing, positioning, backgrounds, fit}

\definecolor{darkblue}{RGB}{0, 0, 139}
\DeclareMathOperator{\supp}{supp}
\usepackage{caption}
\usepackage{subcaption}
\usetikzlibrary{shapes.geometric}
\usetikzlibrary{arrows,hobby}
\usepackage[shortlabels]{enumitem}
\usepackage[english]{babel}

\allowdisplaybreaks

\usepackage[sort,numbers]{natbib}
\usepackage{thmtools}
\usepackage{hyperref}
\usepackage{fancyhdr}

\hypersetup{
    colorlinks=true,
    linkcolor=blue,
    filecolor=black,
    urlcolor=blue,
    citecolor=red,
}
\usepackage{orcidlink} % MUST be loaded after hyperref
\usepackage{cleveref}  % MUST be loaded after hyperref and orcidlink
\usepackage[margin=0.90in, heightrounded]{geometry}
\usepackage{bigints}
\usepackage{url}
\usepackage{esint}
\newtheorem{theorem}{Theorem}[section]
\newtheorem{lemma}[theorem]{Lemma}
\newtheorem{proposition}[theorem]{Proposition}

\theoremstyle{definition}

\newtheorem{remark}[theorem]{Remark}
\numberwithin{equation}{section}
\usepackage[hyperpageref]{backref}
\newcommand*\rd{\mathbb{R}^d}

\newcommand*\N{\Sigma_1^{\al}}
\newcommand{\al} {\alpha}
\newcommand{\pa} {\partial}

\newcommand{\de} {\delta}

\newcommand{\la} {\lambda}
\newcommand{\La} {\Lambda}

\newcommand{\Gr} {\nabla}
\newcommand{\no} {\nonumber}
\newcommand{\noi} {\noindent}
\newcommand{\var} {\varepsilon}
\newcommand{\ra} {\rightarrow}
\newcommand{\embd} {\hookrightarrow}

\newcommand{\bee} {\begin{equation}}
	\newcommand{\eee} {\end{equation}}
\newcommand{\bea} {\begin{eqnarray}}
	\newcommand{\eea} {\end{eqnarray}}
\newcommand{\Bea} {\begin{eqnarray*}}
	\newcommand{\Eea} {\end{eqnarray*}}
\def\d{\,{\rm d}}
\def\dx{\,{\rm d}x}
\def\dy{\,{\rm d}y}

\def\dr{\,{\rm d}r}

\def\C{{\mathcal C}}
\def\D{{\mathcal D}}

\def\R{{\mathbb R}}
\def\N{{\mathbb N}}

\def\({{\Big(}}
\def\){{\Big)}}
\def\cc{{\C_c^\infty}}
\def\p{{p^{\prime}}}

\def\dx{\,{\rm d}x}
\def\dxy{\,{\rm d}x\,{\rm d}y}

\def\dxnyn{\,{\rm d}\bar{x}_n\,{\rm d}\bar{y}_n}
\def\dz{\,{\rm d}z}

\def\dr{\,{\rm d}r}

\DeclarePairedDelimiter\abs{\lvert}{\rvert}
\DeclarePairedDelimiter\norm{\lVert}{\rVert}
\def\wps{{\mathcal{D}^{s,p}}}
\usepackage{orcidlink}
\def\l@subsection{\@tocline{2}{0pt}{2pc}{6pc}{}}
\def\l@subsubsection{\@tocline{3}{0pt}{8pc}{8pc}{}}
\makeatother
\usepackage{xpatch}
\makeatletter   
\xpatchcmd{\@tocline}
{\hfil\hbox to\@pnumwidth{\@tocpagenum{#7}}\par}
{\ifnum#1<0\hfill\else\dotfill\fi\hbox to\@pnumwidth{\@tocpagenum{#7}}\par}
{}{}

\def\l@subsection{\@tocline{2}{0pt}{2pc}{6pc}{}}
\def\l@subsubsection{\@tocline{3}{0pt}{8pc}{8pc}{}}
\makeatother
\usepackage{fancyhdr}
\title[Critical Fractional $p$-Hardy-Sobolev Equations]{Critical Fractional $p$-Hardy Sobolev Equations: Global compactness and multiplicity of positive solutions}
\author[N. Biswas, S. Chakraborty and D. Mukherjee]{Nirjan Biswas$^1$\,\orcidlink{0000-0002-3528-8388} \and Souptik Chakraborty$^2$\,\orcidlink{0009-0004-1867-0560} \and Debangana Mukherjee$^3$\,\orcidlink{0000-0002-1516-840X}}
\address{\rm  $^1$ Department of Mathematics, Indian Institute of Science Education and Research Pune \\
Dr. Homi Bhabha Road, Maharashtra 411008, India}
\address{\rm $^2$ Gandhi Institute of Technology and Management (GITAM) University, Department of Mathematics and Statistics, Hyderabad, Telangana,  502329, India.}
\address{\rm $^3$ Krea University, School of Interwoven Arts and Sciences (SIAS), Sricity, India}
\email{nirjaniitm@gmail.com, soupchak9492@gmail.com, debangana.mukherjee@krea.edu.in}
\thanks{$^3$Corresponding author}
\begin{document}
\begin{abstract}
		We study the critical fractional $p$-Hardy-Sobolev equation
		\begin{equation}\tag{$\mathcal{P}$}\label{a-main}
				(-\Delta_p)^s u -\mu\dfrac{\abs{u}^{p-2}u}{|x|^{sp}}=\dfrac{|u|^{p^*_s(\alpha)-2}u}{|x|^{\alpha}}+f \;\mbox{ in }\,\mathbb{R}^d, \quad u\in \mathcal{D}^{s,p}(\mathbb{R}^d),
		\end{equation}
		where $1<p<\infty$, $0<s<1$, $0\leq\alpha<sp<d$, $\mu>0$, $p^*_s(\al):= p(d-\al)/(d-sp)$ is the critical Hardy-Sobolev exponent, and $f$ is a nontrivial nonnegative functional in $(\mathcal{D}^{s,p}(\mathbb{R}^d))^*$. We first establish global compactness results for Palais-Smale sequences associated with the corresponding energy functional. When $\alpha>0$, the loss of compactness is described by dilations of solutions of the Hardy-Sobolev limit problem. The case $\alpha=0$ has a different structure: in addition to Hardy profiles, pure Sobolev profiles may occur when the centre of concentration escapes from the Hardy singularity relative to its scale. We give a direct centre-scale analysis of these two concentration regimes and obtain the corresponding energy decomposition and profile separation. As an application, under an explicit smallness assumption on $f$, we first obtain a positive solution for \eqref{a-main} with negative energy. We then construct a nonlinear path based on hidden convexity whose energy remains strictly below the first bubbling threshold. A minimax argument, combined with the global compactness theorem, then yields a second distinct positive solution for \eqref{a-main}.
	\end{abstract}
\maketitle
\tableofcontents
\section{Introduction}
In critical elliptic problems, variational methods naturally encounter a significant loss of compactness. Because the Sobolev embedding is continuous rather than compact at the critical exponent, bounded Palais-Smale sequences are not guaranteed to converge strongly. Instead, their energy can accumulate at various locations or scales. Understanding how this energy concentrates is crucial for ensuring compactness at specific energy levels, thereby allowing us to establish the existence and multiplicity of weak solutions. Hardy-Sobolev problems introduce a further complication due to their singular potential: the Hardy term singles out the origin, in contrast to the potentially different invariance structure of the critical nonlinear term.

For $s\in (0,1), p \in (1, \infty)$, and $d>sp$, we consider the following non-homogeneous equation driven by the fractional $p$-Laplace Hardy Sobolev operator:
\begin{equation}\tag{$\mathcal{P}_{\al}$}\label{MainEq}
			(-\Delta_p)^s u -\mu\dfrac{\abs{u}^{p-2}u}{|x|^{sp}}=\dfrac{|u|^{p^*_s(\al)-2}u}{|x|^{\al}}+f \;\mbox{ in }\,\mathbb{R}^d,\,\,u>0 \text{ on }\rd, \, u \in \wps(\rd),\, f \in (\wps(\rd))^*,
\end{equation}
where $\mu >0$ is a parameter and  the critical Hardy-Sobolev exponent $p^*_s(\al):=\frac{p(d-\al)}{d-sp}$, where $\al \in [0,sp)$. Here $f\in (\wps(\rd))^*$ is a nonnegative nontrivial functional, i.e., $\prescript{}{(\wps)^*}{\langle}f,\phi\rangle_{\wps}\geq 0$ whenever $\phi \geq 0$ in $\wps(\rd)$. The fractional $p$-Laplace operator $(-\Delta _p)^s$ is defined as
\begin{equation*}
   (-\Delta_{p})^{s}u(x) = C(d,s,p) \text{ P.V.} \int_{\mathbb{R}^{d}} \frac{|u(x) - u(y)|^{p-2}(u(x)-u(y))}{|x-y|^{d+sp}}\, \dy, \; \text{for}~x \in \mathbb{R}^{d},
\end{equation*} 
where $C(d,s,p)$ is the normalization constant, and P.V. denotes the principal value. The Sobolev space $\wps(\rd)$ (we call it $\wps$) is defined as the completion of $\mathcal{C}^{\infty}_{c}(\rd)$ under the Gagliardo seminorm
$$[u]_{s,p}\coloneqq \left( \; \iint_{\rd \times \rd} \frac{|u(x)-u(y)|^p}{|x-y|^{d+sp}}\dxy \right)^{\frac{1}{p}}.$$
The space $\wps$ has the following characterization (see \cite[Theorem 3.1]{Brasco2019characterisation}): 
\begin{equation*}
\wps (\rd) \coloneqq\left\{u\in L^{p^*_s}(\rd):[u]_{s,p}<\infty\right\},
\end{equation*}
where $p^*_s =\frac{dp}{d-sp}$ is the critical exponent, $[\cdot]_{s,p}$ is an equivalent norm in $\wps (\rd)$ and it is a reflexive Banach space. For details on $\wps$ and its associated embedding results, we refer to \cite{Brasco2019characterisation, DiPaVa}. Recall the  $p$-fractional Hardy inequality (see \cite[Theorem 1.1]{Frank-JFA-2008}): 
\begin{align}\label{HS}
    \mu_{d,s,p} \int_{\rd} \frac{\abs{u(x)}^p}{\abs{x}^{sp}} \dx \le \iint_{\rd \times \rd}\frac{|u(x)-u(y)|^p}{|x-y|^{d+sp}}\dxy, \; \forall \, u \in \wps,
\end{align}
where $\mu_{d,s,p}$ is the best constant (see \cite[Theorem 1.1]{Frank-JFA-2008}). 
If $0 <\mu<\mu_{d,s,p}$, then 
\begin{align}\label{equivalent-norm}
    [u]_{\mu} := \left( [u]_{s,p}^p - \mu \int_{\mathbb{R}^d}\frac{|u|^p}{|x|^{sp}} \dx \right)^{\frac{1}{p}}, 
\end{align}
is an equivalent norm in $\wps (\rd)$, i.e., there exists $C_{\text{eqiv}}>0$ such that $C_{\text{eqiv}} [u]_{s,p} \le [u]_{\mu}$, for every $u \in \wps$. Next, the following Hardy-Sobolev inequality is the interpolation between \eqref{HS} and the critical Sobolev inequality $\norm{u}_{L^{p^*_s}(\rd)} \le C(d,s,p) [u]_{s,p}$, for all $u \in \wps (\rd)$:
\begin{equation}\label{HS1}
    C(d,s,p,\al) \left( \int_{\rd} \frac{\abs{u(x)}^{p^*_s(\al)}}{\abs{x}^{\al}} \dx \right)^{\frac{p}{p^*_s(\al)}} \le  \iint_{\rd \times \rd}\frac{|u(x)-u(y)|^p}{|x-y|^{d+sp}}\dxy, \; \forall \, u \in \wps,
\end{equation}
We say $u$ is a weak solution of \eqref{MainEq}, if the following identity holds:
\begin{align}\label{weak1}
    &\iint_{\R^{d} \times \R^d} \frac{|u(x)-u(y)|^{p-2}(u(x)-u(y))(\phi(x)-\phi(y))}{|x-y|^{d+sp}} \dxy - \mu \int_{\rd} \frac{\abs{u(x)}^{p-2} u(x)}{\abs{x}^{sp}} \phi(x) \dx \no \\ 
    &=\int_{\rd} \frac{\abs{u(x)}^{p^*_s(\al)-2}u(x)}{\abs{x}^{\al}} \phi(x) \dx +\prescript{}{(\wps)^*}{\langle}f,\phi{\rangle}_{\wps}, \; \forall \, \phi \in \wps,
\end{align}
 For $\mu \in (0,\mu_{d,s,p})$, we consider the following energy functional associated with \eqref{MainEq}:
\begin{align}\label{energy}
    I_{\mu,\al,f}(u) := \frac{1}{p}[u]_{s,p}^p - \frac{\mu}{p} \int_{\rd}\frac{\abs{u}^p}{\abs{x}^{sp}} \dx - \frac{1}{p^*_s(\al)} \int_{\rd} \frac{\abs{u}^{p^*_s(\al)}}{\abs{x}^{\al}}\dx - \prescript{}{(\wps)^*}{\langle}f,u{\rangle}_{\wps}, \; \forall \, u \in \wps.
\end{align}
In view of \eqref{equivalent-norm} and \eqref{HS1}, $I_{\mu,\al,f}$ is well-defined. Further, $I_{\mu,\al,f} \in \mathcal{C}^1(\wps, \R)$, and a critical point $I_{\mu,\al,f}$ corresponds to a weak solution to \eqref{MainEq}.

\subsection{Main results}
A sequence $\{ u_n \} \subset \wps$ is said to be a Palais-Smale (PS) sequence for $I_{\mu,\al,f}$ at level $\eta$, if $I_{\mu,\al,f}(u_n) \ra \eta$ in $\R$ and $I'_{\mu, \al,f}(u_n) \ra 0$ in $(\wps)^*$ as $n \ra \infty$. The functional $I_{\mu,\al,f}$ is said to satisfy the (PS) condition at level $\eta$, if every (PS) sequence at level $\eta$ has a convergent subsequence. Due to the noncompactness of the embedding $\wps \hookrightarrow L^{p^*_s(\al)}(\rd, \abs{x}^{-\al})$,  every (PS) sequence of $I_{\mu,\al,f}$ does not converge strongly in $\wps$. Moreover, the weak limit of the (PS) sequence can be zero even if $\eta>0$. 

In this paper, we classify the (PS) sequence of $I_{\mu,\al,f}$. For $0 \le \mu < \mu_{d,s,p}$, we consider the following quantity: 
\begin{align}\label{best-constant}
    S_{\mu,\al}(d,s,p,\al) := \inf_{u \in \wps \setminus \{ 0\}} \frac{[u]_{s,p}^p - \mu \displaystyle   \int_{\rd} \frac{\abs{u(x)}^p}{\abs{x}^{sp}}\dx }{\displaystyle \left(\int_{\rd} \frac{\abs{u(x)}^{p^*_s(\al)}}{\abs{x}^{\al}} \dx\right)^{\tfrac{p}{p_s^*(\al)}}}.
\end{align}
It is known from \cite[Theorem 1.1]{Brasco2016} that $S_{0,0}>0$ is attained by a positive extremum, which is radially symmetric, radially decreasing at the origin, and has a certain decay at infinity.
For $\mu >0$, in \cite[Theorem 1.1]{Shen24} (when $\al=0$) and in \cite[Theorem 1.2]{ARSJ2020} (when $\al>0$), the authors proved that $S_{\mu,\al}>0$ is attained by a non-negative extremum which is again positive, radially symmetric and radially decreasing with respect to the origin. These extrema (up to a multiplication of a normalized constant) satisfy the following equations weakly: 
\begin{equation}\label{limit-problem-intro}
\begin{aligned}
   &(-\Delta_p)^s u = |u|^{p^*_s-2}u \;\mbox{ in }\,\mathbb{R}^d, \; u \in \wps, \\
   &(-\Delta_p)^s u -\mu\dfrac{\abs{u}^{p-2}u}{|x|^{sp}}=\frac{|u|^{p^*_s(\al)-2}u}{\abs{x}^{\al}} \;\mbox{ in }\,\mathbb{R}^d, \; u \in \wps. 
\end{aligned}
\end{equation} 
We first consider the weighted case $0<\alpha<sp$, where the critical weight fixes the origin, and the loss of compactness is generated only by dilations. This leads to the following global compactness result.

\begin{theorem}[PS decomposition: for $\al>0$]\label{PS-decomposition}
 Let $p \in (1, \infty), s \in (0,1), sp \in (\al, d)$, and $\mu \in (0,\mu_{d,s,p})$. Assume that $f \in (\wps)^*$. Let $\{u_n\}$ be a \textup{(PS)} sequence for $I_{\mu,\al,f}$ at level $\eta$. Then there exists a subsequence \textup{(still denoted by $\{ u_n \}$)} for which the following hold:
  
  \noi there exist $k\in \N \cup \{0\}$, sequence $\{ r_n^i \}_n \subset \R^+$, functions $u, \tilde{u}_i \in \wps$, for $1\le i\le k$, such that $u$ weakly satisfies \eqref{MainEq} without sign assumptions, $\tilde{u}_i$ weakly satisfies
   \begin{align}\label{homogeneous}
    (-\Delta_p)^s \tilde{u}_i -\mu\dfrac{\abs{\tilde{u}_i}^{p-2}\tilde{u}_i}{|x|^{sp}}=\frac{|\tilde{u}_i|^{p^*_s(\al)-2}\tilde{u}_i}{\abs{x}^{\al}} \;\mbox{ in }\,\mathbb{R}^d, \; \tilde{u}_i \in \wps \setminus \{0\},
\end{align}
such that 
\begin{align*}
    &(a) \textbf{ Profile decomposition: }\text{the (PS) sequence decomposes as } \; u_n = u + \sum_{i=1}^{k} C_{r_n^i}(\tilde{u}_i) + o_n(1), \\ 
    &(b) \textbf{ Energy decoupling: } \eta = I_{\mu,\al,f}(u) + \sum_{i=1}^{k} I_{\mu,\al,0}(\tilde{u}_i), \\ 
    &(c) \textbf{ Behavior of parameters and asymptotic orthogonality: } r_n^i \ra 0 \text{ or } \infty, \text{ and } \\
    &\qquad \left| \log \left( \frac{r_n^i}{r_n^j} \right) \right| \rightarrow \infty, \text{ for } i \neq j, 1 \le i, j \le k,
\end{align*}
where $o_n(1) \ra 0$ as $n \ra \infty$ in $\wps$, $C_{r_n^i}\tilde{u}_i(x) \coloneqq (r_n^i)^{-\frac{d-sp}{p}}\tilde u_i(\tfrac{x}{r_n^i})$. In the case $k=0$, the above expression holds without $\tilde{u}_i$ and $r_n^i$.
\end{theorem}
The global compactness theorem allows us to rule out the possibility of bubbling for the minimizing sequence at some particular energy levels constructed below. As a first application, we obtain a positive solution of negative energy.

\begin{theorem}[Existence: for $\al>0$]\label{existence}
Let $p \in (1, \infty), s \in (0,1), 0<\al< sp<d$, and $\mu \in (0,\mu_{d,s,p})$. Assume that $f$ is a nontrivial nonnegative functional in $(\wps)^*$, and 
    $$\norm{f}_{(\wps)^*} \le \left( \frac{1}{p} -  \frac{p-1}{p^*_s(\al)(p^*_s(\al) -1)}\right) \left( \frac{p-1}{p^*_s(\al)-1} S_{\mu,\al}^{\frac{p^*_s(\al)}{p}} \right)^{\frac{p-1}{p^*_s(\al) - p}},$$
    where $S_{\mu,\al}$ is defined in \eqref{best-constant}.
    Then \eqref{MainEq} admits a positive weak solution $u_{\al}$ with $I_{\mu,\al,f}(u_{\al})<0$.
\end{theorem} 
The endpoint case $\alpha=0$ has a different compactness structure. The critical nonlinearity is now translation invariant, whereas the Hardy potential still distinguishes the origin. Consequently, in addition to Hardy profiles, pure Sobolev profiles may occur. The corresponding global compactness theorem is stated next.

\begin{theorem}[PS decomposition: for $\al=0$]\label{PS-decomposition-II}
 Let $p \in (1, \infty), s \in (0,1)$, $d>sp$ and $\mu \in (0,\mu_{d,s,p})$. Assume that $f \in (\wps)^*$. Let $\{u_n\}$ be a \textup{(PS)} sequence for $I_{\mu,0,f}$ at level $\eta$. Then there exists a subsequence \textup{(still denoted by $\{ u_n \}$)} for which the following hold:

\noi there exist $k_1,\, k_2 \in \N \cup \{0\}$, sequence $\{ r_n^i \}_n \subset \R^+$ for $1\le i\le k_1$, and sequences $\{ x_n^j \}_n \subset \rd, \{ R_n^j \}_n \subset \R^+$ for $1\le j\le k_2$, functions $u, \tilde{u}_i, \tilde{U}_j$ {\rm(}where $1\le i\le k_1$ and $1\le j \le k_2${\rm)} such that $u$ weakly satisfies \eqref{MainEq} {\rm(}with $\al=0${\rm)} without sign assumptions, $\tilde{u}_i$ weakly satisfies   
  \begin{align*}
    (-\Delta_p)^s \tilde{u}_i -\mu\dfrac{\abs{\tilde{u}_i}^{p-2}\tilde{u}_i}{|x|^{sp}}= |\tilde{u}_i|^{p^*_s-2}\tilde{u}_i \;\mbox{ in }\,\mathbb{R}^d, \; \tilde{u}_i \in \wps \setminus \{0\},
\end{align*}
and $\tilde{U}_j$ weakly satisfies
\begin{align*}
    (-\Delta_p)^s \tilde{U}_j=|\tilde{U}_j|^{p^*_s-2}\tilde{U}_j \;\mbox{ in }\,\mathbb{R}^d, \; \tilde{U}_j \in \wps \setminus \{0\},
\end{align*}
such that 
\begin{align*}
    &(a) \textbf{ Profile decomposition: }\text{the (PS) sequence decomposes as}\\
    &\qquad\qquad\qquad\qquad\qquad u_n = u + \sum_{i=1}^{k_1} C_{r_n^i}(\tilde{u}_i) + \sum_{j=1}^{k_2} C_{x_n^j, R_n^j}(\tilde{U}_j) + o_n(1), \\
    &(b) \textbf{ Energy decoupling: } \eta = I_{\mu,0,f}(u) + \sum_{i=1}^{k_1} I_{\mu,0,0}(\tilde{u}_i) + \sum_{j=1}^{k_2} I_{0,0,0}(\tilde{U}_j), \\ 
    &(c)\textbf{ Behavior of parameters: } r_n^i \ra 0 \text{ or } \infty, \text{ for } 1\le i\le k_1,\\
    &\text{ and } x_n^j \ra x^j \in \rd \text{ or } \abs{x_n^j} \ra \infty, \frac{R_n^j}{\abs{x_n^j}} \ra 0, \text{ for } 1\le j \le k_2,\\
    &(d) \textbf{ Asymptotic orthogonality: }\left| \log \left( \frac{r_n^i}{r_n^j} \right) \right| \rightarrow \infty, \text{ for } i \neq j, 1 \le i, j \le k_1, \text{ and } \\
    &\qquad \left| \log \left( \frac{R_n^i}{R_n^j} \right) \right| + \left| \frac{x_n^i - x_n^j}{R_n^i}  \right| \rightarrow \infty, \text{ for } i \neq j, 1 \le i, j \le k_2,
\end{align*}
where $o_n(1) \ra 0$ as $n \ra \infty$ in $\wps$, $C_{r_n^i}\tilde{u}_i(x) \coloneqq (r_n^i)^{-\frac{d-sp}{p}}\tilde u_i\left(\frac{x}{r_n^i}\right)$, $C_{x_n^j, R_n^j}\tilde U_j(x)\coloneqq (R_n^j)^{-\frac{d-sp}{p}}\tilde U_j\left(\frac{x-x_n^j}{R_n^j}\right)$. In the case $k_1=0$ and $k_2=0$, the above expression holds without $\tilde{u}_i, r_n^i, \tilde{U}_j, R_n^j$, and $x_n^j$.
\end{theorem}
Although the profile decomposition is more delicate when $\alpha=0$, the same variational strategy can be used once both types of concentration are excluded. This yields the following existence result. 

\begin{theorem}[Existence: for $\al=0$]\label{existence-II}
Let $p \in (1, \infty), s \in (0,1), d>sp$, and $\mu \in (0,\mu_{d,s,p})$. Assume that $f$ is a nontrivial nonnegative functional in $(\wps)^*$, and 
    $$\norm{f}_{(\wps)^*} \le \left( \frac{1}{p} -  \frac{p-1}{p^*_s(p^*_s -1)}\right) \left( \frac{p-1}{p^*_s-1} S_{\mu}^{\frac{p^*_s}{p}} \right)^{\frac{p-1}{p^*_s - p}},$$
    where $S_{\mu}$ is defined in \eqref{best-constant}.
    Then \eqref{MainEq} admits a positive weak solution $u_{0}$ with $I_{\mu,0,f}(u_{0})<0$.
\end{theorem} 
Theorem \ref{existence} and Theorem \ref{existence-II} provide a first positive solution with negative
energy. Our final result shows that the global compactness theory can also be used at a higher min-max critical level to produce a second positive solution, uniformly for $0\leq\alpha<sp$.

\begin{theorem}[Multiplicity]\label{Thm:second}
Let $p \in (1,\infty)$, $s \in (0,1)$, $0 \le \al<sp<d$ and $\mu \in (0,\mu_{d,s,p})$. Assume that $f \in (\wps)^*$ is nontrivial and nonnegative and satisfies
\begin{align}\label{eq:threshold2}
    \norm{f}_{(\wps)^*}\le \left( \frac1p-\frac{p-1}{p^*_s(\al)\left( p^*_s(\al)-1\right)}\right)\left(\frac{p-1}{p^*_s(\al)-1}\,S_{\mu}^{\frac{p^*_s(\al)}{p}}\right)^{\frac{p-1}{p^*_s(\al)-p}}.
\end{align}
Then \eqref{MainEq} admits at least two distinct positive weak solutions $u_{\al}$ and $v_{\al}$, with
\begin{align*}
    I_{\mu,\al,f}(v_{\al})=\eta \ge c_1^{\al}>c_0^{\al}=I_{\mu,\al,f}(u_{\al}).
\end{align*}
\end{theorem}

\subsection{Literature review}
The classification of Palais-Smale (PS) sequences at critical growth was pioneered by Struwe \cite{Struwe}, who studied the energy functional
\begin{align*}
    I_{\la}(u) = \frac{1}{2} \int_{\Omega} \abs{\Gr u}^2 - \frac{\la}{2} \int_{\Omega} u^2 - \frac{1}{2^*} \int_{\Omega} \abs{u}^{2^*},\; u \in \mathcal{D}_0^{1,2}(\Omega),
\end{align*}
where $\la \in \R$, $\Omega \subset \rd$ is a smooth bounded domain, $d>2$, $2^*=\frac{2d}{d-2}$, and $\mathcal{D}_0^{1,2}(\Omega)= \{ u \in L^{2^*}(\rd) : \abs{\Gr u} \in L^2(\rd), u=0 \text{ in } \rd \setminus \Omega \}$. Critical points of $I_\la$ are precisely the weak solutions of the Br\'ezis-Nirenberg problem
\begin{equation}\label{brezis-nirernberg}
   -\Delta u = \la u + \abs{u}^{2^*-2}u \text{ in } \Omega, \quad u=0 \text{ on } \pa \Omega,
\end{equation}
first introduced and studied by Br\'{e}zis and Nirenberg \cite{Brezis-Nirenberg} using sharp test-function expansions against the Aubin-Talenti extremals of the Sobolev inequality \cite{Aubin, Talenti}. Struwe \cite{Struwe} proved that if $\{u_n\}$ is a (PS) sequence of $I_\la$ at level $c$, then there exist an integer $k\ge0$, sequences $\{x_n^i\}_n\subset\rd$, $\{r_n^i\}_n\subset\R^+$, and functions $u\in\mathcal{D}_0^{1,2}(\Omega)$, $\tilde u_i \in \mathcal{D}^{1,2}(\rd)$ for $1\le i\le k$ (where $\mathcal{D}^{1,2}(\rd)=\{u\in L^{2^*}(\rd):\abs{\Gr u}\in L^2(\rd)\}$), such that $u$ weakly solves \eqref{brezis-nirernberg}, each $\tilde u_i$ weakly solves the limit problem $-\Delta \tilde u_i = \abs{\tilde u_i}^{2^*-2}\tilde u_i$ in $\rd$, and
\begin{align*}
    u_n = u + \sum_{i=1}^k \tilde{u}_i^{x_n^i,r_n^i} + o_n(1) \quad \text{in } \mathcal{D}^{1,2}(\rd), \quad \tilde{u}_i^{y,r}(x) \coloneqq  r^{-\frac{d-2}{2}} \tilde{u}_i \left( \frac{x-y}{r} \right),
\end{align*}
with energy quantization
\begin{align*}
    c = I_{\la}(u) + \sum_{i=1}^k I_{\infty}(\tilde{u}_i) + o_n(1), \quad I_{\infty}(u) = \frac{1}{2} \int_{\rd} \abs{\Gr u}^2 - \frac{1}{2^*} \int_{\rd} \abs{u}^{2^*}.
\end{align*}
This decomposition became the basic tool for locating the compactness threshold in critical elliptic problems and, combined with the min-max scheme, for producing multiplicity results. Shortly after, Cerami, Fortunato and Struwe \cite{CFS} used a Struwe-type decomposition to show that the number of nontrivial solutions of \eqref{brezis-nirernberg} is bounded below by the number of eigenvalues of $-\Delta$ lying in a suitable spectral window depending on the best Sobolev constant and $|\Omega|$.

On unbounded domains, the loss of compactness is more severe, since translation invariance is not broken by a Dirichlet boundary condition, and bubbling may occur at any point drifting to infinity along a whole sequence of directions, not merely by concentration at a fixed boundary point. Existence results in this setting were first obtained under symmetry or exterior-domain assumptions by Esteban and Lions \cite{Esteban-Lions}, Benci and Cerami \cite{Benci-Cerami}, and Bahri and Lions \cite{Bahri-Lions}. Ramos, Wang and Willem \cite{Ramos-Wang-Willem} obtained a global compactness decomposition for positive solutions of critical elliptic equations on unbounded domains, allowing bubbling both at finite points and ``at infinity''. Further developments and applications of Palais-Smale decomposition lemmas on unbounded and exterior domains include the work of Cerami and Molle \cite{Cerami-Molle}, and Molle and Passaseo \cite{Molle-unbdd}.

The extension of Struwe's decomposition to the $p$-Laplace operator $-\Delta_p u = -\operatorname{div}(\abs{\Gr u}^{p-2}\Gr u)$ required substantially new tools, since the loss of the Hilbert-space structure prevents the elegant orthogonal-splitting arguments available for $p=2$. Existence results for the corresponding Br\'{e}zis-Nirenberg problem for the $p$-Laplacian were obtained by Garc\'{i}a,  Azorero and Peral \cite{GarciaAzorero-Peral1, GarciaAzorero-Peral2}, Guedda and V\'{e}ron \cite{Guedda-Veron}, and Egnell \cite{Egnell}, generally via delicate direct test-function estimates rather than a full classification of (PS) sequences. A complete Struwe-type global compactness decomposition for the $p$-Laplacian at critical growth, valid for all $p\in(1,d)$ and non-negative (PS) sequences, was obtained by Mercuri and Willem \cite{Mercuri-Willem}, who established a representation theorem showing that any such sequence decomposes, up to a remainder vanishing in $\D^{1,p}(\rd)$, into a solution of the limit problem on $\Omega$ plus a finite sum of rescaled Aubin-Talenti-type $p$-bubbles solving the critical $p$-Laplace equation on $\rd$, together with the corresponding quantization of the energy level; a key technical ingredient is a Pohozaev-type identity adapted to the nonlinear operator.

In contrast to the semilinear case, a full Struwe-type decomposition for the $p$-Laplacian on unbounded domains, or on $\rd$ itself with a lower-order perturbation, remains comparatively less developed. Existence results in this direction typically combine the Mercuri-Willem decomposition \cite{Mercuri-Willem} on the bounded part of the domain with concentration-compactness arguments to rule out vanishing and bubbling at infinity; see e.g. Gon\c{c}alves and Alves \cite{Goncalves-Alves} for existence results for the $m$-Laplacian on $\rd$. Without a sign condition on the (PS) sequence, even the bounded-domain decomposition of \cite{Mercuri-Willem} is not known to hold in full generality, and the corresponding classification on unbounded domains for $p\ne2$ remains, to a large extent, open.

For the fractional Laplacian $(-\Delta)^s$, $s\in(0,1)$, the compactness properties of the critical embedding $H^s(\rd)\hookrightarrow L^{2^*_s}(\rd)$, $2^*_s=\frac{2d}{d-2s}$, were first analyzed via profile decomposition by Palatucci and Pisante \cite{Palatucci-Pisante-CVPDE}, extending the classical concentration-compactness dichotomy of Lions to the fractional Sobolev setting, and subsequently refined into a global compactness statement for (PS) sequences of the associated fractional Br\'ezis-Nirenberg functional on bounded domains by Palatucci and Pisante \cite{Palatucci-Pisante-NA}. For $f\in(\D^{s,2})'$ and a bounded potential $0<a\in L^\infty(\rd)$ with $a(x)\to1$ as $\abs{x}\to\infty$, Bhakta and Pucci \cite[Proposition~2.1]{Bhakta-Pucci} classified the (PS) sequences of the corresponding functional on the whole space $\rd$,
\begin{equation}
   I_{a,f}(u) \coloneqq \frac{1}{2} [u]_{s,2}^2 - \frac{1}{2^*_s} \int_{\rd} a(x) \abs{u}^{2^*_s} - {}_{(\D^{s,2})'}\langle f,u\rangle_{\D^{s,2}},\;  u \in \D^{s,2},
\end{equation}
showing, motivated by the earlier work of Palatucci and Pisante \cite{Palatucci-Pisante-CVPDE, Palatucci-Pisante-NA}, that (PS) sequences decompose into a solution of the perturbed problem plus a finite sum of rescaled nonlocal Aubin-Talenti bubbles concentrating either at finite points $x^i\in\rd$ or at infinity, with associated bubble-interaction condition
\begin{align}\label{bub-1}
\left| \log \left( \frac{r_n^i}{r_n^j} \right) \right| + \left| \frac{x_n^i - x_n^j}{r_n^i}  \right| \longrightarrow \infty, \quad i\ne j,
\end{align}
and matching energy quantization. A general concentration-compactness principle for the fractional Sobolev embedding directly on unbounded domains, allowing systematically for bubbling at infinity, was subsequently established by Bonder, Saintier and Silva \cite{Bonder-Saintier-Silva}, who applied it to a generalized fractional Br\'ezis-Nirenberg problem with variable coefficients on unbounded domains. Subsequently, Bhakta, Chakraborty, Miyagaki, and Pucci in \cite{MoSoMiPa} examined the global compactness results of fractional Laplace systems for $p=2$, and it was later extended for $p \in (1, \infty)$, by Biswas and Chakraborty in \cite{NS2025}.

The interplay between the critical Sobolev term and a Hardy-type singular potential $\abs{x}^{-2}$ (respectively $\abs{x}^{-2s}$ in the nonlocal setting) introduces an additional, independent source of noncompactness. In the local setting, this phenomenon was first studied by Jannelli \cite{Jannelli}, and multiplicity and existence results for the corresponding Hardy-Sobolev critical problem were obtained by Ghoussoub and Yuan \cite{Ghoussoub-Yuan}, and Cao and Peng \cite{Cao-Peng-Hardy}, the latter establishing a global compactness result for the associated singular functional. Smets \cite{Smets-TAMS} studied the whole-space Schr\"{o}dinger-type equation
\begin{align}\label{Hardy-Sobolev-II}
-\Delta u - \mu\frac{u}{\abs{x}^2} = K(x) \abs{u}^{2^*-2}u \ \text{ in } \rd,\quad u \in \mathcal{D}^{1,2}(\rd),
\end{align}
$\mu>0$, $K\in L^\infty(\rd)$, and showed that in the presence of the Hardy potential noncompactness arises through \emph{two distinct} bubble profiles: the usual Aubin-Talenti bubble, and a second, genuinely different profile solving the limit Hardy--Sobolev equation ($K\equiv1$); see also Bhakta and Sandeep \cite{Bhak-San} for the analogous phenomenon for Hardy-Sobolev-Maz'ya-type equations with cylindrical singularities. For $p=2$, $s\in(0,1)$ and $\al\in(0,2s)$, Bhakta, Chakraborty and Pucci \cite{BCP} obtained a complete classification of (PS) sequences for the critical Hardy-Sobolev functional associated with the fractional Hardy--Sobolev exponent $2^*_s(\al)=\frac{2(d-\al)}{d-2s}$, extending the Palatucci-Pisante and Bhakta-Pucci decompositions \cite{Palatucci-Pisante-NA, Bhakta-Pucci} to the singular setting for $\al\in(0,2s)$. The remaining endpoint case $\al=0$, corresponding to the pure fractional Hardy potential $\abs{x}^{-2s}$ studied at the level of existence theory in \cite{Smets-TAMS} for $p=2$, exhibits the same qualitative dichotomy between an Aubin-Talenti-type bubble and a distinct Hardy-Sobolev bubble described above. An analogous global compactness decomposition exhibiting \emph{both} the pure Sobolev and the Hardy-Sobolev profiles at $\al=0$ is, to the best of our knowledge, not known even for the corresponding local $p$-Laplace operator (with $p\neq2$) on any unbounded domain.

In \cite[Theorem 1.1]{Brasco-2018}, Brasco, Squassina, and Yang examined the global compactness property of Palais-Smale sequences associated with the following nonlocal functional:
\begin{align*}
    I_{p,s}(u) \coloneqq \frac{1}{p} [u]_{s,p}^p + \frac{1}{p} \int_{\Omega} g \abs{u}^p - \frac{\mu}{p^*_s} \int_{\Omega} \abs{u}^{p^*_s}, \; u \in \mathcal{D}_0^{s,p}(\Omega),
\end{align*}
where $\Omega$ is a smooth bounded domain in $\rd$ with $d>sp$, $g \in L^{\frac{d}{sp}}(\Omega)$, and $\mathcal{D}_0^{s,p}(\Omega) = \{u \in \wps : u=0 \text{ in } \rd \setminus \Omega \}$. This result was recently extended by Biswas in \cite{biswas2026}, who studied the global compactness property of Palais-Smale sequences associated with the following nonlocal functional:
\begin{align*}
    I_{p,s, \alpha}(u) := \frac{1}{p} [u]_{s,p}^p - \frac{\mu}{p} \int_{\Omega} \frac{\abs{u}^p}{\abs{x}^{sp}} \dx + \frac{1}{p} \int_{\Omega} g\abs{u}^p \dx
    - \frac{1}{p^*_s(\al)} \int_{\Omega} \frac{|u|^{p^*_s(\al)}}{|x|^{\al}} \dx, \; \forall \, u \in \mathcal{D}_0^{s,p}(\Omega),
\end{align*}
where $\Omega$ is a smooth bounded domain in $\rd$ with $d>sp$,  $\alpha \ge 0$, $p^*_s(\al):=\frac{p(d-\al)}{d-sp}$, and $g \in L^{\frac{d-\alpha}{sp-\alpha}}(\Omega)$. This global compactness decomposition exhibits \emph{both} the pure Sobolev and the Hardy-Sobolev profiles at $\al=0$.

\subsection{Difficulties and Novelties} 
The whole-space problem considered here has a different feature. For $p\neq2$, the quadratic structure available in the linear case is lost, and the interaction of different profiles has to be treated directly at the level of the fractional $p$-energy. Moreover, at the endpoint $\alpha=0$, the critical nonlinearity is also translation invariant, whereas the Hardy potential still singles out the origin. This leads to two different concentration regimes. By keeping track of the concentration centre $y_n$ and scale $r_n$, and in particular of the relative parameter $\frac{|y_n|}{r_n}$, which plays the pivotal role, we identify both Hardy profiles and pure Sobolev profiles and obtain their energy decomposition and asymptotic separation. The resulting global compactness theorem is then combined with a hidden-convexity path and a min-max argument to obtain a second positive solution of the nonhomogeneous problem.

In the following, we outline the strategies associated with the fully non-linear equation \eqref{MainEq}.

\noi \textbf{Dealing with the endpoint case:} The endpoint case $\alpha=0$ requires a separate argument. When $\alpha>0$, the weight $|x|^{-\alpha}$ in the critical term fixes the origin, and concentration can occur only through dilations. When $\alpha=0$, however, the critical nonlinearity becomes translation invariant, while the Hardy potential remains centred at the origin. Thus, a concentrating sequence may move away from the singularity, and both its centre and its scale have to be monitored. In linear, conformally invariant problems, one can study concentration at infinity using a Kelvin transform. However, this tool is not available in our setting because the nonlinear fractional $p$-energy is defined for $p \neq 2$. We therefore avoid any inversion argument and instead use a direct approach based on the moving concentration function 
$Q_n(r):=\sup_{y\in\mathbb R^d}\int_{B(y,r)} |u_n|^{p_s^*}\,dx$. We choose the centre $y_n$ and a scalar $r_n$ where a fixed amount of mass is concentrated, and rescale around $(y_n, r_n)$. We then study the ratio $\sigma_n:=\frac{y_n}{r_n}.$ This quantity describes the concentration behaviour. If $\sigma_n$ remains bounded, the rescaled sequence interacts with the Hardy singularity, and the limit profile solves the Hardy-Sobolev equation. If $|\sigma_n| \to \infty$, the concentration drifts away from the origin, the Hardy term vanishes in the limit, and the profile solves the standard critical Sobolev equation.  This direct centre-scale analysis allows us to treat both regimes within the same extraction procedure, without using the Kelvin transform. We further prove the corresponding energy decomposition and asymptotic separation of the extracted profiles. In particular, the $\alpha=0$ decomposition contains both Hardy bubbles and free Sobolev bubbles, a phenomenon absent in the weighted case $\alpha>0$. The two concentration regimes described above are illustrated in the figure above. Choose a centre $y_n$ and a scalar $r_n$ such that a fixed amount of mass is concentrated.

\medskip
\begin{center}
\begin{tikzpicture}[>=Stealth, thick, color=blue, every node/.style={text=black}]
% ==========================================
% CASE I: Bounded Case
% ==========================================
\begin{scope}[shift={(0,0)}]
    % 1. Section Title (Centered perfectly above)
    \node[align=center, font=\large\bfseries] at (2.5, 0.5) {(I) Bounded: $\frac{|y_n|}{r_n} = o_n(1)$};

    % 2. Top Picture
    \draw[->] (0, -2.5) -- (5, -2.5) node[right] {$x_1$};
    \draw[->] (1, -2.5) node[below left] {$0$} -- (1, -0.5) node[right] {$x_2$};
    
    \coordinate (Y1) at (3, -1.5);
    \filldraw[fill=cyan!20] (Y1) circle (0.8cm);
    \fill (Y1) circle (1.5pt) node[below] {$y_n$};
    \draw[->] (Y1) -- ++(30:0.8cm) node[midway, above left=-2pt, scale=0.8] {$r_n$};
    \node[above] at (3, -0.6) {$B(y_n, r_n)$};

    % Flow Arrow Down
    \draw[->] (2.5, -3.2) -- (2.5, -3.8);

    % 3. Rescale Text (Centered perfectly between pictures)
    \node[align=center] at (2.5, -4.6) {
        Rescale around $(y_n, r_n)$ \\[1.5ex]
        $\widetilde{u}_n(z) = r_n^{\frac{d-sp}{p}} u_n(r_n z + y_n)$
    };

    % Flow Arrow Down
    \draw[->] (2.5, -5.4) -- (2.5, -6.0);

    % 4. Bottom Picture (Zoomed-in local view)
    \draw (0, -9.5) rectangle (5, -6);
    
    % Draw the filled circle FIRST so it sits in the background
    \filldraw[fill=cyan!20] (2.5, -7.75) circle (0.8cm); 

    % Draw axes and center point ON TOP of the circle
    \draw[->] (0.2, -7.75) -- (4.5, -7.75) node[right] {$z_1$};
    \draw[->] (2.5, -9) -- (2.5, -6.5) node[right] {$z_2$};
    \node[below right] at (2.5, -7.75) {$0$};
    
    % Sigma tick
    \draw (1.2, -7.65) -- (1.2, -7.85) node[below] {$-\sigma_n$};

    % 5. Conclusion Text (Centered perfectly below)
    \node[align=center] at (2.5, -10.6) {
        $\sigma_n = \frac{y_n}{r_n} \to \sigma \in \mathbb{R}^d$ \\[1.5ex] 
        \textbf{(Hardy profile)}
    };
\end{scope}

% ==========================================
% CASE II: Unbounded Case
% ==========================================
% Shifted horizontally by 8.5cm to sit nicely alongside Case I
\begin{scope}[shift={(8.5,0)}] 
    % 1. Section Title (Centered perfectly above)
    \node[align=center, font=\large\bfseries] at (2.5, 0.5) {(II) Unbounded: $\frac{|y_n|}{r_n} \to \infty$};

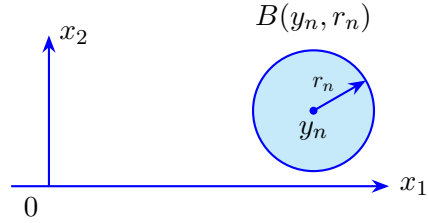
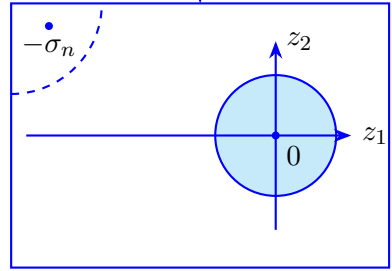
    \draw[->] (0, -2.5) -- (5, -2.5) node[right] {$x_1$};
    \draw[->] (0.5, -2.5) node[below left] {$0$} -- (0.5, -0.5) node[right] {$x_2$};
    
    \coordinate (Y2) at (4.0, -1.5); 
    \filldraw[fill=cyan!20] (Y2) circle (0.8cm);
    \fill (Y2) circle (1.5pt) node[below] {$y_n$};
    \draw[->] (Y2) -- ++(30:0.8cm) node[midway, above left=-2pt, scale=0.8] {$r_n$};
    \node[above] at (4.0, -0.6) {$B(y_n, r_n)$};
    
    \draw[->] (2.5, -3.2) -- (2.5, -3.8);
    \node[align=center] at (2.5, -4.6) {
        Rescale around $(y_n, r_n)$ \\[1.5ex]
        $\widetilde{u}_n(z) = r_n^{\frac{d-sp}{p}} u_n(r_n z + y_n)$
    };
    
    \draw[->] (2.5, -5.4) -- (2.5, -6.0);
    \draw (0, -9.5) rectangle (5, -6);
    \filldraw[fill=cyan!20] (3.5, -7.75) circle (0.8cm); 
    \draw[->] (0.2, -7.75) -- (4.5, -7.75) node[right] {$z_1$};
    \draw[->] (3.5, -9) -- (3.5, -6.5) node[right] {$z_2$};
    \fill (3.5, -7.75) circle (1.5pt) node[below right] {$0$}; 
    
    \draw[dashed, thick] (0, -7.2) arc (-90:0:1.2cm);
    \fill (0.5, -6.3) circle (1.5pt) node[below, text=black, inner sep=3pt] {$-\sigma_n$};
    
    \node[align=center] at (2.5, -10.6) {
        $\sigma_n = \frac{y_n}{r_n} \text{ and } |\sigma_n| \to \infty$ \\[1.5ex] 
        $|z + \sigma_n|^{-sp} \to 0$ locally \textbf{(pure Sobolev profile)}
    };
\end{scope}
\end{tikzpicture}

\noi In the above figure, after rescaling around the concentration ball $B(y_n, r_n)$, the ball becomes $B(0,1)$, and the Hardy singularity is located at $-\sigma_n=-\frac{y_n}{r_n}$. If $\sigma_n$ remains bounded, the singularity remains visible, and a Hardy profile occurs. If $|\sigma_n| \to \infty$, the singularity escapes to infinity and a pure Sobolev profile occurs.
\end{center}

\medskip
\noi \textbf{A nonlinear path for the multiplicity result:} Let $u_{\al}$ be a positive solution of negative energy, given by Theorem \ref{existence} when $\al>0$ and by Theorem \ref{existence-II} when $\al=0$. A natural first choice is the additive path $t \mapsto u_{\al}+tV_{\al}$, where $V_{\al} \in \wps$ is an extrema for $S_{\mu, \alpha}$. However, this path is not suitable for $p \neq 2$. Since $I_{\mu,\al,f}'(u_{\al})=0$, we have
\begin{align*}
     I_{\mu,\al,f}\left( u_{\al}+tV_{\al}\right)=I_{\mu,\al,f}(u_{\al})+I_{\mu,\al,0}\left( tV_{\al}\right)+R(t),
\end{align*}
where $R(t)$ consists of the remainder terms. By \eqref{eq:V-max}, the middle term takes at most $\frac{sp-\al}{p(d-\al)}\,S_{\mu}^{\frac{d-\al}{sp-\al}}$, so for the existence of a second solution, $R(t)<0$ is required (in view of Proposition \ref{prop-compact}). 

However, $R(t)$ contains the term
\begin{equation}\label{eq:additive-remainder}
    \frac1p \iint_{\rd \times \rd}\frac{\abs{Du_{\al}+tDV_{\al}}^p-\abs{Du_{\al}}^p-t^p\abs{DV_{\al}}^p-pt\abs{Du_{\al}}^{p-2}Du_{\al}\,DV_{\al}}{\abs{x-y}^{d+sp}}\dxy,
\end{equation}
where we write $Dv(x,y):=v(x)-v(y)$. For every $p \ne 2$, the integrand of \eqref{eq:additive-remainder} has no fixed sign. Writing $F_p(a,b):=\abs{a+b}^p-\abs{a}^p-\abs{b}^p-p\abs{a}^{p-2}ab$, we get
\begin{equation}\label{eq:Fp}
      F_p(1,1)=2^p-p-2, \quad F_p(1,-2)=2p-2^p.
\end{equation}
Both vanish at $p=2$. The function $p \mapsto 2^p-p-2$ is convex and vanishes at $p=2$, with positive derivative, so it is negative on $(1,2)$ and positive on $(2,\infty)$, and $p \mapsto 2p-2^p$ is concave and vanishes at $p=1$ and at $p=2$, so it is positive on $(1,2)$ and negative on $(2,\infty)$. So for $p>2$, the first value in \eqref{eq:Fp} is positive, the second is negative.
So, \eqref{eq:additive-remainder} has to be handled, and it has to be compared with the negative part of $R(t)$, which is delicate. 

We choose a different path $$\gamma(t):=\left( u_{\al}^p+t^pV_{\al}^p\right)^{\frac1p}.$$ At each point $x$ of $\rd$, $\gamma(t)(x)$ is the $\ell^p$ norm of the pair $\left( u_{\al}(x),tV_{\al}(x)\right)$ in $\R^2$. We note also that $\gamma(0)=u_{\al}$ and $\gamma(t)\ge u_{\al}>0$ for every $t \ge 0$. Nothing here uses the value of $\al$.
\medskip

\noi \textbf{Organization:} The organization of this article is outlined as follows. In Section \ref{se-2}, we present technical lemmas that will be applied throughout the study. Following this, Sections \ref{se-3} and \ref{se-4}  are dedicated to analyze the Palais-Smale decompositions of the energy functional $I_{\mu,\al,f}$, considering the cases $\al>0$ and $\al=0$, respectively. Section \ref{se-5} provides the proof of existence for the first positive weak solution to \eqref{MainEq}. Finally, in Section \ref{se-6}, we establish the multiplicity of positive weak solutions to \eqref{MainEq}.

\section{Preliminary}\label{se-2}
We collect the notation and compactness tools that will be used in both global compactness arguments. In particular, the Br\'ezis-Lieb type decompositions below allow us to separate the energy and the Euler-Lagrange equation along weakly convergent
sequences.
\subsection{Notation} We use the following notations and conventions for the paper. 
\begin{enumerate}[(a)] 
    \item For $u,v \in \wps$, set
\begin{align*}
    &\mathcal{A}(u,v) \coloneqq \iint_{\rd \times \rd} \frac{\abs{u(x) - u(y)}^{p-2} (u(x) - u(y)) (v(x) - v(y))}{\abs{x-y}^{d+sp}} \dxy, \\
    &\mathcal{A}_1(u,v) \coloneqq \iint_{\rd \times \rd} \frac{\abs{u(x) - u(y)}^{p-1} \abs{v(x) - v(y)}}{\abs{x-y}^{d+sp}} \dxy.
\end{align*}
   \item For $\la>0$ and $y \in \rd$, we denote
   \begin{align*}
       C_{\la} u(x) \coloneqq \la^{-\frac{d-sp}{p}} u\left(\frac{x}{\la} \right), \text{ and } C_{y, \la}u(x) \coloneqq \la^{-\frac{d-sp}{p}} u \left(\frac{x-y}{\la}\right).
   \end{align*}
   \item For brevity, we write $S_{\mu,\al}=S$ (when $\mu=0$ and $\al=0$), and $S_{\mu,\al} = S_{\mu}$ (when $\al=0$).
    \item Set
  \begin{align*}
    p^*_s(0):=p^*_s=\frac{dp}{d-sp}, \; \vartheta:=\frac{d-sp}{p}, \; \Theta_{\mu,\al}:=\frac{sp-\al}{p(d-\al)}\,S_{\mu}^{\frac{d-\al}{sp-\al}}.
   \end{align*} 
   \item Let $J_q(t)=\abs{t}^{q-2}t$, where $t \in \R$ and $q>1$.
   \item We define $$\norm{u}_{\al}:=\left( \int_{\rd}\frac{\abs{u}^{p^*_s(\al)}}{\abs{x}^{\al}}\dx\right)^{\frac{1}{p^*_s(\al)}}, \; \forall \, u \in \wps.$$
   \item In \eqref{best-constant}, we denote $\overline{S}:= S_{0,0}$ and $\overline{S}_{\mu} := S_{\mu,0}$.
   \item We write $Du(x,y):=u(x)-u(y)$, for $x,y \in \rd$. 
  \item Throughout the paper, $C=C(a,b,c,\cdots)$ denotes a generic positive constant that varies from line to line.
\end{enumerate}

\subsection{Technical lemmas}
We discuss the classical Brézis-Lieb Lemma and its consequences. 

\begin{lemma}\label{BL}
    Let $1<q< \infty$. Let $\{ f_n \} \subset L^{q}(\rd)$ be a bounded sequence such that $f_n(x) \ra f(x)$ a.e. $x \in \rd$. Then the following hold: 
    \begin{enumerate}
        \item[\rm{(i)}] $\norm{f_n}_{L^q(\rd)}^q - \norm{f_n - f}_{L^q(\rd)}^q + o_n(1) = \norm{f}_{L^q(\rd)}^q.$
        \item[\rm{(ii)}] Consider the function $J_q$ defined as $J_q(t)=\abs{t}^{q-2}t$. Then 
        \begin{align*}
           J_q(f_n) - J_q(f_n-f) = J_q(f) + o_n(1) \text{ in } L^{q'}(\rd).
        \end{align*}
    \end{enumerate}
\end{lemma}
\begin{proof}
Proof of (i) follows from \cite{Br-Li}, and proof of (ii) follows from \cite[Lemma 3.2]{Mercuri-Willem}. 
\end{proof}
The above lemma leads to the following convergence.  

\begin{lemma}\label{convergence-BL}
    Let $\{u_n\}$ weakly converge to $u$ in $\wps$ with $u_n(x) \ra u(x)$ a.e. $x \in \rd$. Then, up to a subsequence, the following hold
   \begin{enumerate}
       \item[\rm{(i)}] $[u_n]_{s,p}^p - [u_n - u]_{s,p}^p = [u]^p_{s,p} + o_n(1)$. 
       \item[\rm{(ii)}] For $0 \le \tilde{\al} \le sp$, $\displaystyle \int_{\rd} \frac{\abs{u_n}^{p^*_s(\tilde{\al})}}{\abs{x}^{\tilde{\al}}} \dx -  \int_{\rd} \frac{\abs{u_n-u}^{p^*_s(\tilde{\al})}}{\abs{x}^{\tilde{\al}}} \dx= \int_{\rd} \frac{\abs{u}^{p^*_s(\tilde{\al})}}{\abs{x}^{\tilde{\al}}} \dx + o_n(1)$.
       \item[\rm{(iii)}] Consider the function $J_p$ defined as $J_p(t)=\abs{t}^{p-2}t$. Then 
       \begin{align*}
           \frac{J_p(u_n(x) - u_n(y))}{\abs{x-y}^{\frac{d+sp}{p'}}} - \frac{J_p\left((u_n(x)-u(x)) - (u_n(y)-u(y))\right)}{\abs{x-y}^{\frac{d+sp}{p'}}} = \frac{J_p(u(x) - u(y))}{\abs{x-y}^{\frac{d+sp}{p'}}} + o_n(1),
       \end{align*}
       in $L^{p'}(\mathbb{R}^{2d})$.
       \item[\rm{(iv)}] For $0 \le \tilde{\al} \le sp$, consider the function $J_{p^*_s(\al)}$ defined as $J_{p^*_s(\tilde{\al})}(t)=\abs{t}^{p^*_s(\tilde{\al})-2}t$. Then 
       \begin{align*}
           \frac{J_{p^*_s(\tilde{\al})}(u_n(x))}{\abs{x}^{\frac{\tilde{\al}}{(p^*_s(\tilde{\al}))'}}} - \frac{J_{p^*_s(\tilde{\al})}(u_n(x)-u(x))}{\abs{x}^{\frac{\tilde{\al}}{(p^*_s(\tilde{\al}))'}}} = \frac{J_{p^*_s(\tilde{\al})}(u(x))}{\abs{x}^{\frac{\tilde{\al}}{(p^*_s(\tilde{\al}))'}}} + o_n(1),
       \end{align*}
       in $L^{(p^*_s(\tilde{\al}))'}(\rd)$.
   \end{enumerate}
    \end{lemma}
The following lemma states the convergence of some integrals. For proof, we refer to \cite[Lemma 2.5]{NS2025}.

\begin{lemma}\label{convergence-integrals}
Let $\{ u_n \}$ weakly converge to $u$ in $\wps$. 
    \begin{enumerate}
        \item[\rm{(i)}] Let $0\le \tilde{\al} \le sp $. Then up to a subsequence
    \begin{align*}
        & \lim_{n \ra \infty} \int_{\rd} \frac{\abs{u_n(x)}^{p^*_s(\tilde{\al}) -2} u_n(x)}{\abs{x}^{\tilde{\al}}} \phi(x) \dx = \int_{\rd} \frac{\abs{u(x)}^{p^*_s(\tilde{\al}) -2} u(x)}{\abs{x}^{\tilde{\al}}} \phi(x) \dx, \\ & \lim_{n \ra \infty} \int_{\rd} \frac{\abs{\phi(x)}^{p^*_s(\tilde{\al}) -2} \phi(x)}{\abs{x}^{\tilde{\al}}} u_n(x) \dx = \int_{\rd} \frac{\abs{\phi(x)}^{p^*_s(\tilde{\al}) -2} \phi(x)}{\abs{x}^{\tilde{\al}}} u(x) \dx,
    \end{align*}
    for every $\phi \in \wps$.
    \item[\rm{(ii)}] Then up to a subsequence
    \begin{align*}
    &\lim_{n \ra \infty} \iint_{\rd\times\rd}\frac{|u_n(x)-u_n(y)|^{p-2}(u_n(x)-u_n(y))(\phi(x)-\phi(y))}{|x-y|^{d+sp}} \dxy  \\
    &=\iint_{\rd\times\rd}\frac{|u(x)-u(y)|^{p-2}(u(x)-u(y))(\phi(x)-\phi(y))}{|x-y|^{d+sp}} \dxy,
    \end{align*}
    for every $\phi \in \wps$.
    \item[\rm{(iii)}] Then up to a subsequence
    \begin{align*}
    &\lim_{n \ra \infty} \iint_{\rd\times\rd}\frac{|\phi(x)-\phi(y)|^{p-2}(\phi(x)-\phi(y))(u_n(x)-u_n(y))}{|x-y|^{d+sp}} \dxy  \\
    &=\iint_{\rd\times\rd}\frac{|\phi(x)-\phi(y)|^{p-2}(\phi(x)-\phi(y))(u(x)-u(y))}{|x-y|^{d+sp}} \dxy,
    \end{align*}
    for every $\phi \in \wps$.
    \end{enumerate}
\end{lemma}

We state the following lemma due to Bahri in \cite[Lemma~1.2 of Technical Lemmas]{Bahri}.

\begin{lemma}\label{Bahri}
Let $q>1$. There exists $C=C(q)$ such that
\begin{align}
    \Bigg{|}\bigg{|}\sum_{j=0}^{k}a_j\bigg{|}^{q-1}\sum_{j=0}^{k}a_j - \sum_{j=0}^{k}\abs{a_j}^{q-1}{a_j}\Bigg{|} \le C(q) \sum_{0\leq i\neq j\leq k}\abs{a_j}^{q-1}\abs{a_i},  \quad \forall \, a_0,\dots,a_{k}\in \R.
\end{align}
% \begin{equation}
%       \left| \; \bigg| \sum_{j=0}^{\ell}a_j \bigg|^q-\sum_{j=0}^{\ell}\abs{a_j}^q \;\right| \le C \sum_{0 \le i \ne j \le \ell}\abs{a_j}^{q-1}\abs{a_i}, \quad \forall\, a_0,\dots,a_{\ell}\in \R.
% \end{equation}
\end{lemma}

\section{Global compactness result: The case $\alpha>0$}\label{se-3}
%This section studies the Palais-Smale decomposition of $I_{\mu,\al,f}$.
We now prove the global compactness theorem in the weighted case $0<\alpha<sp$. Since both the Hardy potential and the critical Hardy-Sobolev terms are centred at the origin, translations do not generate additional profiles; the relevant loss of compactness is
therefore described by dilations.

\subsection{Behavior of bounded sequences under dilations} 
We first record the behaviour of the fractional $p$-energy under dilations and characterize when two dilation parameters become asymptotically separated. Let $\mathcal{D}\subset \mathcal{U}(\wps)$ be a class of isometric operators induced by dilations on $\rd$, where
\begin{align*}
  &\mathcal{U}(\wps) \coloneqq \left\{\phi:\wps\to\wps : [\phi(u)]_{s,p}=[u]_{s,p}\right\},\\
  &\mathcal{D} \coloneqq \left\{C_{\lambda}\in \mathcal{U}(\wps) : C_{\lambda}u(x)=\lambda^{-\frac{d-sp}{p}}u \left( \frac{x}{\lambda} \right), \; \forall \, u\in\wps; \lambda \in (0, \infty) \right\}.\
\end{align*}
Observe that, for every $u, v \in \wps$, 
\begin{align}\label{A-behaviour}
\mathcal{A} \left( C_{\delta}u, C_{\lambda} v \right) = \mathcal{A}\left( u, C_{\frac{\lambda}{\delta}} v \right) = \mathcal{A}(C_{\frac{\delta}{\la}}u,v), \text{ where } \delta, \lambda \in (0, \infty). 
\end{align}

\begin{proposition}\label{Cylambda}
    If $\lambda_n\to \lambda$ in $(0,\infty)$, then for all $u\in \wps$, $C_{\lambda_n}u \rightarrow C_{\lambda}u$ in the norm topology of $\wps$.
\end{proposition}

\begin{proof} 
For $u\in \cc(\rd)$,  
\begin{align*}
   C_{\lambda_n}u(x) = \lambda_n^{-\frac{d-sp}{p}}u \left(\frac{x}{\lambda_n}\right) \ra \lambda^{-\frac{d-sp}{p}}u \left(\frac{x}{\lambda}\right) = C_{\la} u(x), 
\end{align*}
for every $x \in \rd$ as $n\to\infty$. Further for each $n\in\mathbb{N}$, $  [C_{\lambda_n}u ]_{s,p} = [u]_{s,p} = [C_{\lambda}u ]_{s,p}.$
Therefore, applying Lemma \ref{convergence-BL}-(i), for all $u\in \cc(\rd)$ we get $[C_{\lambda_n}u - C_{\lambda}u]_{s,p}\to 0,\text{ as } n \to \infty.$
Next, let $u\in \wps$ and $\varepsilon>0$ be given. By the density of $\cc(\rd)$ in $\wps$, there exists $v\in \cc(\rd)$ such that 
\begin{align*}
 [u-v]_{s,p} < \frac{\varepsilon}{3}.
\end{align*}
Hence there exists $k \in \N$ such that for all $n \ge k$,
\begin{align*}
[C_{\lambda_n}u - C_{\lambda}u]_{s,p} &\leq [C_{\lambda_n}u-C_{\lambda_n}v]_{s,p} + [C_{\lambda_n}v-C_{\lambda}v]_{s,p} + [C_{\lambda}u - C_{\lambda}v]_{s,p}\no\\
&= 2[u-v]_{s,p} + [C_{\lambda_n}v-C_{\lambda}v]_{s,p} < \varepsilon.
\end{align*}
Thus, $ C_{\lambda_n}u \rightarrow C_{\lambda}u$ for all $u\in \wps$.
\end{proof}

The following proposition measures the non-compactness of $\wps \hookrightarrow L^{p^*_s(\al)}(\rd, \abs{x}^{-\al})$ (with $\al>0$) under the conformal group action of dilation.  

\begin{proposition}\label{weak-bub-I}
For any sequence $\{\delta_n\},\,\{\lambda_n\}\subset (0, \infty)$, we have
\begin{align*}
\left| \log\left(\frac{\delta_n}{\lambda_n}\right) \right| \to\infty \Longleftrightarrow \mathcal{A}(C_{\delta_n}u,C_{\lambda_n}v)\to 0, \text{ as } n \ra \infty,
\end{align*}
for all $u,\,v\in\wps$. 
\end{proposition}
\begin{proof}
In view of \eqref{A-behaviour}, it is enough to show the following: 

For any sequence $\{\lambda_n\} \subset (0, \infty)$, 
\begin{align}\label{limit-2}
    \left| \log(\lambda_n) \right| \to \infty  \Longleftrightarrow \mathcal{A}(C_{\lambda_n}u, v)\to 0 \text{ and } \mathcal{A}(u,C_{\lambda_n}v)\to 0\text{ as }n\to\infty,
\end{align}
for all $u, v\in \wps$.

We first show that for $u \in \wps \setminus \{0\}$, if $\mathcal{A}(C_{\lambda_n}u, v)\to 0$, then $\left| \log(\lambda_n) \right| \to \infty$. On the contrary, suppose $\la_n \ra \la>0$ as $n \ra \infty$.  Using Proposition \ref{Cylambda}, we have $C_{\lambda_n}u \ra C_{\lambda}u$ as $n \ra \infty$. By hypothesis, 
\begin{equation}\label{weak-conv-1}
\lim_{n \ra \infty} \mathcal{A}(C_{\lambda_n}u,w) = 0, \; \forall \,w\in\wps.
\end{equation}
In particular, for $w=C_{\lambda}u\in \wps$, $\mathcal{A}(C_{\lambda_n}u, C_{\lambda}u) \to 0,\text{ as }n\to\infty.$
Since $\mathcal{A} \in \mathcal{C}^1(\mathcal{W})$, we get
\begin{align*}
[u]^p_{s,p} = [C_{\lambda}u]_{s,p}^p = \mathcal{A}(C_{\lambda} u, C_{\lambda} u) = 0,
\end{align*}
which contradicts the fact that $u \neq 0$. Therefore, $\left| \log(\lambda_n) \right| \to \infty$ as $n \ra \infty$. 

From \eqref{A-behaviour}, $\mathcal{A}(C_{\lambda_n}u,v) = \mathcal{A}(u,C_{\lambda_n^{-1}}v), \; \forall \, (u,v) \in \mathcal{W}.$  
Set $\delta_n = \frac{1}{\la_n}$, and observe that 
\begin{align*}
    \abs{\log(\delta_n)} \ra \infty \Longleftrightarrow \abs{\log(\la_n)} \ra \infty.
\end{align*}
In view of this observation, for the converse part, it is enough to show that $\abs{\log(\la_n)} \ra \infty$ implies $\mathcal{A}(u,C_{\lambda_n}v)\to 0$ for all $u, v \in \wps$. 
Due to the density of $\cc(\rd)$ in $\wps$, it is enough to show $\mathcal{A}(u,C_{y_n,\lambda_n}v)\to 0$ for all $u, v \in \cc(\rd)$. If $\la_n \ra \infty$, then observe that $C_{\lambda_n}v \ra 0$ a.e. in $\rd$. Moreover, since $\{ C_{\lambda_n}v\}$ is bounded in $\wps$, from the reflexivity, up to a subsequence, $C_{\lambda_n}v \rightharpoonup w$ in $\wps$. The uniqueness of the limit yields $w=0$ a.e. in $\rd$.  Hence, using Lemma \ref{convergence-integrals}-(iii), we infer that $\mathcal{A}(u,C_{\lambda_n}v) \ra 0$ for all $u\in\wps$ as $n \ra \infty$. Now we consider the case where $\la_n \ra 0$. In this case, using a change of variables, we write 
\begin{align*}
    \mathcal{A}(u, C_{\lambda_n}v) &= \la_n^{-\frac{d-sp}{p}} \iint_{\rd \times \rd} \frac{\abs{u(x) - u(y)}^{p-2}(u(x) - u(y)) \left( v(\frac{x}{\la_n}) - v(\frac{y}{\la_n}) \right)}{\abs{x-y}^{d+sp}} \dxy \\
    &= \la_n^{\frac{d-sp}{p'}} \iint_{\rd \times \rd} \frac{\abs{u(\overline{x} \la_n) - u(\overline{y} \la_n)}^{p-2}(u(\overline{x} \la_n) - u(\overline{y} \la_n)) \left( v(\overline{x}) - v(\overline{y}) \right)}{\abs{\overline{x}-\overline{y}}^{d+sp}} \d \overline{x} \d \overline{y}.
\end{align*}
Since $u, v \in \cc(\rd)$, there exists $R>0$ such that $\text{supp}(u),\, \text{supp}(v) \subset B(0, R)$. So the above integral vanishes over $B(0,R)^c \times B(0,R)^c$. Applying H\"{o}lder's inequality with the pair $(p,p')$, we get 
\begin{align}\label{conv-1}
    \left| \mathcal{A}(u, C_{\lambda_n}v) \right| & \le 2 \left( \la_n^{d-sp} \iint_{\rd \times B(0,R)} \frac{\abs{u(\overline{x} \la_n) - u(\overline{y} \la_n)}^{p}}{\abs{\overline{x}-\overline{y}}^{d+sp}} \d \overline{x} \d \overline{y} \right)^{\frac{1}{p'}} \no \\
    &\quad \left(  \iint_{\rd \times B(0,R)} \frac{\abs{v(\overline{x}) - v(\overline{y}) }^p}{\abs{\overline{x}-\overline{y}}^{d+sp}} \d \overline{x} \d \overline{y} \right)^{\frac{1}{p}}.
\end{align}
Using the change of variables, we see that 
\begin{align*}
    \la_n^{d-sp} \iint_{\rd \times B(0,R)} \frac{\abs{u(\overline{x} \la_n) - u(\overline{y} \la_n)}^{p}}{\abs{\overline{x}-\overline{y}}^{d+sp}} \d \overline{x} \d \overline{y} = \int_{B(0, \la_nR)} \left( \int_{\rd} \frac{\abs{u(x) - u(y)}^{p}}{\abs{x-y}^{d+sp}} \dx \right) \dy.
\end{align*}
Since $[u]_{s,p} < \infty$ and $|B(0, \la_nR)| \ra 0$ as $n \ra \infty$, from the absolute continuity of the Lebesgue integral,
\begin{align*}
    \lim_{n \ra \infty} \int_{B(0, \la_nR)} \left( \int_{\rd} \frac{\abs{u(x) - u(y)}^{p}}{\abs{x-y}^{d+sp}} \dx \right) \dy = 0.
\end{align*}
Therefore, using \eqref{conv-1}, we get $\mathcal{A}(u, C_{\lambda_n}v) \ra 0$ as $n \ra \infty$. This completes the proof.   
\end{proof}

Using the same set of arguments as used in Proposition \ref{weak-bub-I}, we have the following. 

\begin{proposition}\label{weak-bub-II}
For any sequence $\{\delta_n\},\,\{\lambda_n\}\subset (0, \infty)$, we have
\begin{align*}
\left| \log\left(\frac{\delta_n}{\lambda_n}\right) \right| \to\infty \Longleftrightarrow \mathcal{A}_1(C_{\delta_n}u,C_{\lambda_n}v)\to 0, \text{ as } n \ra \infty,
\end{align*}
for all $u,\,v\in\wps$. 
\end{proposition}

\subsection{Proof of Theorem \ref{PS-decomposition}} With the dilation analysis in hand, we extract the profiles of a Palais-Smale sequence and prove the energy decomposition and parameter separation stated in Theorem \ref{PS-decomposition}.
Since $\{u_n\} \subset \wps$ is a (PS) sequence of $I_{\mu,\al,f}$ at level $\eta$, 
\begin{align}\label{PSD-1}
    I_{\mu,\al,f}(u_n) - \frac{1}{p^*_s(\al)}  \prescript{}{(\wps)^*}{\langle} I'_{\mu, \al,f}(u_n),(u_n){\rangle}_{\wps} \le \eta + o_n(1) + o_n(1) [u_n]_{s,p}.
\end{align}
Now 
\begin{align*}
    & \text{ L.H.S. of \eqref{PSD-1} } \ge C_{\text{eqiv}} \left( \frac{1}{p} - \frac{1}{p^*_s(\al)} \right) [u_n]_{s,p}^p - \left( 1 - \frac{1}{p^*_s(\al)} \right) \left( \norm{f}_{(\wps)^*} [u_n]_{s,p}\right) \\
    & \ge C_{\text{eqiv}} \left( \frac{1}{p} - \frac{1}{p^*_s(\al)} \right) [u_n]_{s,p}^p -  \left( 1 - \frac{1}{p^*_s(\al)} \right) \norm{f}_{(\wps)^*} [u_n]_{s,p}.
\end{align*}
In view of R.H.S. of \eqref{PSD-1}, $\{ u_n \}$ is a bounded sequence in $\wps$. By the reflexivity of $\wps$, let $\{ u_n \}$ weakly converge to $\tilde{u}$ in $\wps$ (up to a subsequence). Since $I'_{\mu,\al, f}(u_n) \ra 0$ in $(\wps)^*$, for every $\phi \in \wps$ we have
\begin{align*}
    \mathcal{A}(u_n , \phi) & - \mu \int_{\rd} \frac{\abs{u_n(x)}^{p-2} u_n(x)}{\abs{x}^{sp}} \phi(x) \dx \no \\ 
    &=\int_{\rd} \frac{\abs{u_n(x)}^{p^*_s(\al)-2}u_n(x)}{\abs{x}^{\al}} \phi(x) \dx +\prescript{}{(\wps)^*}{\langle}f,\phi{\rangle}_{\wps}, \; \forall \, \phi \in \wps.
\end{align*}
Taking the limit as $n \ra \infty$ in the above identity and using Lemma \ref{convergence-integrals}, we see that $\tilde{u}\in \wps$ satisfies \eqref{weak1} weakly. We divide the rest of the proof into several steps. 

\noi \textbf{Step 1:} We claim that $\{ u_n - \tilde{u}\}$ is a (PS) sequence for $I_{\mu,\al,0}$ at level $\eta -I_{\mu,\al,f}(\tilde{u})$. Set $\tilde{u}_n = u_n - \tilde{u}$. Using Lemma \ref{convergence-BL} and $\tilde{u}_n \rightharpoonup 0$ in $\wps$, we get
\begin{align*}
    I_{\mu,\al,0}(\tilde{u}_n) & = \frac{1}{p} [\tilde{u}_n]_{s,p}^p - \frac{\mu}{p} \int_{\rd} \frac{\abs{\tilde{u}_n}^p}{\abs{x}^{sp}} \dx - \frac{1}{p^*_s(\al)} \int_{\rd} \frac{\abs{\tilde{u}_n}^{p^*_s(\al)}}{\abs{x}^{\al}} \dx \\
    & = \frac{1}{p} \left( [u_n]_{s,p}^p - [\tilde{u}]_{s,p}^p \right) - \frac{\mu}{p} \left( \int_{\rd} \frac{\abs{u_n}^p - \abs{\tilde{u}}^p}{\abs{x}^{sp}} \dx \right) - \frac{1}{p^*_s(\al)} \left( \int_{\rd} \frac{\abs{u_n}^{p^*_s(\al)} - \abs{\tilde{u}}^{p^*_s(\al)}}{\abs{x}^{\al}} \dx \right) \\
    & -\prescript{}{(\wps)^*}{\langle}f,u_n{\rangle}_{\wps} + \prescript{}{(\wps)^*}{\langle}f, \tilde{u}{\rangle}_{\wps} + o_n(1) \\
    & = I_{\mu,\al,f}(u_n) - I_{\mu,\al,f}(\tilde{u}) + o_n(1).
\end{align*}
Hence $I_{\mu,\al,0}(\tilde{u}_n) \ra \eta - I_{\mu,\al,f}(\tilde{u})$ as $n \ra \infty$. Further, for $\phi \in \wps$, using $\tilde{u}_n \rightharpoonup 0$ in $\wps$, and Lemma \ref{convergence-integrals}, we have 
\begin{align*}
    &\prescript{}{(\wps)^*}{\langle} I_{\mu,\al,0}'(\tilde{u}_n), \phi {\rangle}_{\wps} \\
    &= \mathcal{A}(\tilde{u}_n , \phi) - \mu \int_{\rd} \frac{\abs{ \tilde{u}_n }^{p-2} \tilde{u}_n}{\abs{x}^{sp}} \phi\,\dx - \int_{\rd} \frac{\abs{\tilde{u}_n}^{p^*_s(\al) -2} \tilde{u}_n}{\abs{x}^{\al}} \phi \,\dx \rightarrow 0, \text{ as } n \ra \infty.
\end{align*}
Thus, the claim holds.

\noi \textbf{Step 2:} Suppose $u_n \ra \tilde{u}$ in $\wps$. From the continuity of $I_{\mu,\al,f}$, we get $\eta = I_{\mu,\al,f}(\tilde{u})$, and Theorem \ref{PS-decomposition} holds for $k=0$. So, we assume that  $u_n \not\ra \tilde{u}$ in $\wps$. In view of Step 1, $\prescript{}{(\wps)^*}{\langle} I_{\mu,\al,0}'(\tilde{u}_n), \tilde{u}_n {\rangle}_{\wps} \ra 0$, which implies 
\begin{align}\label{del-0}
    0 < c \le C_{\text{eqiv}} [\tilde{u}_n]_{s,p}^p \le  [\tilde{u}_n]^p_{s,p} - \mu \int_{\rd} \frac{\abs{\tilde{u}_n}^p}{\abs{x}^{sp}} \dx = \int_{\rd} \frac{\abs{\tilde{u}_n}^{p^*_s(\al)}}{\abs{x}^{\al}} \dx + o_n(1).
\end{align}
In this step, we construct a sequence $\{ \hat{u}_n \}$ from $\{ \tilde{u}_n \}$ in such a way that the $\wps$-norm is preserved and $\{ \hat{u}_n \}$ weakly goes to a non-zero limit $\hat{u} \in \wps$. Moreover, we show that $\hat{u}$ weakly solves the limiting equation \eqref{limit-problem-intro}-(II). In view of \eqref{del-0}, there exists $\delta_1>0$ such that 
\begin{align*}
    \inf_{n \in \N} \int_{\rd}  \frac{\abs{\tilde{u}_n}^{p^*_s(\al)}}{\abs{x}^{\al}} \dx = \delta_1.
\end{align*}
We take $0< \de< \de_1$ and consider the Levy concentration function
\begin{align*}
    Q_n(r) \coloneqq \int_{B(0,r)}  \frac{\abs{\tilde{u}_n}^{p^*_s(\al)}}{\abs{x}^{\al}} \dx.
\end{align*}
Observe that $Q_n(0) =0$ and $Q_n(\infty) > \delta$. Further, we can verify that $Q_n$ is continuous on $\R^+$ (see \cite[Lemma 3.1]{Brasco-2018}). Hence, there exists $k \in \N$ and $r_n \in \R^+$  such that for all $n \ge k$, 
\begin{align}\label{int-1}
Q_n(r_n) = \int_{B(0, r_n)}  \frac{\abs{\tilde{u}_n}^{p^*_s(\al)}}{\abs{x}^{\al}} \dx = \delta. 
\end{align}
For $n \ge k$, we set 
\begin{align*}
\hat{u}_n(z) \coloneqq r_n^{\frac{d-sp}{p}} \tilde{u}_n(r_n z ), \text{ for } z \in \rd. 
\end{align*}
Using the change of variables and using \eqref{int-1}, observe that
\begin{align}\label{int-1.5}
\int_{B(0, 1)} \frac{\abs{\hat{u}_n}^{p^*_s(\al)}}{\abs{x}^{\al}} \dx = \delta. 
 \end{align}  
By observing the fact that $[\tilde{u}_n]_{s,p} = [\hat{u}_n]_{s,p}$, the sequence $\{ \hat{u}_n\}$ is bounded in $\wps$.  Let $\hat{u}_n \rightharpoonup \hat{u}$ in $\wps$. 
In this step, we show that $\hat{u} \neq 0$. Suppose $\hat{u} = 0$. Consider $\phi \in \cc(B(0,1))$ with $0 \le \phi \le 1$. 
Set
\begin{align*}
    \phi_n(z) \coloneqq \phi\left(\frac{z}{r_n}\right) \tilde{u}_n(z), \text{ for } z \in \rd. 
\end{align*}
Note that $\text{supp}(\phi_n) \subset B(0, r_n)$. Since $\{ \tilde{u}_n \}$ is a (PS) sequence of $I_{\mu,\al,0}$, we have  
\begin{align}\label{int-2}
  \mathcal{A}(\tilde{u}_n, \phi_n) = \mu \int_{\rd} \frac{\abs{\tilde{u}_n}^{p-2} \tilde{u}_n}{\abs{x}^{sp}} \phi_n \dx + \int_{\rd} \frac{\abs{\tilde{u}_n}^{p^*_s(\al)-2} \tilde{u}_n}{\abs{x}^{\al}} \phi_n \dx + o_n(1).
\end{align}
Now we estimate $\mathcal{A}(\tilde{u}_n, \phi_n)$. Using the change of variables $\bar{x}_n = \frac{x}{r_n}, \bar{y}_n = \frac{y}{r_n}$, we write 
\begin{align*}
    &\mathcal{A}(\tilde{u}_n, \phi_n) \\
    &= \iint_{\rd \times \rd} \frac{\abs{\tilde{u}_n(x) - \tilde{u}_n(y)}^{p-2}(\tilde{u}_n(x) - \tilde{u}_n(y)) \left( \phi(\frac{x}{r_n}) \tilde{u}_n(x) -  \phi(\frac{y}{r_n}) \tilde{u}_n(y) \right)}{\abs{x-y}^{d+sp}} \dxy \\
    &= r_n^{d-sp} \iint_{\rd \times \rd} \frac{\abs{\tilde{u}_n(r_n \bar{x}_n) - \tilde{u}_n(r_n \bar{y}_n)}^{p-2}(\tilde{u}_n(r_n \bar{x}_n) - \tilde{u}_n(r_n \bar{y}_n))}{\abs{\bar{x}_n-\bar{y}_n}^{d+sp}} \\
    & \quad \quad \quad \quad \left( \phi(\bar{x}_n)  \tilde{u}_n(r_n \bar{x}_n) -  \phi(\bar{y}_n) \tilde{u}_n(r_n \bar{y}_n) \right) \dxnyn \\
    & = \iint_{\rd \times \rd} \frac{\abs{\hat{u}_n(\bar{x}_n) - \hat{u}_n(\bar{y}_n)}^{p-2}(\hat{u}_n(\bar{x}_n) - \hat{u}_n(\bar{y}_n)) \left( \phi(\bar{x}_n) \hat{u}_n(\bar{x}_n) - \phi(\bar{y}_n) \hat{u}_n(\bar{y}_n) \right)}{\abs{\bar{x}_n-\bar{y}_n}^{d+sp}} \dxnyn.
\end{align*}
Applying the H\"{o}lder's inequality with the  pair $(p,p')$,
\begin{align}\label{integral-estimate}
    \left| \mathcal{A}(\tilde{u}_n, \phi_n) \right| \le [\hat{u}_n]_{s,p}^{p-1} \left( \iint_{\rd \times \rd} \frac{\abs{\phi(x)\hat{u}_n(x) - \phi(y)\hat{u}_n(y)}^p}{\abs{x-y}^{d+sp}} \dxy \right)^{\frac{1}{p}}.
\end{align}
To estimate the right-hand side integral of \eqref{integral-estimate}, we split 
\begin{align*}
    &\iint_{\rd \times \rd} \frac{\abs{\phi(x)\hat{u}_n(x) - \phi(y)\hat{u}_n(y)}^p}{\abs{x-y}^{d+sp}} \dxy \\
    & = \left( \iint_{B(0,1) \times B(0,1)} + 2 \iint_{B(0,1) \times B(0,1)^c}\right) \frac{\abs{\phi(x)\hat{u}_n(x) - \phi(y)\hat{u}_n(y)}^p}{\abs{x-y}^{d+sp}} \dxy := I_1 + I_2.
\end{align*}
We now show that $I_1 = o_n(1)$. For that
\begin{align*}
    &\iint_{B(0,1) \times B(0,1)} \frac{\abs{\phi(x)\hat{u}_n(x) - \phi(y)\hat{u}_n(y)}^p}{\abs{x-y}^{d+sp}} \dxy \\
    &\le 2^{p-1} \iint_{B(0,1) \times B(0,1)} \left( \abs{\hat{u}_n(x)}^p  \frac{\abs{\phi(x) - \phi(y)}^p}{\abs{x-y}^{d+sp}} + \abs{\phi(y)}^p \frac{\abs{ \hat{u}_n(x)- \hat{u}_n(y) }^p}{\abs{x-y}^{d+sp}} \right) \dxy,
\end{align*}
where 
\begin{align*}
    \iint_{B(0,1) \times B(0,1)} \abs{\phi(y)}^p \frac{\abs{ \hat{u}_n(x)- \hat{u}_n(y) }^p}{\abs{x-y}^{d+sp}} \dxy \le \norm{\phi}_{L^{\infty}(\rd)}^p \norm{ \hat{u}_n }_{\wps}^p \le C.
\end{align*}
Moreover, using $\abs{\phi(x) - \phi(y)} \le C\abs{x-y}$, we see that 
\begin{align*}
    \iint_{B(0,1) \times B(0,1)} \abs{\hat{u}_n(x)}^p  \frac{\abs{\phi(x) - \phi(y)}^p}{\abs{x-y}^{d+sp}} \dxy & \le C^p \iint_{B(0,1) \times B(0,1)}  \frac{\abs{\hat{u}_n(x)}^p}{\abs{x-y}^{d+sp-p}} \dxy \\
    & \le C^p \int_{B(0,1)} \left( \int_{B(0,2)} \frac{\dz}{\abs{z}^{d+sp-p}} \right) \abs{\hat{u}_n(x)}^p \dx \le C.
\end{align*}
This proves the finiteness of the integral. Next, we prove that $I_2$ is finite. Observe that
\begin{align*}
    I_2 = \iint_{B(0,1) \times B(0,1)^c} \frac{\abs{\phi(x)\hat{u}_n(x)}^p}{\abs{x-y}^{d+sp}} \dxy \le \norm{\phi}_{L^{\infty}(\rd)}^p \iint_{B(0,1) \times B(0,1)^c} \frac{\abs{\hat{u}_n(x)}^p}{\abs{y-x}^{d+sp}} \dxy,
\end{align*}
where using the change of variables, we estimate the last integral as
\begin{align*}
    \int_{B(0,1)} \abs{\hat{u}_n(x)}^p  \left( \int_{\abs{z}>1} \frac{\dz}{\abs{z}^{d+sp}} \right) \dx \le C(d,s,p) \int_{B(0,1)} \abs{\hat{u}_n(x)}^p \dx = o_n(1), 
\end{align*}
where $o_n(1)$ comes from the compact embedding $\wps \hookrightarrow L^p_{loc}(\rd)$ and $\hat{u}=0$. Hence $I_2=o_n(1)$, and 
\begin{align}\label{finiteness}
    \iint_{\rd \times \rd} \frac{\abs{\phi(x)\hat{u}_n(x) - \phi(y)\hat{u}_n(y)}^p}{\abs{x-y}^{d+sp}} \dxy \le C,
\end{align}
for some $C>0$. Next, we show that $I_1 = o_n(1)$. From the compact embedding of $\wps \hookrightarrow L^{p}_{loc}(\rd)$ and $\hat{u}=0$, we have $\hat{u}_n(x) \ra 0$ pointwise a.e. $x \in B(0,1)$. This implies $\abs{\phi(x)\hat{u}_n(x) - \phi(y)\hat{u}_n(y)} \ra 0$ pointwise a.e. $x,y \in B(0,1)$. Define
\begin{align*}
    U_n(x,y) := \frac{(\phi(x)\hat{u}_n(x) - \phi(y)\hat{u}_n(y))}{\abs{x-y}^{\frac{d+sp}{p}}}, \; \text{ for } x,y \in \rd. 
\end{align*}
In view of \eqref{finiteness}, 
\begin{align*}
    \norm{U_n}_{L^p(\rd \times \rd)} \le C, \; \forall \, n \in \mathbb{N}. 
\end{align*}
Therefore, applying Vitali's convergence theorem, we conclude that
\begin{align*}
    \lim_{n \ra \infty} \iint_{B(0,1) \times B(0,1)} \frac{\abs{\phi(x)\hat{u}_n(x) - \phi(y)\hat{u}_n(y)}^p}{\abs{x-y}^{d+sp}} \dxy = \lim_{n \ra \infty} \norm{U_n}_{L^p(B(0,1) \times B(0,1))} =  0.
\end{align*}
The above convergence yields $I_1 = o_n(1)$. Accumulating all the estimates, we get $\mathcal{A}(\tilde{u}_n, \phi_n) = o_n(1)$. Now we show that 
\begin{align}\label{smallness-1}
    \int_{\rd} \frac{\abs{\tilde{u}_n}^{p-2} \tilde{u}_n}{\abs{x}^{sp}} \phi_n \dx = o_n(1). 
\end{align}
Using the change of variables and \eqref{HS}, we write 
\begin{align*}
    \int_{\rd} \frac{\abs{\tilde{u}_n}^{p-2} \tilde{u}_n}{\abs{x}^{sp}} \phi_n \dx = \int_{B(0,1)} \frac{\abs{\hat{u}_n}^{p}}{\abs{x}^{sp}} \phi \dx \le C(d,s,p) \norm{\phi}_{L^{\infty}(\rd)} [\hat{u}_n ]^p_{s,p} \le C.
\end{align*}
Further, using the compact embedding $\wps \hookrightarrow L^p_{loc}(\rd)$ and $\hat{u}=0$, we see that $\frac{\abs{\hat{u}_n}^{p}}{\abs{x}^{sp}} \phi(x) \ra 0$ a.e. in $B(0,1)$. The Vitali's convergence theorem yields
\begin{align*}
    \lim_{n \ra \infty} \int_{B(0,1)} \frac{\abs{\hat{u}_n}^{p}}{\abs{x}^{sp}} \phi \dx=0,
\end{align*}
which implies \eqref{smallness-1}. 
Hence, in view of  \eqref{int-2}, we have 
\begin{align*}
o_n(1) = \int_{\rd} \frac{\abs{\tilde{u}_n}^{p^*_s(\al)-2} \tilde{u}_n}{\abs{x}^{\al}} \phi_n \dx = \int_{\rd} \frac{\abs{\hat{u}_n}^{p^*_s(\al)}}{\abs{x}^{\al}} \phi \dx,   
\end{align*}
where the last identity holds using the change of variables. Since $\phi \in \cc(B(0,1))$ is arbitrary,  for any $r\in (0,1)$ we can choose $\phi \equiv 1$ on $B_r$. Therefore,
\begin{align*}
 o_n(1) =  \int_{B_r}  \frac{\abs{\hat{u}_n}^{p^*_s(\al)}}{\abs{x}^{\al}} \dx,\text{ for any }0<r<1,
\end{align*}
which contradicts \eqref{int-1.5}. Thus, we conclude $\hat{u} \neq 0$. Next, we show that 
\begin{align}\label{limitofrn}
\lim_{n \ra \infty} r_n \in \{ 0, \infty \}.     
\end{align}
On the contrary, assume that $r_n \ra r_0$ for some $r_0>0$. Since $\hat{u} \neq 0$, we can choose $R>>1$ large enough so that $\norm{ \hat{u}}_{L^p(B(0,R))} >0$. Now using the compact embedding of $\wps \hookrightarrow L_{loc}^p(\rd)$ and applying the change of variables, we see that
\begin{align}\label{limitofrn-1}
    0 < \norm{ \hat{u}}_{L^p(B(0,R))} = \norm{ \hat{u}_n }_{L^p(B(0,R))} +o_n(1) = r_n^{-s} \norm{\tilde{u}_n}_{L^p(B(0,r_nR)} + o_n(1).
\end{align}
Further, since $r_n \ra r_0$, there exists $R_1>0$ such that $B(0,r_nR) \subset B(0,R_1)$. Now again using the compact embedding of $\wps \hookrightarrow L_{loc}^p(\rd)$, 
\begin{align*}
   \lim_{n \ra \infty} r_n^{-s} \norm{\tilde{u}_n}_{L^p\left(B(0,r_nR)\right)} \le r_0^{-s} \lim_{n \ra \infty} \norm{\tilde{u}_n}_{L^p\left(B(0,R_1)\right)} = 0,
\end{align*}
which contradicts \eqref{limitofrn-1}. Therefore, \eqref{limitofrn} holds.

Next, we show that the non-zero weak limit $\hat{u}$ weakly solves the following limiting equation
\begin{align*}
    (-\Delta_p)^s \hat{u} -\mu\dfrac{\abs{\hat{u}}^{p-2}\hat{u}}{|x|^{sp}}=\dfrac{|\hat{u}|^{p^*_s(\al)-2}\hat{u}}{|x|^{\al}} \;\mbox{ in }\,\mathbb{R}^d. 
\end{align*}
Take $\phi, \psi \in \wps$.  From Step 2, since $\hat{u}_n \rightharpoonup \hat{u}$, using Lemma \ref{convergence-integrals}-(ii) we get $ \lim_{n\to\infty} \mathcal{A}(\hat{u}_n, \phi) = \mathcal{A}(\hat{u}, \phi).$ For $n \in \mathbb{N}$, we set 
\begin{align*}
    \phi_n(z) = r_n^{- \frac{d-sp}{p}} \phi \left(\frac{z}{r_n} \right), \text{ for } z \in \rd.
\end{align*}
Note that $[\phi_n]_{s,p} = [\phi]_{s,p}$. Next, using the change of variables $\overline{x}_n = r_n x,\,\overline{y}_n = r_n y$, 
\begin{align}\label{invariant-1}
    & \mathcal{A}(\hat{u}_n, \phi) = r_n^{\frac{d-sp}{p'}} \iint_{\rd \times \rd} \frac{\abs{\tilde{u}_n(r_n x) - \tilde{u}_n(r_n y)}^{p-2} (\tilde{u}_n(r_n x) - \tilde{u}_n(r_n y)) (\phi(x) - \phi(y))}{\abs{x-y}^{d+sp}}\dx\dy \no \\
    & = r_n^{-\frac{d-sp}{p}} \iint_{\rd \times \rd} \frac{\abs{\tilde{u}_n(r_n x) - \tilde{u}_n(r_n y)}^{p-2} (\tilde{u}_n(r_n x) - \tilde{u}_n(r_n y))(\phi(x) - \phi(y))}{\abs{r_n x-r_n y}^{d+sp}}  \dx\dy  \no \\
    & = r_n^{- \frac{d-sp}{p}} \iint_{\rd \times \rd} \frac{\abs{\tilde{u}_n(\overline{x}_n) - \tilde{u}_n(\overline{y}_n)}^{p-2} (\tilde{u}_n(\overline{x}_n) - \tilde{u}_n(\overline{y}_n)) \left(\phi \left(\frac{\overline{x}_n}{r_n} \right) - \phi \left(\frac{\overline{y}_n}{r_n} \right) \right)}{\abs{\overline{x}_n - \overline{y}_n}^{d+sp}} \d \overline{x}_n \d \overline{y}_n \no \\
    &= \mathcal{A}(\tilde{u}_n, \phi_n).
\end{align}
Similarly, for $\tilde{\al} \in [0, sp]$, we also have  
\begin{align}\label{invariant-2}
    \int_{\rd} \frac{\abs{\hat{u}_n}^{p^*_s(\tilde{\al})-2} \hat{u}_n}{\abs{x}^{\tilde{\al}}} \phi \dx = \int_{\rd} \frac{\abs{\tilde{u}_n}^{p^*_s(\tilde{\al})-2} \tilde{u}_n}{\abs{x}^{\tilde{\al}}} \tilde{\phi}_n \dx.  
\end{align}
Now using $\prescript{}{(\wps)^*}{\langle} I_{\mu,\al,0}'(\tilde{u}_n), \phi_n {\rangle}_{\wps} =o_n(1)$, \eqref{invariant-1}, and \eqref{invariant-2}, we see that 
\begin{align}\label{limit-1}
    \mathcal{A}(\hat{u}_n, \phi) - \mu \int_{\rd} \frac{\abs{\hat{u}_n(x)}^{p-2} \hat{u}_n(x)}{\abs{x}^{sp}} \phi(x) \dx &= \mathcal{A}(\tilde{u}_n, \phi_n) - \mu \int_{\rd} \frac{\abs{\tilde{u}_n(x)}^{p-2} \tilde{u}_n(x)}{\abs{x}^{sp}} \phi_n(x) \dx \no \\
    &=\int_{\rd} \frac{\abs{\tilde{u}_n(x)}^{p^*_s(\al)-2}\tilde{u}_n(x)}{\abs{x}^{\al}} \phi_n(x) \dx + o_n(1) \no \\
    &=\int_{\rd} \frac{\abs{\hat{u}_n(x)}^{p^*_s(\al)-2} \hat{u}_n(x)}{\abs{x}^{\al}} \phi(x) \dx + o_n(1),
\end{align}
where the last line is obtained again by using the change of variables. Now taking $n \ra \infty$ in \eqref{limit-1} and applying Lemma \ref{convergence-integrals}-(i), we complete this step. 

\noi \textbf{Step 3:} We set 
\begin{align*}
    w_n(z) = \tilde{u}_n(z) - r_n^{-\frac{d-sp}{p}} \hat{u} \left( \frac{z}{r_n} \right) \text{ and } \tilde{w}_n(z) = r_n^{\frac{d-sp}{p}} w_n(r_n z), \text{ for } z \in \rd.
\end{align*}
Note that $[w_n]_{s,p} = [\tilde{w}_n]_{s,p}$. This step shows that $\{w_n\}$ is a (PS) sequence of $I_{\mu,\al,0}$ at level $\eta - I_{\mu,\al,f}(\tilde{u}) - I_{\mu,\al,0}(\hat{u})$. Observe that $\tilde{w}_n = \hat{u}_n - \hat{u}$. Hence the norm invariance gives $[ w_n ]_{s,p} = [ \tilde{w}_n ]_{s,p} = [ \hat{u}_n - \hat{u} ]_{s,p}$. Now applying Lemma \ref{convergence-BL}-((i),(ii)), we see that 
\begin{align*}
    I_{\mu,\al,0} (w_n) & = \frac{1}{p}[w_n]_{s,p}^p - \frac{\mu}{p} \int_{\rd}\frac{\abs{w_n}^p}{\abs{x}^{sp}} \dx - \frac{1}{p^*_s(\al)} \int_{\rd} \frac{\abs{w_n}^{p^*_s(\al)}}{\abs{x}^{\al}}\dx \\
    & = \frac{1}{p} \left( [\hat{u}_n]_{s,p}^p - [\hat{u}]_{s,p}^p \right) - \frac{\mu}{p} \left( \int_{\rd}\frac{\abs{\hat{u}_n}^p}{\abs{x}^{sp}} \dx - \int_{\rd}\frac{\abs{\hat{u}}^p}{\abs{x}^{sp}} \dx \right) \\
    & - \frac{1}{p^*_s(\al)} \left( \int_{\rd} \frac{\abs{\hat{u}_n}^{p^*_s(\al)}}{\abs{x}^{\al}}\dx - \int_{\rd} \frac{\abs{\hat{u}}^{p^*_s(\al)}}{\abs{x}^{\al}}\dx \right) + o_n(1) \\
    & = \frac{1}{p} \left( [\tilde{u}_n]_{s,p}^p - [\hat{u}]_{s,p}^p \right) - \frac{\mu}{p} \left( \int_{\rd}\frac{\abs{\tilde{u}_n}^p}{\abs{x}^{sp}} \dx - \int_{\rd}\frac{\abs{\hat{u}}^p}{\abs{x}^{sp}} \dx \right) \\
    & - \frac{1}{p^*_s(\al)} \left( \int_{\rd} \frac{\abs{\tilde{u}_n}^{p^*_s(\al)}}{\abs{x}^{\al}}\dx - \int_{\rd} \frac{\abs{\hat{u}}^{p^*_s(\al)}}{\abs{x}^{\al}}\dx \right) + o_n(1) \\
    &= I_{\mu, \al,0}(\tilde{u}_n) - I_{\mu, \al, 0}(\hat{u}) + o_n(1) = \eta - I_{\mu,\al,f}(\tilde{u}) - I_{\mu, \al, 0}(\hat{u}) + o_n(1).
\end{align*}
Next, we show $\prescript{}{(\wps)^*}{\langle} I_{\mu,\al,0}'(w_n), \phi {\rangle}_{\wps} \ra 0$ for every $\phi \in \wps$. Using the fact that $\{w_n\}$ is bounded in $\wps$, we apply \eqref{HS1}, to obtain the following for some $C>0$:
\begin{align*}
    \left| \prescript{}{(\wps)^*}{\langle} I_{\mu,\al,0}'(w_n), \phi {\rangle}_{\wps} \right| \le C[\phi]_{s,p}, \text{ for every } \phi \in \wps. 
\end{align*}
Hence using the density of $\cc (\rd)$ in $\wps$, and using the uniform boundedness principle, it is enough to show $\prescript{}{(\wps)^*}{\langle} I_{\mu,\al,0}'(w_n), \phi {\rangle}_{\wps} \ra 0$ for every $\phi \in \cc (\rd)$. For $\phi \in \cc (\rd)$, we define 
\begin{align*}
    \hat{\phi}_n(z) =  r_n^{\frac{d-sp}{p}} \phi(r_n z), \text{ for } z \in \rd.
\end{align*}
Since $[ \hat{\phi}_n ]_{s,p} = [ \phi ]_{s,p}$, the sequence $\{ \hat{\phi}_n \}$ is bounded in $\wps$, and  and up to a subsequence $\hat{\phi}_n \rightharpoonup v_{\al}$ in $\wps$. Since $r_n \ra 0$ or $\infty$, in either case we get $\hat{\phi}_n\rightharpoonup 0$ in $\wps$. Indeed, when $r_n\to 0$, $\hat{\phi}_n\to 0$ uniformly in $\R^{d}$ and $\{ \hat{\phi}_n\}$ is bounded in $\wps$, so it has  a weak limit (up to a subsequence) in $\wps$, which must coincide with $0$. Therefore, $v_{\al}=0$ a.e. in $\rd$. Now, when $r_n\to\infty$, $\supp({\phi}) \subset B(0,R)$, for some $R>0$, will give $\supp({\hat{\phi}_n})$ to concentrate at a point (maybe at infinity), which renders $\hat{\phi}_n \rightharpoonup0$ in $\wps$. Now, using the change of variables, 
\begin{align}\label{PS-sublevel}
    & \prescript{}{(\wps)^*}{\langle} I_{\mu,\al,0}'(w_n), \phi {\rangle}_{\wps} \no \\ & = \mathcal{A}(w_n, \phi) - \mu \int_{\rd} \frac{\abs{w_n(x)}^{p-2} w_n(x)}{\abs{x}^{sp}} \phi(x) \dx - \int_{\rd} \frac{\abs{w_n(x)}^{p^*_s(\al)-2} w_n(x) }{\abs{x}^{\al}} \phi(x) \dx \no \\
    & = \mathcal{A}(\tilde{w}_n, \hat{\phi}_n) - \mu \int_{\rd} \frac{\abs{\tilde{w}_n(x)}^{p-2} \tilde{w}_n(x)}{\abs{x}^{sp}} \hat{\phi}_n(x) \dx - \int_{\rd} \frac{\abs{\tilde{w}_n(x)}^{p^*_s(\al)-2} \tilde{w}_n(x) }{\abs{x}^{\al}} \hat{\phi}_n(x) \dx.
\end{align}
Using Lemma \ref{convergence-BL}-(iii) and using H\"{o}lder's inequality with the  pair $(p, p')$ and further using $[\hat{\phi}_n]_{s,p} = [\phi]_{s,p}$ we get $ \mathcal{A}(\tilde{w}_n, \hat{\phi}_n) -  \mathcal{A}(\hat{u}_n, \hat{\phi}_n) + \mathcal{A}(\hat{u}, \hat{\phi}_n) = o_n(1)$. Further, the change of variables yield $\mathcal{A} (\tilde{w}_n, \hat{\phi}_n) -  \mathcal{A}(\tilde{u}_n, \phi) + \mathcal{A}(\hat{u}, \hat{\phi}_n) = o_n(1)$. Now using the fact that $\tilde{u}_n \rightharpoonup 0$ and $\hat{\phi}_n \rightharpoonup 0$ in $\wps$, applying Lemma \ref{convergence-integrals}-((ii) and (iii)), we get $\mathcal{A}(\tilde{u}_n, \phi) = o_n(1)$ and $\mathcal{A}(\hat{u}, \hat{\phi}_n) = o_n(1)$. Therefore, $\mathcal{A}(\tilde{w}_n, \hat{\phi}_n) = o_n(1)$. Similarly, using Lemma \ref{convergence-BL}-(iv) and H\"{o}lder's inequality with the  pair $(p^*_s(\al), (p^*_s(\al))')$, and then using \eqref{HS1} and $[\hat{\phi}_n]_{s,p} = [\phi]_{s,p}$ we get the following 
for $\tilde{\al} \in [0, sp]$:
\begin{align*}
   &\int_{\rd} \frac{\abs{\tilde{w}_n(x)}^{p^*_s(\tilde{\al})-2} \tilde{w}_n(x) }{\abs{x}^{\tilde{\al}}} \hat{\phi}_n(x) \dx - \int_{\rd} \frac{\abs{\hat{u}_n(x)}^{p^*_s(\tilde{\al})-2} \hat{u}_n(x) }{\abs{x}^{\tilde{\al}}} \hat{\phi}_n(x) \dx \\
   &=\int_{\rd} \frac{\abs{\hat{u}(x)}^{p^*_s(\tilde{\al})-2} \hat{u}(x) }{\abs{x}^{\tilde{\al}}} \hat{\phi}_n(x) \dx + o_n(1).
\end{align*}
Again, the change of variables yields, 
\begin{align*}
    &\int_{\rd} \frac{\abs{\tilde{w}_n(x)}^{p^*_s(\tilde{\al})-2} \tilde{w}_n(x) }{\abs{x}^{\tilde{\al}}} \hat{\phi}_n(x) \dx - \int_{\rd} \frac{\abs{\tilde{u}_n(x)}^{p^*_s(\tilde{\al})-2} \tilde{u}_n(x) }{\abs{x}^{\tilde{\al}}} \phi(x) \dx \\
   &=\int_{\rd} \frac{\abs{\hat{u}(x)}^{p^*_s(\tilde{\al})-2} \hat{u}(x) }{\abs{x}^{\tilde{\al}}} \hat{\phi}_n(x) \dx + o_n(1).
\end{align*}
Therefore, using Lemma \ref{convergence-integrals}-(i), for $\tilde{\al} \in [0, sp]$, we get
\begin{align*}
    \int_{\rd} \frac{\abs{\tilde{w}_n(x)}^{p^*_s(\tilde{\al})-2} \tilde{w}_n(x) }{\abs{x}^{\tilde{\al}}} \hat{\phi}_n(x) \dx = o_n(1). 
\end{align*}
In view of \eqref{PS-sublevel}, we finally get $\prescript{}{(\wps)^*}{\langle} I_{\mu,\al,0}'(w_n), \phi {\rangle}_{\wps} = o_n(1)$. Thus, $\{w_n\}$ becomes a (PS) sequence of $I_{\mu,\al,0}$ at level $\eta - I_{\mu,\al,f}(\tilde{u}) - I_{\mu,\al,0}(\hat{u})$.

\noi \textbf{Step 4:} Now, starting from a (PS) sequence $\{ \tilde{u}_n \}$ of $I_{\mu,\al,0}$ we have extracted further (PS) sequences at a level which is strictly lower than the level of $\{ \tilde{u}_n \}$, and with a fixed amount of decrease in every step, since
\begin{align}\label{fixed amount decrease}
    I_{\mu,\al,0}(\hat{u}) \ge \frac{sp-\al}{p(d-\al)} S_{\mu,\al}^{\frac{d-\al}{sp-\al}}.
\end{align}

For the verification of \eqref{fixed amount decrease}, from Step 2, we note that 
\begin{align*}
    [\hat{u}]_{\mu}^{p} = \int_{\rd} \frac{|\hat{u}(x)|^{p^*_s(\al)}}{|x|^{\al}} \dx,
\end{align*}
which implies $I_{\mu,\al,0}(\hat{u}) = \frac{sp-\al}{p(d-\al)}[\hat{u}]_{\mu}^p$, and moreover, in view of \eqref{HS1}, the above identity gives
\begin{align*}
    [\hat{u}]_{\mu}^{p} \ge S_{\mu,\al} [\hat{u}]_{\mu}^{\frac{p}{p^*_s(\al)}} \Longrightarrow \left( [\hat{u}]_{\mu}^p \right)^{\frac{sp-\al}{d-\al}} \ge S_{\mu,\al}.
\end{align*}
Therefore, \eqref{fixed amount decrease} holds.

Since $\sup_n[\tilde u_n]_{\mu}< \infty$, there exists $k\in \N$ such that this process terminates after the $ k$-th step and the last (PS) sequence strongly converges to $0$. Let $\tilde{u}_1$ and $\tilde{u}_2$ be two non-zero weak limits appearing from two different (PS) sequences of distinct levels. Then in the same spirit of \cite{Tintarev} (Page 130, Theorem~3.3) and using \cite[Lemma 2.6]{NS2025}, we get 
\begin{align*}
    \mathcal{A}(C_{r^1_n} \tilde{u}_1,C_{r^2_n} \tilde{u}_2) = \mathcal{A}\left(\tilde{u}_1, C_{\frac{r_n^2}{r_n^1}} \tilde{u}_2 \right) \ra 0,\mbox{ as }n\to\infty
\end{align*}
Hence, in view of Proposition \ref{weak-bub-II}, we get
\begin{align*}
  \bigg|\log\left(\frac{r^1_n}{r_n^2}\right)\bigg| \to\infty,\text{ as }n\to\infty.
\end{align*}
This completes the proof. \qed

\section{Global compactness result: The case $\alpha=0$}\label{se-4}
We next turn to the endpoint case $\alpha=0$. In contrast with the previous section, translations now preserve the critical nonlinear term, and the concentration centre may move independently of the Hardy singularity. This produces two different limiting problems and requires a separate centre-scale analysis. 
For $\al=0$, \eqref{MainEq} becomes
\begin{equation}\tag{$\mathcal{P}_0$}\label{mainalphazero}
    (-\Delta_p)^s u -\mu\dfrac{\abs{u}^{p-2}u}{|x|^{sp}}=|u|^{p^*_s-2}u+f \;\mbox{ in }\,\rd, \quad u \in \wps,\; f \in (\wps)^*.
\end{equation}
The energy functional for \eqref{mainalphazero} is given by
\begin{align} \label{eq:energy}
 I_{\mu,0,f}(u)=\frac{1}{p}[u]_{\mu}^p-\frac{1}{p^*_s}\int_{\rd}\abs{u}^{p^*_s}\dx-\prescript{}{(\wps)^*}{\langle}f,u{\rangle}_{\wps}, \; \text{for all }\, u \in \wps.
\end{align}
In particular, from \eqref{eq:energy}, we denote $I_{\mu,0,0}$ and $I_{0,0,0}$, by taking $f=0$, and by taking $f=0,\,\mu=0$ respectively. We recall the notation $\vartheta \coloneqq \frac{d-sp}{p}$, that will be used simultaneously.

We recall the following two best constants:
\begin{align}\label{eq:constants}
\overline{S}_{\mu} \coloneqq \inf_{u \in \wps \setminus \{0\}} \frac{[u]_{\mu}^p}{\displaystyle \left(\int_{\rd}\abs{u}^{p^*_s}\dx \right)^{\frac{p}{p^*_s}}}, \text{ and } \overline{S} \coloneqq \inf_{u \in \wps \setminus \{0\}} \frac{[u]_{s,p}^p}{\displaystyle \left(\int_{\rd}\abs{u}^{p^*_s}\dx \right)^{\frac{p}{p^*_s}}}.
\end{align}
From the equivalence of $[\cdot]_{\mu}$ and $[\cdot]_{s,p}$, we have 
\begin{align}\label{eq:cons-order}
 0<C_{\text{eqiv}}^p\, \overline{S} \le \overline{S}_{\mu}\le \overline{S}.
\end{align}
The minimizers of \eqref{eq:constants} weakly solve the following limiting equations (up to scalar muliplication):
\begin{align}
    &(-\Delta_p)^s v -\mu\dfrac{\abs{v}^{p-2}v}{|x|^{sp}}=|v|^{p^*_s-2}v \;\mbox{ in }\,\rd, \quad v \in \wps (\rd), \tag{$\mathcal{L}_{\mu}$}\label{limit-mu} \\
    &(-\Delta_p)^s w =|w|^{p^*_s-2}w \;\mbox{ in }\,\rd, \quad w \in \wps (\rd). \tag{$\mathcal{L}_{0}$}\label{limit-zero}
    \end{align}
%\souptik{I Think the equation \eqref{limit-mu} it is for all $u\in \wps (\rd\setminus \{0\})$ and the equation \eqref{limit-zero} is on whole $\wps$.}

\begin{proposition}
Let $\mu \in (0,\mu_{d,s,p})$. Then $\overline{S}_{\mu}<\overline{S}$. Further, there exists $V \in \wps \setminus \{0\}$ which weakly solves \eqref{limit-mu} and satisfies
     \begin{equation}\label{eq:attained}
          \frac{[V]_{\mu}^p}{\displaystyle \left( \int_{\rd}\abs{V}^{p^*_s}\dx\right)^{\frac{p}{p^*_s}}}=\overline{S}_{\mu}, \, \text{ and } \, [V]_{\mu}^p=\int_{\rd}\abs{V}^{p^*_s}\dx=\overline{S}_{\mu}^{\frac{d}{sp}}.
     \end{equation}
\end{proposition}
\begin{proof} 
From \cite{Brasco2016}, $\overline{S}$ is attained by some $U \in \wps \setminus \{0\}$. By \eqref{HS}, $\int_{\rd}\frac{\abs{U}^p}{\abs{x}^{sp}}\dx<\infty$, and it is positive as $U \ne 0$. Since $\mu>0$ we therefore have $[U]_{\mu}<[U]_{s,p}$, and so
\begin{align*}
    \overline{S}_{\mu}\le \frac{[U]_{\mu}^p}{\left( \int_{\rd}\abs{U}^{p^*_s}\dx\right)^{\frac{p}{p^*_s}}}<\frac{[U]_{s,p}^p}{\left( \int_{\rd}\abs{U}^{p^*_s}\dx\right)^{\frac{p}{p^*_s}}}=\overline{S}.
\end{align*}
The rest of the proof follows from \cite[Theorem 1.1]{Shen24}.
\end{proof}
\subsection{Behaviour of bounded sequences under translations and dilations} We begin by studying the joint action of translations and dilations. The aim is to identify when a translated and rescaled profile vanishes weakly and to determine how the Hardy term behaves when the concentration centre moves away from the origin.

Let $\mathcal{T} \subset \mathcal{U}(\wps)$ be the following classes of isometric operators induced by the composition of translations and dilations:
\begin{align*}
   \mathcal{T}\coloneqq \left\{C_{y,\lambda} \in \mathcal{U}(\wps): C_{y,\lambda}u(x)\coloneqq \lambda^{-\vartheta}u \left(\frac{x-y}{\lambda}\right), \; \forall \, u\in\wps;\,y\in\rd;\,\lambda \in (0, \infty)\right\}. 
\end{align*}
A change of variable shows that
\begin{equation}\label{eq:invariance}
[C_{y,\la}u]_{s,p}=[u]_{s,p}, \quad \int_{\rd}\abs{C_{y,\la}u}^{p^*_s}\dx = \int_{\rd}\abs{u}^{p^*_s}\dx, \; \forall\,y\in\rd;\,\lambda \in (0, \infty).
\end{equation}
Further, when $y=0$,
\begin{equation}\label{eq:hardy-scale}
    \int_{\rd}\frac{\abs{C_{0,\la}u}^p}{\abs{x}^{sp}}\dx = \int_{\rd}\frac{\abs{u}^p}{\abs{x}^{sp}}\dx, \; \forall\, \lambda \in (0, \infty).
\end{equation}
\begin{remark}\label{change-of-variable-I}
 Let $w \in \wps$. For $R_n>0$ and $x_n \in \rd$, we consider
\begin{gather*}
     g_n:=C_{x_n,R_n}w, \text{ and } \sigma_n:=\frac{x_n}{R_n}.
\end{gather*} 
For every $\tilde{\al} \in [0,sp]$, the change of variables $x=R_n z+x_n$ gives
\begin{equation}\label{eq:two-integrals}
    \int_{\rd}\frac{\abs{g_n(x)}^{p^*_s(\tilde{\al})}}{\abs{x}^{\tilde{\al}}}\dx = \int_{\rd}\frac{\abs{w(z)}^{p^*_s(\tilde{\al})}}{\abs{z+\sigma_n}^{\tilde{\al}}}\dz,
\end{equation}
where all the powers of $R_n$ cancel due to the fact that $ -\vartheta p^*_s(\tilde{\al}) + d - \tilde{\alpha} = 0.$
Suppose $\abs{\sigma_n}\ra \infty$, i.e., the point of concentration runs away from the origin. We consider $w \in \cc(\rd)$ with $\supp(w) \subset B(0,R)$. Choose $n$ large enough so that $\abs{\sigma_n}\ge 2R$, which gives
\begin{align*}
    |z+\sigma_n| = |z-(-\sigma_n)| \ge |\sigma_n|-|z| \ge |\sigma_n|-R \ge \frac{|\sigma_n|}{2}. 
\end{align*}
Therefore,
\begin{equation}\label{eq:1}
 \int_{\rd}\frac{\abs{g_n(x)}^{p^*_s(\tilde{\al})}}{\abs{x}^{\tilde{\al}}}\dx \le \frac{2^{\tilde{\al}}}{\abs{\sigma_n}^{\tilde{\al}}} \int_{\rd}\abs{w}^{p^*_s(\tilde{\al})}\dz.
\end{equation}
If $\tilde{\al}>0$, then the R.H.S of \eqref{eq:1} is $o_n(1)$. %On the other hand, when $\tilde{\al}=0$, the critical term does not move and as a consequence, a (PS) sequence corresponding to $I_{\mu,0,f}$ can lose compactness in two different ways.
On the other hand, when $\tilde \al=0$, the $L^{p_s^*}$ mass of $g_n$ remain only bounded. Thus, mass escaping to infinity is still a possibility in this case, and any (PS) sequence corresponding to $I_{\mu,0,f}$ may possess two different concentration profiles.
\end{remark}

\begin{remark}
A direct computation gives 
\begin{gather}\label{eq:composition}
 C_{y_1,\la_1}\circ C_{y_2,\la_2}=C_{y_1+\la_1y_2,\la_1\la_2}, \text{ and } \left( C_{y,\la}\right)^{-1}=C_{-\frac{y}{\la},\frac{1}{\la}}.
\end{gather}
Further, by \cite[Lemma 2.6]{NS2025},
\begin{gather}\label{eq:A-behviour}
\mathcal{A}\left(C_{y_1,\la_1}u,C_{y_2,\la_2}v\right)=\mathcal{A}\left(u,\,C_{\frac{y_2-y_1}{\la_1},\frac{\la_2}{\la_1}}v\right).
\end{gather}
The same identity holds for $\mathcal{A}_1$.   
\end{remark}

 \begin{lemma}\label{lem:1}
Let $\{v_n\} \subset \wps$ be bounded. Then the following hold:
\begin{enumerate}
    \item[\rm (a)] If $\int_{\rd}v_n \psi \dx \ra \int_{\rd}v \psi \dx$ for every $\psi \in \cc(\rd)$ and some $v \in \wps$, then $v_n \rightharpoonup v$ in $\wps$.
    \item[\rm (b)] If $v_n \rightharpoonup v$ in $\wps$, $\tau_n \ra \tau$ in $(0,\infty)$ and $\xi_n \ra \xi$ in $\rd$, then $C_{\xi_n,\tau_n}v_n \rightharpoonup C_{\xi,\tau}v$ in $\wps$.
\end{enumerate}
\end{lemma}

\begin{proof}
(a) Since, $\sup_{n}[v_n]_{s,p}\le M<+\infty$, every subsequence of $\{v_n\}$ possesses a further subsequence converging weakly to some element of $\wps$. First, we show that every weak subsequential limit is $v$. Suppose for some subsequence $\{v_{n_{k}}\}$ of $\{v_n\}$, we have $v_{n_k}\rightharpoonup w$ weakly in $\wps$. Observe that, every $g\in L^{(p_s^*)'}(\rd)$ induces an element $l_g\in (\wps(\rd))^*$ defined by 
    $\prescript{}{(\wps)^*}{\langle} l_g, f\rangle_{\wps}\coloneqq \int_{\rd} fg\,\dx$, for every $f\in\wps$.
    Then, for $l\in (\wps(\rd))^*$, by weak convergence, $\prescript{}{(\wps)^*}{\langle} l, v_{n_{k}}\rangle_{\wps}\ra \prescript{}{(\wps)^*}{\langle} l, w\rangle_{\wps}$ as $k\to \infty$. Now, taking $g=\psi \in\cc (\rd)\subset L^{(p_s^*)'}(\rd)$ and using the hypothesis we get, $\int_{\rd} (v-w)\psi\dx =0$ for every $\psi \in \cc(\rd)$, which further implies $v=w$ a.e. on $\rd$.
    \noindent Now fix $l\in (\wps(\rd))^*$ and set $a_n\coloneqq \prescript{}{(\wps)^*}{\langle} l, v_n\rangle_{\wps}$. Then $|a_n| \leq M\|l\|_{(\wps)^*}$. If possible, $\prescript{}{(\wps)^*}{\langle} l, v_n\rangle_{\wps} \not\to \prescript{}{(\wps)^*}{\langle} l, v\rangle_{\wps}$. Then, there exists $\var >0$ and a subsequence $\{a_{n_i}\}$ of $\{a_n\}$ with 
\begin{align}\label{wc-1}
    \left| a_{n_i}- \prescript{}{(\wps)^*}{\langle} l, v\rangle_{\wps}\right| \geq \var, \text{ for every }i\in\mathbb{N}.
\end{align}
    The corresponding subsequence $\{v_{n_i}\}$ is still bounded in $\wps$ and hence by previous argument $\{v_{n_i}\}$ has a further subsequence $\left\{v_{n_{i_j}}\right\}$ such that as $j\to\infty$, $v_{n_{i_j}} \rightharpoonup v$ weakly in $\wps$ which enforces $\prescript{}{(\wps)^*}{\langle} l, v_{n_{i_j}}\rangle_{\wps} \to \prescript{}{(\wps)^*}{\langle} l, v\rangle_{\wps}$ contradicting \eqref{wc-1}. This shows $\prescript{}{(\wps)^*}{\langle} l, v_{n}\rangle_{\wps}\to \prescript{}{(\wps)^*}{\langle} l, v\rangle_{\wps}$ for every $l\in (\wps(\rd))^*$ completing the proof.

%Every $g \in L^{(p^*_s)'}(\rd)$ induces an element of $(\wps)^*$, due to the dense embedding $\wps \embd L^{p^*_s}(\rd)$. So, up to a subsequence, if $v_n \rightharpoonup \tilde{v}$ in $\wps$, then $\int_{\rd} v_n g \ra \int_{\rd} \tilde{v}g$ for every $g \in L^{(p^*_s)'}(\rd)$. Taking $g=\psi \in \cc(\rd)$ and using the hypothesis, $\int_{\rd} (\tilde{v}-v)\psi=0$ for all $\psi \in \cc(\rd)$, which implies $\tilde{v}=v$ a.e. in $\rd$. Since $\wps$ is reflexive and $\{v_n\}$ is bounded, every subsequence of $\{v_n\}$ has again a subsequence converging weakly in $\wps$, and the weak limit is always $v$. Hence, the whole sequence converges weakly to $v$.

(b) By \eqref{eq:invariance}, $\left\{ C_{\tau_n,\xi_n}v_n\right\}$ is bounded in $\wps$. We fix $\psi \in \cc(\rd)$ with $\supp \psi \subset B(0,R)$. From the change of variables $x=\tau_n z+\xi_n$, we get 
\begin{equation*}
     \int_{\rd}(C_{\xi_n,\tau_n}v_n)(x)\psi(x) \dx = \tau_n^{d-\vartheta}\int_{\rd}v_n(z)\,\psi_n(z)\dz, \quad \psi_n(\cdot):=\psi \left( \tau_n \cdot+\xi_n\right).
\end{equation*}
Observe that $\norm{\psi_n}_{L^{\infty}(\rd)} \le \norm{\psi}_{L^{\infty}(\rd)}$ and  $\psi_n(x) \ra \psi(\tau x+\xi)$ for every $x \in \rd$. Since $\tau_n \ra \tau>0$ and $\{\xi_n\}$ is bounded, we note that, for every large $n \in \N$, $\psi_n(z)\ne 0$ implies $\abs{z}\le 2\left( R+\sup_n \abs{\xi_n}\right)/\tau$. By dominated convergence, $\psi_n \ra \psi(\tau \cdot+\xi)$ in $L^{(p^*_s)'}(\rd)$.
Since $v_n \rightharpoonup v$ in $L^{p^*_s}(\rd)$, and $\tau_n^{d-\vartheta}\ra \tau^{d-\vartheta}$, we get
\begin{equation*}
    \int_{\rd}(C_{\xi_n,\tau_n}v_n)(x)\psi(x) \dx \ra \tau^{d-\vartheta}\int_{\rd}v(z)\psi(\tau z+\xi)\dz=\int_{\rd}(C_{\xi,\tau}v)(x)\psi(x) \dx, \; \forall \, \psi \in \cc(\rd),
\end{equation*}
where the identity follows from a change of variables. Now, we apply (a), to conclude $C_{\xi_n,\tau_n}v_n \rightharpoonup C_{\xi,\tau}v$ in $\wps$.
The case $\tau_n=\tau=1$ gives us $v_n(\cdot-\xi_n)\rightharpoonup v(\cdot-\xi)$, since $(C_{1,\xi}u)(\cdot) = u(\cdot-\xi)$.
\end{proof}

\begin{lemma}\label{strong-convergence}
If $\la_n \ra \la$ in $(0,\infty)$ and $y_n \ra y$ in $\rd$, then $C_{y_n,\la_n}u \ra C_{y,\la}u$ in $\wps$ for every $u \in \wps$.
\end{lemma}
\begin{proof}
Let $u \in \cc(\rd)$. Then $C_{y_n,\la_n}u(x)\ra C_{y,\la}u(x)$ for every $x \in \rd$, and from \eqref{eq:invariance}, $[C_{y_n,\la_n}u]_{s,p}=[u]_{s,p}=[C_{y,\la}u]_{s,p}$. Hence using Lemma \ref{convergence-BL}-(i), $[C_{y_n,\la_n}u-C_{y,\la}u]_{s,p}^p=[C_{y_n,\la_n}u]_{s,p}^p-[C_{y,\la}u]_{s,p}^p+o_n(1)=o_n(1)$. Now, for $u \in \wps$ and $\var>0$ we choose $\phi \in \cc(\rd)$ with $[u-\phi]_{s,p}<\var/3$. Using \eqref{eq:invariance}, we have
\begin{gather*}
 [C_{y_n,\la_n}u-C_{y,\la}u]_{s,p}\le [C_{y_n,\la_n}(u-\phi)]_{s,p}+[C_{y_n,\la_n}\phi-C_{y,\la}\phi]_{s,p}+[C_{y,\la}(\phi-u)]_{s,p},
\end{gather*}
and the right-hand side is smaller than $\var$ for $n$ large.
\end{proof}
We define \begin{align}\label{eq:Lambda}
  \La_n:=\abs{\log \la_n}+\frac{\abs{y_n}}{\la_n}.
\end{align}
Notice that $\La_n \ra \infty$ if and only if $\abs{\log \la_n}+\abs{y_n}\ra \infty$: if $\abs{\log \la_n}\le M$ then $e^{-M} \le \la_n \le e^M$, so $\abs{y_n}$ and $\abs{y_n}/\la_n$ are bounded simultaneously. In \cite[Lemma~3]{Palatucci-Pisante-CVPDE}, for $p=2$, the authors also established that if $\left| \log(\lambda_n) \right| + |y_n| \to \infty$, then $C_{y_n, \la_n } u \rightharpoonup 0$ in $\mathcal{D}^{s,2}$ for every $u \in \mathcal{D}^{s,2}$. Their proof mainly uses stronger density results for $\mathcal{D}^{s,2}$, namely the density of $A\coloneqq \left\{ f \in \mathcal{S}(\rd): \mathcal{F}(f) \in \C_{c}^{\infty}(\rd \setminus \{ 0\}) \right\}$ in $\mathcal{D}^{s,2}$ (note that $\cc(\rd)$ is not contained in $A$). It follows that, if $u\in A$ then one has $(-\Delta)^su \in \mathcal{S}(\rd)$ (using the Fourier representation of $(-\Delta)^s$). However, for $p \neq 2$, we do not have a similar density result in $\wps$ and the Fourier representation of $(-\Delta_p)^s$.

The following proposition measures the non-compactness of $\wps \hookrightarrow L^{p^*_s}(\rd)$ under the conformal group action of dilations and translations.

\begin{proposition}\label{lem:2}
Let $\{\la_n\}\subset (0,\infty)$, $\{y_n\}\subset \rd$ and let $\La_n$ be as in (\ref{eq:Lambda}). Then the following are equivalent:
\begin{enumerate}
    \item[\rm (i)] $\La_n \ra \infty$.
    \item[\rm (ii)] $C_{y_n,\la_n}v \rightharpoonup 0$ in $\wps$ for every $v \in \wps$.
    \item[\rm (iii)] $\mathcal{A}_1(u,C_{y_n,\la_n}v)\ra 0$ for all $u,v \in \wps$.
\end{enumerate}
Moreover, {\rm (i)} also implies $\mathcal{A}_1(C_{y_n,\la_n}v,u)\ra 0$, and therefore
\begin{gather}\label{eq:A-vanish}
 \mathcal{A}(u,C_{y_n,\la_n}v)\ra 0 \; \text{and}\quad \mathcal{A}(C_{y_n,\la_n}v,u)\ra 0, \; \forall\, u,v \in \wps.
\end{gather}
\end{proposition}
\begin{proof}
    We use the elementary inequality $\abs{a}^{p-1}\le c_p\left( \abs{a-b}^{p-1}+\abs{b}^{p-1}\right)$ for $a,b \in \R$, together with H\"{o}lder's inequality with the pair $(\p,p)$, which gives us,
    \begin{equation}\label{eq:A1-bound}
        \mathcal{A}_1(u,v)\le [u]_{s,p}^{p-1}[v]_{s,p}, \text{ and } \abs{\mathcal{A}(u,v)}\le \mathcal{A}_1(u,v), \quad \forall\, u,v \in \wps.
    \end{equation}
Now, $\mathcal{A}_1$ is continuous in each of its two arguments, so we assume $u,v \in \cc(\rd)$.

\textbf{(i) $\Longrightarrow$ (ii).} By (\ref{eq:invariance}), $\{C_{y_n,\la_n}v\}$ is bounded in $\wps$. Due to the density of $\cc(\rd)$ in $\wps$, it is enough to show $C_{y_n,\la_n}v \rightharpoonup 0$ in $\wps$ for every $v \in \cc(\rd)$. Further, in view of Lemma \ref{lem:1}-(a), it is required to show $\int_{\rd}C_{y_n,\la_n}v\,\psi \dx \ra 0$, for every $v,\psi \in \cc(\rd)$.  Passing to a subsequence such that $\la_n \ra \la_0 \in [0,\infty]$, we note that $\abs{\log \la_n}+\abs{y_n}\ra \infty$. So we consider the following three cases.

If $\la_0=0$, then the change of variables gives
\begin{equation*}
    \left| \int_{\rd}C_{y_n,\la_n}v\, \psi \dx \right| = \la_n^{d-\vartheta}\left| \int_{\rd}v(z)\psi(\la_nz+y_n)\dz\right| \le \la_n^{d-\vartheta}\norm{\psi}_{L^{\infty}(\rd)}\norm{v}_{L^1(\rd)}\ra 0,
\end{equation*}
  since $d-\vartheta=d-\frac{d}{p}+s>0$. 
  
  If $\la_0=\infty$, then $\norm{C_{y_n,\la_n}v}_{L^{\infty}(\rd)}=\la_n^{-\vartheta}\norm{v}_{L^{\infty}(\rd)}\ra 0$, so the integral tends to $0$. 
  
  If $\la_0 \in (0,\infty)$, then $\abs{y_n}\ra \infty$. Since $\supp (C_{y_n,\la_n}v)\subset B(y_n,\la_n R)$ for a fixed $R$, this support is disjoint from $\supp \psi$ for $n$ large, and the integral is zero.

  \textbf{(ii) $\Longrightarrow$ (i).} Suppose $\La_n \not\ra \infty$, so up to a subsequence $\la_n \ra \la_0 \in (0,\infty)$ and $y_n \ra y_0 \in \rd$. Hence, Lemma \ref{strong-convergence} gives $C_{y_n,\la_n}v \ra C_{y_0,\la_0}v$ in $\wps$, and $C_{y_0,\la_0}v \ne 0$ when  $v \ne 0$. This contradicts (ii).

  \textbf{(i) $\Longrightarrow$ (iii).} Let $u,v \in \cc(\rd)$ with $\supp (u) \subset B(0,R_1)$ and $\supp (v) \subset B(0,R_2)$. We set $D_n \coloneqq B(y_n,\la_nR_2)\supset \supp (C_{y_n,\la_n}v)$. The integrand of $\mathcal{A}_1(u,C_{y_n,\la_n}v)$ is symmetric in $(x,y)$ and vanishes unless $x$ or $y$ lies in $B(0,R_1)$, and also unless $x$ or $y$ lies in $D_n$. Hence, for $E$ equal to $B(0,R_1)$ or to $D_n$, 
  \begin{equation}\label{eq:A1-split}
        \mathcal{A}_1\left(u,C_{y_n,\la_n}v\right)\le 2 \iint_{E \times \rd}\frac{\abs{u(x)-u(y)}^{p-1}\abs{C_{y_n,\la_n}v(x)-C_{y_n,\la_n}v(y)}}{\abs{x-y}^{d+sp}}\dxy.
  \end{equation}
We introduce the two finite Borel measures on $\rd$:
\begin{equation*}
     \nu_u(E):=\iint_{E \times \rd}\frac{\abs{u(x)-u(y)}^p}{\abs{x-y}^{d+sp}}\dxy, \text{ and } \nu_n(E):=\iint_{E \times \rd}\frac{\abs{C_{y_n,\la_n}v(x)-C_{y_n,\la_n}v(y)}^p}{\abs{x-y}^{d+sp}}\dxy,
\end{equation*}
  of total mass $[u]_{s,p}^p$ and $[v]_{s,p}^p$ respectively. Both are absolutely continuous with respect to the Lebesgue measure. A change of variables shows
\begin{equation}\label{eq:scaling}
       \nu_n(E)=\nu_v\left( \frac{E-y_n}{\la_n}\right), \quad \text{for all}\, E \subset \rd \text{ Borel}.
\end{equation}
Applying H\"{o}lder's inequality with the pair $(\p,p)$ in \eqref{eq:A1-split}, we get
\begin{equation}\label{eq:A1-holder}
     \mathcal{A}_1\left(u,C_{y_n,\la_n}v\right)\le 2\,\nu_u(E)^{\frac{1}{\p}}\nu_n(E)^{\frac{1}{p}}, \quad \text{ where } E \in \{B(0,R_1),\,D_n\}.
\end{equation}
We now pass again to a subsequence with $\la_n \ra \la_0 \in [0,\infty]$. If $\la_0=0$, we choose $E=D_n$ in \eqref{eq:A1-holder}, then $\abs{D_n}=\abs{B(0,R_2)}\la_n^d \ra 0$, so $\nu_u(D_n)\ra 0$, while $\nu_n(D_n)\le [v]_{s,p}^p$. If $\la_0=\infty$, choose $E=B(0,R_1)$, by \eqref{eq:scaling}, $\nu_n(B(0,R_1))=\nu_v\left(B(-y_n/\la_n,\,R_1/\la_n)\right)\ra 0$ since the radius $R_1/\la_n$ tends to $0$, while $\nu_u(B(0,R_1))\le [u]_{s,p}^p$. If $\la_0 \in (0,\infty)$, then $\abs{y_n}\ra \infty$. By choosing $E=D_n$, and note that $\abs{D_n}$ stays bounded while $D_n$ escapes every ball, so $\nu_u(D_n)\ra 0$. In the three cases, the right-hand side of \eqref{eq:A1-holder} tends to zero.  Therefore, $\mathcal{A}_1\left(u,C_{y_n,\la_n}v\right) \ra 0$ for every $u,v \in \cc(\rd)$, and using the density of $\cc(\rd)$ in $\wps$,  $\mathcal{A}_1\left(u,C_{y_n,\la_n}v\right) \ra 0$ for every $u,v \in \wps$.

\textbf{(iii) $\Longrightarrow$ (i).} Suppose $\La_n \not\ra \infty$, and we consider $\la_n \ra \la_0 \in (0,\infty)$, $y_n \ra y_0$ up a subsequence. We fix $v \ne 0$ and put $u:=C_{y_0,\la_0}v$. By the continuity of $\mathcal{A}_1$ which follows from \eqref{eq:A1-bound} we get, $\mathcal{A}_1(u,C_{y_n,\la_n}v)\ra \mathcal{A}_1(u,u)=[v]_{s,p}^p>0$, which contradicts~(iii).

Finally, because the integrand of $\mathcal{A}(C_{y_n,\la_n}v,u)$ is again symmetric in $(x,y)$ and shares the support of $\mathcal{A}_1$, applying \eqref{eq:A1-split} and H\"{o}lder's inequality with the pair $(p,\p)$, we get $\mathcal{A}(C_{y_n,\la_n}v,u)\le 2\nu_n(E)^{\frac{1}{\p}}\nu_u(E)^{\frac{1}{p}}$. So \eqref{eq:A-vanish} follows from the second inequality in \eqref{eq:A1-bound}.
\end{proof}

As an application of the above proposition and \eqref{eq:A-behviour}, we have the following result, which extends \cite[Proposition 2.8, 2.11, and Lemma 2.12]{NS2025}.

\begin{lemma}\label{weak-bub-converse-III}
For any sequence $\{(a_n,\delta_n)\},\,\{(y_n,\lambda_n)\}\subset \rd \times (0, \infty)$, we have
\begin{align*}
&\left| \log\left(\frac{\delta_n}{\lambda_n}\right) \right| + \left| \frac{a_n-y_n}{\lambda_n} \right| \to\infty \Longleftrightarrow \mathcal{A}_1(C_{a_n,\delta_n}u,C_{y_n,\lambda_n}v)\to 0, \text{ and } \\
&\left| \log\left(\frac{\delta_n}{\lambda_n}\right) \right| + \left| \frac{a_n-y_n}{\lambda_n} \right| \to\infty \Longleftrightarrow \mathcal{A}(C_{a_n,\delta_n}u,C_{y_n,\lambda_n}v)\to 0,
\end{align*}
as $n \ra \infty$ and for all $u,\,v\in\wps$,  
\end{lemma}

For $\sigma \in \rd$, we consider the following functional on $\wps$:
 \begin{gather}\label{eq:shifted-functional}
   \mathcal{I}_{\sigma}(u):=\frac1p [u]_{s,p}^p-\frac{\mu}{p}\int_{\rd}\frac{\abs{u}^p}{\abs{z+\sigma}^{sp}}\dz-\frac{1}{p^*_s}\int_{\rd}\abs{u}^{p^*_s}\dz, \; \forall \, u \in \wps.
 \end{gather}
By Lemma \ref{lem:far}-(a) and \eqref{HS1}, $\mathcal{I}_{\sigma}$ is well defined and of class $\mathcal{C}^1$.

\begin{lemma}\label{lem:far}
    Let $\{\sigma_n\}\subset \rd$. Then the following hold:
\begin{enumerate}
    \item[\rm (a)] For every $\zeta \in \rd$ and every $v \in \wps$,
    \begin{gather}\label{eq:shifted-Hardy}
     \mu_{d,s,p}\int_{\rd}\frac{\abs{v(z)}^p}{\abs{z-\zeta}^{sp}}\dz \le [v]_{s,p}^p.
    \end{gather}
    \item[\rm (b)] If $\abs{\sigma_n}\ra \infty$, then for every $v \in \wps$,
    \begin{gather}\label{eq:far-vanish}
 \lim_{n \ra \infty}\int_{\rd}\frac{\abs{v(z)}^p}{\abs{z+\sigma_n}^{sp}}\dz = 0.
    \end{gather}
    \end{enumerate}
\end{lemma}
\begin{proof}
(a) Set $\tilde{v}_{\zeta}(w):=v(w+\zeta)$. Then $$\int_{\rd}\frac{\abs{v(z)}^p}{\abs{z-\zeta}^{sp}}\dz=\int_{\rd}\frac{\abs{\tilde{v}_{\zeta}(w)}^p}{\abs{w}^{sp}}\,{\rm d}w, \text{ and } [\tilde{v}_{\zeta}]_{s,p}=[v]_{s,p},$$
since the Gagliardo seminorm is invariant under translations. Now we apply $p$-fractional Hardy's inequality \eqref{HS} to $\tilde{v}_{\zeta}$ to get the conclusion.

(b) Let $\var>0$ and we choose $\phi \in \cc(\rd)$ say, $\supp \phi \subset B(0,R)$, with $[v-\phi]_{s,p}^p<\var$. For $n$ large with $\abs{\sigma_n}\ge 2R$, we have $\abs{z+\sigma_n}\ge \abs{\sigma_n}/2$ on $\supp \phi$, which implies
\begin{equation*}
     \int_{\rd}\frac{\abs{\phi(z)}^p}{\abs{z+\sigma_n}^{sp}}\dz \le \frac{2^{sp}}{\abs{\sigma_n}^{sp}}\norm{\phi}_{L^p(\rd)}^p \ra 0.
\end{equation*}
Now we use (a) with $\zeta=-\sigma_n$, which gives a bound that does not depend on $n$,
\begin{equation*}
    \int_{\rd}\frac{\abs{v(z)-\phi(z)}^p}{\abs{z+\sigma_n}^{sp}}\dz \le \frac{[v-\phi]^p_{s,p}}{\mu_{d,s,p}}<\frac{\var}{\mu_{d,s,p}}.
\end{equation*}
Since $\abs{a+b}^p \le 2^{p-1}(\abs{a}^p+\abs{b}^p)$, we get $$\limsup_{n \ra \infty} \int_{\rd}\frac{\abs{v(z)}^p}{\abs{z+\sigma_n}^{sp}}\dz \le \frac{2^{p-1}\var}{\mu_{d,s,p}}.$$ Letting $\var \ra 0$, we obtain \eqref{eq:far-vanish}.
\end{proof}

\begin{lemma}\label{lem:dual}
Let $\mu \in (0,\mu_{d,s,p})$. Then the following hold:
\begin{enumerate}
    \item[\rm (a)] For all $\la>0$ and $y \in \rd$, writing $\sigma:=\frac{y}{\la}$,
    \begin{equation}\label{eq:isometry}
          \norm{I_{\mu,0,0}'(u)}_{(\wps)^*}=\left\|\mathcal{I}_{\sigma}'\left( \left( C_{y,\la}\right)^{-1}u\right)\right\|_{(\wps)^*} \; \text{for all }\, u \in \wps.
    \end{equation}
    \item[\rm (b)] Let $\{\sigma_n\}\subset \rd$ satisfy either $\sigma_n=0$ for every $n$, or $\abs{\sigma_n}\ra \infty$. Let $v_n \rightharpoonup v$ in $\wps$ with $v_n(x) \ra v(x)$ for a.e. $x \in \rd$. Then, up to a subsequence, \begin{equation}\label{eq:dual}
        \norm{\mathcal{I}_{\sigma_n}'\left( v_n-v\right)-\mathcal{I}_{\sigma_n}'\left( v_n \right)+\mathcal{I}_{\sigma_n}'\left( v \right)}_{(\wps)^*}\ra 0.
    \end{equation}
\end{enumerate}
 \end{lemma}
\begin{proof}
  (a) We write $\hat{u}\coloneqq C_{y,\la}^{-1}u$ and $\hat{\phi}\coloneqq C_{y,\la}^{-1}\phi$. From the change of variables, 
\begin{align*}
    \prescript{}{(\wps)^*}{\langle} \mathcal{I}_{\sigma}'(\hat{u}),\hat{\phi}{\rangle}_{\wps} & = \mathcal{A}(\hat{u}, \hat{\phi}) - \mu \int_{\rd} \frac{|\hat{u}(z)|^{p-2} \hat{u}(z) \hat{\phi}(z)}{|z+\sigma|^{sp}} \dz - \int_{\rd} |\hat{u}(z)|^{p^*_s-2} \hat{u}(z) \hat{\phi}(z) \dz \\
    & = \mathcal{A}(u, \phi) - \mu \int_{\rd} \frac{|u(x)|^{p-2} u(x) \phi(x)}{|x|^{sp}} \dx - \int_{\rd} |u(x)|^{p^*_s-2} u(x) \phi(x) \dx  \\
    &= \prescript{}{(\wps)^*}{\langle} I_{\mu,0,0}'(u),\phi{\rangle}_{\wps},  \; \forall\, \phi \in \wps.
\end{align*}
Further, since each $C_{y,\la}$ is an isometry, the map $\phi \mapsto \hat{\phi}$ is a bijection of $\wps$ onto itself with $[\hat{\phi}]_{s,p}=[\phi]_{s,p}$ (by \eqref{eq:invariance}). Hence, \eqref{eq:isometry} holds.

(b) We write $V_n:=v_n-v$ and $J_q(t):=\abs{t}^{q-2}t$. For $\psi \in \wps$, the quantity inside the norm in \eqref{eq:dual} tested at $\psi$ is $\mathcal{T}_1^n-\mu \mathcal{T}_2^n-\mathcal{T}_3^n$, where
\begin{align*}
     &\mathcal{T}_1^n:=\mathcal{A}(V_n,\psi)-\mathcal{A}(v_n,\psi)+\mathcal{A}(v,\psi), \quad
     \mathcal{T}_2^n:=\int_{\rd}\frac{J_p(V_n)-J_p(v_n)+J_p(v)}{\abs{z+\sigma_n}^{sp}}\psi(z) \dz, \\
     &\mathcal{T}_3^n:=\int_{\rd}\left( J_{p^*_s}(V_n)-J_{p^*_s}(v_n)+J_{p^*_s}(v)\right)\psi(z) \dz.
\end{align*}
For $\mathcal{T}_1^n$, we write $Du(x,y):=u(x)-u(y)$ and apply H\"{o}lder's inequality with the pair $(\p,p)$ on $\R^{2d}$:
\begin{align*}
    \abs{\mathcal{T}_1^n}\le \left\|\frac{J_p(Dv_n)-J_p(DV_n)-J_p(Dv)}{\abs{x-y}^{\frac{d+sp}{\p}}}\right\|_{L^{\p}(\R^{2d})}\,[\psi]_{s,p},
\end{align*}
and the first factor is $o_n(1)$, by Lemma \ref{convergence-BL}-$(iii)$. 

For $\mathcal{T}_3^n$, H\"{o}lder's inequality with the pair $\left( (p^*_s)',p^*_s\right)$ and then \eqref{eq:constants} give
\begin{align*}
    \abs{\mathcal{T}_3^n}\le \left\| J_{p^*_s}(v_n)-J_{p^*_s}(V_n)-J_{p^*_s}(v)\right\|_{L^{(p^*_s)'}(\rd)}\,\overline{S}^{-\frac1p}[\psi]_{s,p},
\end{align*}
and the first factor is $o_n(1)$, by Lemma \ref{convergence-BL}-$(iv)$ with $\tilde{\al}=0$. 

For $\mathcal{T}_2^n$, H\"{o}lder's inequality with the pair $(\p,p)$, followed by Lemma \ref{lem:far}-(a) applied to $\psi$, gives
\begin{equation}\label{eq:T2}
     \abs{\mathcal{T}_2^n}\le \left( \int_{\rd}\frac{\abs{J_p(V_n)-J_p(v_n)+J_p(v)}^{\p}}{\abs{z+\sigma_n}^{sp}}\dz\right)^{\frac{1}{\p}}\mu_{d,s,p}^{-\frac1p}\,[\psi]_{s,p},
\end{equation}
and it remains to see that the first factor of \eqref{eq:T2} is $o_n(1)$ in both the cases. If $\sigma_n=0$ for every $n$, that factor is exactly the $L^{\p}(\rd)$ norm appearing in Lemma \ref{convergence-BL}-$(iv)$ with $\tilde{\al}=sp$, because $\frac{sp}{(p^*_s(sp))'}\p=sp$. So it is $o_n(1)$. 

If $\abs{\sigma_n}\ra \infty$, we argue differently, because the singular weight translates towards infinity and Lemma \ref{convergence-BL} is not applicable in this case. 
% We use the elementary inequality: for every $\var>0$, there exists $C_{\var}>0$ such that
% \begin{equation}\label{eq:elem-dual}
%     \abs{J_p(a-b)-J_p(a)}^{\p}\le \var \abs{a}^p+C_{\var}\abs{b}^p, \quad \text{for all }\, a,b \in \R.
% \end{equation}
We use the elementary inequality: for all $\var >0$ there exists $C(\var)>0$ such that for all $1< p<\infty$ (see \cite[Page~182]{Ambrosetti-Malchiodi}):
\begin{equation}\label{eq:elem-dual}
    \left|J_p(a+b)-J_p(a)-J_p(b)\right|^{p'} \leq \var |a|^p + C_{\var} |b|^p, \text{ for all }a,\,b\in\R. 
\end{equation}
Applying \eqref{eq:elem-dual} with $a=v_n(z)$, $b=v(z)$, together with $\abs{J_p(v)}^{\p}=\abs{v}^p$ we get,
\begin{equation*}
      \int_{\rd}\frac{\abs{J_p(V_n)-J_p(v_n)+J_p(v)}^{\p}}{\abs{z+\sigma_n}^{sp}}\dz \le 2^{\p-1}\left( \var \int_{\rd}\frac{\abs{v_n}^p}{\abs{z+\sigma_n}^{sp}}\dz+ C_{\var}\int_{\rd}\frac{\abs{v}^p}{\abs{z+\sigma_n}^{sp}}\dz\right).
\end{equation*}
By Lemma \ref{lem:far}-(a) the first integral is at most $\mu_{d,s,p}^{-1}\sup_n [v_n]_{s,p}^p<\infty$, and by Lemma \ref{lem:far}-(b) applied to the fixed function $v$, the second integral tends to $0$. Letting $\var \ra 0$ completes the proof.
\end{proof}

Next, we consider 
\begin{align}\label{H-function}
   \displaystyle \mathcal{H}(u,v) \coloneqq \int_{\rd}\frac{\abs{u(x)}^{p-1}\abs{v(x)}}{\abs{x}^{sp}}\dx, \; \text{for all }\, u,v \in \wps.  
\end{align}

\begin{lemma}\label{lem:3}
The following holds:
\begin{enumerate}
    \item[\rm (a)] It holds
    \begin{align*}
        \abs{\log \la_n}\ra \infty \Longrightarrow \mathcal{H}(u,C_{0,\la_n}v)\ra 0, \text{ and } \, \mathcal{H}(C_{0,\la_n}v,u)\ra 0,
    \end{align*}
    for every $u,v \in \wps$. \vspace{0.1 cm}
    \item[\rm (b)] It holds 
    \begin{align*}
        \left|\log \left(\frac{\tau_n}{\la_n}\right)\right|\ra \infty \Longrightarrow \mathcal{H}(C_{0,\la_n}u,C_{0,\tau_n}v)\ra 0, \text{ and } \, \mathcal{H}(C_{0,\tau_n}v,C_{0,\la_n}u)\ra 0,
    \end{align*}
    for every $u,v \in \wps$.
    \vspace{0.1 cm}
    \item[\rm (c)] Let $\{g_n\}\subset \wps$ satisfies $\sup_n \int_{\rd}\frac{\abs{g_n}^p}{\abs{x}^{sp}}\dx<\infty$. Then it holds 
    \begin{align*}
       \frac{\abs{x_n}}{R_n} \ra \infty \Longrightarrow  \mathcal{H}(g_n,C_{x_n,R_n}w)\ra 0, \text{ and } \, \mathcal{H}(C_{x_n,R_n}w,g_n)\ra 0,
    \end{align*}
    for every $w \in \wps$.
\end{enumerate}
\end{lemma}
\begin{proof}
Set $h_u(x)\coloneqq \abs{u(x)}^p |x|^{-sp}$ for $x \in \rd$. By \eqref{HS}, $h_u \in L^1(\rd)$, and by  \eqref{eq:scaling}, $\norm{h_{C_{0,\la}v}}_{L^1(\rd)}=\norm{h_v}_{L^1(\rd)}$ for every $\la>0$. Applying  H\"{o}lder's inequality with the pair $(\p,p)$ gives
\begin{equation}\label{eq:H-holder}
    \int_{E}\frac{\abs{u(x)}^{p-1}\abs{v(x)}}{\abs{x}^{sp}}\dx \le \left( \int_E h_u \dx \right)^{\frac{1}{\p}}\left( \int_E h_v \dx \right)^{\frac{1}{p}},
\end{equation}
for every Borel set $E \subset \rd$. It is sufficient to prove the first limit by \eqref{eq:H-holder}. We now pass to a subsequence with $\la_n \ra 0$ or $\la_n \ra \infty$.

Let $\la_n \ra 0$ and let $\var>0$. Since $h_u \in L^1(\rd)$, there exists $\de>0$ small such that  $\int_{B(0,\de)}h_u<\var$. Consider $E=\rd$ on the LHS of \eqref{eq:H-holder}, then split the integral on $\rd=B(0,\de)\cup B(0,\de)^c$, then using \eqref{eq:H-holder} we get,
\begin{equation*}
    \mathcal{H}(u,C_{0,\la_n}v)\le \var^{\frac{1}{\p}}\norm{h_v}_{L^1(\rd)}^{\frac{1}{p}}+\norm{h_u}_{L^1(\rd)}^{\frac{1}{\p}}\left( \int_{B(0,\de)^c}h_{C_{0,\la_n}v} \dx\right)^{\frac{1}{p}}.
\end{equation*}
The change of variables gives
$$\int_{B(0,\de)^c}h_{C_{0,\la_n}v}=\int_{B(0, \frac{\de}{\la_n})^c}h_v \ra 0, \text{ since } \frac{\de}{\la_n} \ra \infty, \text{ and } h_v \in L^1(\rd).$$
Now since $\var>0$ is arbitrary, $\limsup_{n \ra \infty} \mathcal{H}(u,C_{0,\la_n}v) \le \var^{\frac{1}{\p}}\norm{h_v}_{L^1(\rd)}^{\frac1p} \ra 0.$ 

Let $\la_n \ra \infty$. We choose $\La>0$ with $\int_{B(0,\La)^c}h_u<\var$, and use $$\int_{B(0,\La)}h_{C_{0,\la_n}v}=\int_{B(0,\frac{\La}{\la_n})}h_v \ra 0,$$ to get $\mathcal{H}(u,C_{0,\la_n}v) \ra 0$ as $n \ra \infty$. 

(b) A change of variables shows $\mathcal{H}(C_{0,\la}u,C_{0,\la}v)=\mathcal{H}(u,v)$ for every $\la>0$.
%indeed the powers of $r$ add up to $-\vartheta(p-1)-\vartheta-sp+d=0$.
Hence $\mathcal{H}(C_{0,\la_n}u,C_{0,\tau_n}v)=\mathcal{H}(u,C_{0,\frac{\tau_n}{\la_n}}v)$, and (b) holds applying (a).

(c) By \eqref{eq:H-holder} with $E=\rd$, we have
\begin{equation}\label{eq:holder-II}
     \mathcal{H}\left(g_n,C_{x_n,R_n}w\right)\le \left( \int_{\rd}\frac{\abs{g_n}^p}{\abs{x}^{sp}}\dx\right)^{\frac{1}{\p}}\left( \int_{\rd}\frac{\abs{C_{x_n,R_n}w}^p}{\abs{x}^{sp}}\dx\right)^{\frac{1}{p}},
\end{equation}
and using \eqref{eq:two-integrals} with $\tilde{\al}=sp$, the second factor equals $\left( \int_{\rd}\frac{\abs{w(z)}^p}{\abs{z+\sigma_n}^{sp}}\dz\right)^{1/p}$ with $\sigma_n=\frac{x_n}{R_n}$, and hence using Lemma \ref{lem:far}-(b),
\begin{align*}
    \lim_{|\sigma_n| \ra \infty} \int_{\rd}\frac{\abs{w(z)}^p}{\abs{z+\sigma_n}^{sp}}\dz =0. 
\end{align*}
The first factor in \eqref{eq:holder-II} is bounded by hypothesis. The other limit $\mathcal{H}(C_{x_n,R_n}w,g_n)\ra 0$ holds similarly by interchanging the exponents.
\end{proof}

\subsection{Proof of Theorem \ref{PS-decomposition-II}} We now combine the preceding translation-dilation estimates with a moving centre function. The relative parameter $\frac{y_n}{r_n}$ will distinguish whether the Hardy singularity remains visible in the rescaled limit or disappears at infinity.

The bound \eqref{PSD-1} of the proof of Theorem \ref{PS-decomposition} holds with $\al=0$ unchanged, so $\{u_n\}$ is bounded in $\wps$ and, up to a subsequence, $u_n \rightharpoonup \tilde{u}$ in $\wps$ and $u_n \ra \tilde{u}$ a.e. in $\rd$, and Lemma \ref{convergence-integrals} shows that $\tilde{u}$ weakly solves \eqref{mainalphazero}. We set $\tilde{u}_n \coloneqq u_n-\tilde{u}$, so that $\tilde{u}_n \rightharpoonup 0$ in $\wps$. We claim that 
\begin{equation}\label{eq:step-1}
    \{\tilde{u}_n\} \text{ is a (PS) sequence for } I_{\mu,0,0} \text{ at level } \eta-I_{\mu,0,f}(\tilde{u}).
\end{equation}
We use Step 1 of the proof of Theorem \ref{PS-decomposition} with $\al=0$, and for the Hardy part we use Lemma \ref{convergence-BL}-(ii) with $\tilde{\al}=sp$, which is allowed because $p^*_s(sp)=p$. For the derivative, we have to produce convergence in the norm of $(\wps)^*$, and we get it from Lemma \ref{lem:dual}-(b) with $\sigma_n=0$, $v_n=u_n$ and $v=\tilde{u}$,  since $\mathcal{I}_{0}=I_{\mu,0,0}$,
\begin{align}\label{global-I}
    I_{\mu,0,0}'\left( \tilde{u}_n\right)=\left[ I_{\mu,0,0}'\left( u_n-\tilde{u}\right)-I_{\mu,0,0}'\left( u_n\right)+I_{\mu,0,0}'(\tilde{u})\right]+I_{\mu,0,0}'\left( u_n\right)-I_{\mu,0,0}'(\tilde{u}),
\end{align}
where the bracket tends to $0$ in $(\wps)^*$ by  Lemma \ref{lem:dual}-(b). For the last two terms of \eqref{global-I}, we observe that 
$$I_{\mu,0,0}(w)=I_{\mu,0,f}(w)+\prescript{}{(\wps)^*}{\langle}f,w{\rangle}_{\wps}$$ gives $I_{\mu,0,0}'(w)=I_{\mu,0,f}'(w)+f$ for every $w \in \wps$, and hence
\begin{equation*}
      I_{\mu,0,0}'\left( u_n\right)-I_{\mu,0,0}'(\tilde{u})=I_{\mu,0,f}'\left( u_n\right)-I_{\mu,0,f}'(\tilde{u})=I_{\mu,0,f}'\left( u_n\right)= o_n(1), \; \text{ in } (\wps)^*,
\end{equation*}
as $I_{\mu,0,f}'(\tilde{u})=0$ and $\{u_n\}$ is a (PS) sequence for $I_{\mu,0,f}$. This proves \eqref{eq:step-1}.

We now work with a (PS) sequence $\{\tilde{u}_n\}$ of $I_{\mu,0,0}$ at some level $c$ with $\tilde{u}_n \rightharpoonup 0$ in $\wps$. If $\tilde{u}_n \ra 0$ in $\wps$ we stop. So assume $\tilde{u}_n \not\ra 0$. Similarly, as in \eqref{del-0}, using $\prescript{}{(\wps)^*}{\langle} I_{\mu,0,0}'(\tilde{u}_n),\tilde{u}_n{\rangle}_{\wps}\ra 0$ and \eqref{equivalent-norm},
\begin{equation}\label{eq:non-vanishing}
     0<\de_1 \le \inf_{n \in \N}\int_{\rd}\abs{\tilde{u}_n}^{p^*_s}\dx.
\end{equation}
We split our proof into several cases.

\noi \textbf{Step 1:} We let the centre of the concentration ball move and take the largest mass over all centres. Let $N=N(d)\in \N$ be a number of balls of radius one that cover $B(0,2)$, we put
\begin{equation}\label{eq:theta-mu}
    \theta_{\mu} \coloneqq \frac{\mu}{\mu_{d,s,p}}\in (0,1), \quad \var_{\mu} \coloneqq \frac{1-\theta_{\mu}}{2\theta_{\mu}}>0,
\end{equation}
and we set 
\begin{equation}\label{eq:delta-star}
    \de^*\coloneqq \frac{1}{N} \left( \frac{\left(1-\theta_{\mu}\right)\overline{S}}{ 4\left(1+\var_\mu \right)}\right)^{\frac{d}{sp}}>0,
\end{equation}
%=\frac{1}{N} \left( \frac{\theta_{\mu} \left( 1-\theta_{\mu}\right)\overline{S}}{2 \left( 1+\theta_\mu\right)} \right)^{\frac{d}{sp}}%
where $\overline{S}$ is the constant in \eqref{eq:constants}. The number $\de^*$ depends only on $d$, $s$, $p$ and $\mu$ (does not depend on the sequence $\tilde{u}_n$). We fix
\begin{gather}\label{eq:delta-choice}
0< \de < \min \left\{ \de_1,\,\de_*\right\},
\end{gather}
and set 
\begin{equation}\label{eq:levy}
Q_n(r) \coloneqq \sup_{y \in \rd}\int_{B(y,r)}\abs{\tilde{u}_n}^{p^*_s} \dx, \text{ for } r>0.
\end{equation}
Applying \cite[Lemma 3.1]{Brasco-2018}, $Q_n$ is continuous on $\R^+$. The supremum in \eqref{eq:levy} is attained, since $y \mapsto \int_{B(y,r)}\abs{\tilde{u}_n}^{p^*_s}$ is continuous and tends to $0$, when $\abs{y}$ is sufficiently large. 
%Next, $Q_n(r)\ra 0$ as $r \ra 0^+$, since given $\var>0$ pick $h \in \mathcal{C}_c(\rd)$ with $\norm{\abs{\tilde{u}_n}^{p^*_s}-h}_{L^1(\rd)}<\var$, and then $\int_{B(y,r)}\abs{\tilde{u}_n}^{p^*_s}\le \var+\norm{h}_{L^{\infty}(\rd)}\abs{B(0,r)}$ for every $y$. This shows that $Q_n$ is continuous on $\R^+$; 
Further, $Q_n$ is nondecreasing, and $Q_n(r)\ra \int_{\rd}\abs{\tilde{u}_n}^{p^*_s}\ge \de_1>\de$ when  $r$ is sufficiently large.
So, we can find $r_n \in (0, \infty)$ and $y_n \in \rd$ with
\begin{equation}\label{eq:choice}
    \int_{B(y_n, r_n)}\abs{\tilde{u}_n}^{p^*_s}\dx=\sup_{y \in \rd}\int_{B(y,r_n)}\abs{\tilde{u}_n}^{p^*_s}\dx=\de.
\end{equation}
We define the following functions
\begin{align*}
    \hat{u}_n(z) = r_n^{\vartheta}\tilde{u}_n(r_nz + y_n), \text{ and } \tilde{\phi}_n(z)= C_{y_n,r_n} \phi(z), \text{ for } z \in \rd,
\end{align*}
where $\phi \in \wps$. In view of \eqref{eq:choice}, 
\begin{align}\label{eq:choice-II}
    \int_{B(0,1)} \abs{\hat{u}_n}^{p^*_s} \dx = \delta. 
\end{align}
Using the norm invariance \eqref{eq:invariance}, we observe that $\prescript{}{(\wps)^*}{\langle} I_{\mu,0,0}'(\tilde{u}_n),\tilde{\phi}_n{\rangle}_{\wps}=o_n(1)$. Hence, using the change of variables, we obtain
\begin{equation}\label{eq:rescaled-PS}
 \mathcal{A}(\hat{u}_n,\phi)-\mu \int_{\rd}\frac{\abs{\hat{u}_n}^{p-2}\hat{u}_n}{\abs{z+\sigma_n}^{sp}}\phi \dz - \int_{\rd}\abs{\hat{u}_n}^{p^*_s-2}\hat{u}_n \phi \dz = o_n(1), \quad \forall\, \phi \in \wps,
\end{equation}
where \begin{gather}\label{eq:sigma}
\sigma_n\coloneqq\frac{y_n}{r_n}.
\end{gather}
We note that \eqref{eq:rescaled-PS} contains the origin, and it is only through the shift $\sigma_n$.
Along a subsequence, one of the following holds:
\begin{align}
\underline{\text{\textbf{Type I}:}} \quad \sigma_n \ra \sigma \in \rd, \; \text{ and } \; \underline{\text{\textbf{Type II}:}} \quad \abs{\sigma_n}\ra \infty.
\end{align}
Indeed, either $\{\sigma_n\}$ has a bounded subsequence, and then a convergent one, or $\abs{\sigma_n}\ra \infty$.
Next, we define
\begin{equation}\label{eq:bar}
     \overline{u}_n(z)\coloneqq\left(C_{0,r_n}\right)^{-1}\tilde{u}_n(z)= C_{0,\frac{1}{r_n}}\tilde{u}_n(z)= r_n^{\vartheta}\tilde{u}_n(r_nz)=\hat{u}_n(z-\sigma_n), \text{ for } z \in \rd.
\end{equation}
Applying Lemma \ref{lem:1}-(b),  with $\tau_n=1$ and $\xi_n=\sigma_n \ra \sigma$,
\begin{equation}\label{eq:barlimit}
    \overline{u}_n \rightharpoonup \overline{u}\coloneqq\hat{u}(\cdot-\sigma) \text{ in } \wps, \; \text{ that is, } \; \overline{u}\ne 0 \Longleftrightarrow \hat{u}\ne 0.
\end{equation}

\noi \textbf{Step 2:} In this step, we assume \textbf{Type I} holds and show $\hat{u}\ne 0$. Then, by \eqref{eq:barlimit}, we get $\overline{u}\ne 0$. On the contrary, suppose $\hat{u}=0$. Then, by the compact embedding $\wps \embd L^p_{loc}(\rd)$,
\begin{equation}\label{eq:local}
     \hat{u}_n \ra 0 \; \text{ in } L^p_{loc}(\rd), \text{ and } M\coloneqq\sup_{n \in \N}[\hat{u}_n]_{s,p}<\infty.
\end{equation}
We choose $\eta \in \cc(\rd)$ with $\supp(\eta) \subset B(0,2)$, $0 \le \eta \le 1$ and $\eta \equiv 1$ on $B(0,1)$, and put $\phi\coloneqq\eta^p$. Since $0 \le \eta \le 1$ and $p>1$, the function $\phi$ is Lipschitz with support in $B(0,2)$, and we have
\begin{equation*}
     \phi_n:=C_{y_n,r_n}\left( \phi \hat{u}_n\right), \text{ so that } [\phi_n]_{s,p}=[\phi \hat{u}_n]_{s,p}\le C(\eta)\left( M+\sup_n \norm{\hat{u}_n}_{L^{p^*_s}(\rd)}\right),
\end{equation*}
i.e., $\{ \phi_n \}$ is bounded. Using $\prescript{}{(\wps)^*}{\langle} I_{\mu,0,0}'(\tilde{u}_n),\phi_n{\rangle}_{\wps}=o_n(1)$ and the change of variables as in \eqref{eq:rescaled-PS}, we get
\begin{equation}\label{eq:test}
     \mathcal{A}\left( \hat{u}_n,\phi \hat{u}_n\right)=\mu \int_{\rd}\frac{\abs{\eta \hat{u}_n}^p}{\abs{z+\sigma_n}^{sp}}\dz+\int_{\rd}\eta^p \abs{\hat{u}_n}^{p^*_s}\dz+o_n(1),
\end{equation}
where we have used $\phi \abs{\hat{u}_n}^{p}=\abs{\eta \hat{u}_n}^p$. For $x,y \in \rd$, from the splitting
\begin{equation*}
     \phi(x)\hat{u}_n(x)-\phi(y)\hat{u}_n(y)=\phi(x)\left( \hat{u}_n(x)-\hat{u}_n(y)\right)+\hat{u}_n(y)\left( \phi(x)-\phi(y)\right),
\end{equation*}
we put
\begin{equation}\label{eq:En}
     E_n:=\iint_{\rd \times \rd}\eta(x)^p\,\frac{\abs{\hat{u}_n(x)-\hat{u}_n(y)}^p}{\abs{x-y}^{d+sp}}\dxy \in [0,M^p],
\end{equation}
to get $\mathcal{A}(\hat{u}_n,\phi \hat{u}_n)=E_n+I_n$ with
\begin{equation*}
      I_n:=\iint_{\rd \times \rd}\frac{\abs{\hat{u}_n(x)-\hat{u}_n(y)}^{p-2}\left( \hat{u}_n(x)-\hat{u}_n(y)\right)\hat{u}_n(y)\left( \phi(x)-\phi(y)\right)}{\abs{x-y}^{d+sp}}\dxy.
\end{equation*}
We claim that 
\begin{equation}\label{eq:I-n}
I_n=o_n(1).
\end{equation}
Set $$K(y) \coloneqq \int_{\rd}\frac{\abs{\phi(x)-\phi(y)}^p}{\abs{x-y}^{d+sp}}\dx, \text{ for } y \in \rd.$$ 
We show that  
\begin{equation}\label{eq:K}
     K \in L^{\infty}(\rd)\cap L^{\frac{d}{sp}}(\rd).
\end{equation}
To prove $K \in L^{\infty}(\rd)$, we split
\begin{align*}
    K(y)=\int_{|x-y|<1}\frac{\abs{\phi(x)-\phi(y)}^p}{\abs{x-y}^{d+sp}}\dx+\int_{|x-y|\geq 1}\frac{\abs{\phi(x)-\phi(y)}^p}{\abs{x-y}^{d+sp}}\dx :=K_1(y)+K_2(y).
\end{align*}
Since $\phi$ is Lipschitz, $|\phi(x)-\phi(y)| \leq \|\nabla \phi\|_{L^{\infty}}|x-y|$. On $|x-y|<1$, we have
\begin{align*}
    K_1(y)\leq C\int_{|z|<1}|z|^{p-d-sp}\dz=C\int_0^1 r^{p-sp-1}\dr<\infty. 
\end{align*}
Now, since $|\phi(x)-\phi(y)| \leq 2\|\phi\|_{L^\infty(\rd)}$,
\begin{equation*}
    K_2(y) \leq C\int_{|z| \geq 1} |z|^{-d-sp}\dz=C\int_1^{\infty}r^{-sp-1}\dr<\infty.
\end{equation*}
So, $K(y) \leq C$ for all $y \in \rd$. Now, it is required to show $K \in L^{\frac{d}{sp}}(\rd)$. If $\abs{y}\ge 4$ and $x \in \supp(\phi)$ which is $B(0,2)$, then we have $\phi(y)=0$ and $\abs{x-y}\ge \frac{\abs{y}}{2}$. Hence $K(y)\le C \abs{y}^{-(d+sp)}$ for $\abs{y}\ge 4$. As $(d+sp)\frac{d}{sp}>d$, we get $K \in L^{\frac{d}{sp}}(B(0,4)^c)$. Also,
\begin{equation*}
    \int_{|y|<4}K(y)^{\frac{d}{sp}}\dy \leq C^{\frac{d}{sp}}|B(0,4)|,
\end{equation*}
implies $K \in L^{\frac{d}{sp}}(B(0,4))$. Thus, \eqref{eq:K} holds. Using H\"{o}lder's inequality with the pair $(\p,p)$ gives $$\abs{I_n}\le M^{p-1}\left( \int_{\rd}\abs{\hat{u}_n}^pK \dz\right)^{\frac1p},$$ where for every $R>0$,
\begin{gather*}
     \int_{\rd}\abs{\hat{u}_n}^pK \dz \le \norm{K}_{L^{\infty}(\rd)}\norm{\hat{u}_n}_{L^p(B(0,R))}^p+\norm{\hat{u}_n}^p_{L^{p^*_s}(\rd)}\left( \int_{B(0,R)^c}K^{\frac{d}{sp}}\dz\right)^{\frac{sp}{d}},
\end{gather*}
using H\"{o}lder's inequality with the pair $\left( \frac{p^*_s}{p},\frac{d}{sp}\right)$. The first term is $o_n(1)$ by (\ref{eq:local}) and the second is small for $R$ large, uniformly in $n$, by \eqref{eq:K}. So, the claim \eqref{eq:I-n} follows.

Next, by observing that  
\begin{equation*}
     \abs{\eta(x)\hat{u}_n(x)-\eta(y)\hat{u}_n(y)}^p \le 2^{p-1}\left( \eta(x)^p\abs{\hat{u}_n(x)-\hat{u}_n(y)}^p+\abs{\hat{u}_n(y)}^p\abs{\eta(x)-\eta(y)}^p\right),
\end{equation*}
we obtain
\begin{equation}\label{eq:cutoff}
    [\eta \hat{u}_n]_{s,p}^p \le 2^{p-1}\left( M^p+\int_{\rd}\abs{\hat{u}_n}^pK_{\eta}\dz \right)<\infty, \text{ where } K_{\eta}(y):=\int_{\rd}\frac{\abs{\eta(x)-\eta(y)}^p}{\abs{x-y}^{d+sp}}\dx,
\end{equation}
using the same bound with $\phi$ in place of $\eta$, which is the bound on $[\phi \hat{u}_n]_{s,p}$. Using instead the sharper elementary inequality $\abs{a+b}^p \le (1+\var)\abs{a}^p+C_{\var}\abs{b}^p$, valid for all $a,b \in \R$ and $\var>0$, we get for every $\var>0$,
\begin{equation}\label{eq:sharp}
     [\eta \hat{u}_n]_{s,p}^p \le (1+\var)E_n+C_{\var}\int_{\rd}\abs{\hat{u}_n}^pK_{\eta}\dz=(1+\var)E_n+o_n(1),
\end{equation}
where the last equality holds again by \eqref{eq:local} and \eqref{eq:K} applied to $K_{\eta}$. Now, we will look at the Hardy term. By Lemma \ref{lem:far}-(a) with $\zeta=-\sigma_n$, and then using \eqref{eq:sharp}, we get,
\begin{equation}\label{eq:hardy}
    \mu \int_{\rd}\frac{\abs{\eta \hat{u}_n}^p}{\abs{z+\sigma_n}^{sp}}\dz \le \frac{\mu}{\mu_{d,s,p}}\,[\eta \hat{u}_n]^p_{s,p}\le \theta_{\mu}(1+\var)E_n+o_n(1).
\end{equation}
We choose $\var:=\var_{\mu}$ as in \eqref{eq:theta-mu} so that $\theta_{\mu}\left( 1+\var_{\mu}\right)=\frac{1+\theta_{\mu}}{2}<1$. For the critical term, since $\eta^p \abs{\hat{u}_n}^{p^*_s}=\abs{\eta \hat{u}_n}^p \abs{\hat{u}_n}^{p^*_s-p}$ and $\supp \eta \subset B(0,2)$, H\"{o}lder's inequality with the pair $\left( \frac{p^*_s}{p^*_s-p},\frac{p^*_s}{p}\right)$ gives
\begin{equation}\label{eq:critical}
      \int_{\rd}\eta^p \abs{\hat{u}_n}^{p^*_s}\dz \le \left( \int_{B(0,2)}\abs{\hat{u}_n}^{p^*_s}\dz\right)^{\frac{p^*_s-p}{p^*_s}}\left( \int_{\rd}\abs{\eta \hat{u}_n}^{p^*_s}\dz\right)^{\frac{p}{p^*_s}}.
\end{equation}
The ball $B(0,2)$ is covered by $N$ balls of radius one. So, using \eqref{eq:choice-II}, we get 
\begin{align*}
    \int_{B(0,2)}\abs{\hat{u}_n}^{p^*_s}\dz \le N \delta.
\end{align*}
For the second factor, using \eqref{eq:sharp}, for every $\var>0$,
\begin{align*}
  \left( \int_{\rd}\abs{\eta \hat{u}_n}^{p^*_s}\dz \right)^{\frac{p}{p^*_s}} \le \frac{[\eta \hat{u}_n]_{s,p}^p}{\overline{S}} \le \frac{(1+\var)E_n+o_n(1)}{\overline{S}}.
\end{align*}
In particular, we take $\var = \var_{\mu}$. Therefore, \eqref{eq:critical} and the choice of $\delta^*$ in \eqref{eq:delta-star} yield
\begin{align*}
    \int_{\rd}\eta^p \abs{\hat{u}_n}^{p^*_s}\dz \le (N \delta)^{\frac{sp}{d}} \frac{(1+\var_{\mu})E_n+o_n(1)}{\overline{S}} \le \frac{1-\theta_{\mu}}{4}E_n + o_n(1).
\end{align*} 
Hence, using \eqref{eq:I-n}, \eqref{eq:hardy}, \eqref{eq:critical} into \eqref{eq:test}, we get
\begin{equation*}
    E_n \le \frac{1+\theta_{\mu}}{2}E_n+\frac{1-\theta_{\mu}}{4}E_n+o_n(1)=\frac{3+\theta_{\mu}}{4}E_n+o_n(1),
\end{equation*}
which implies $$\frac{1-\theta_{\mu}}{4}E_n \le o_n(1).$$ As $\theta_{\mu}<1$ and $E_n \le M^p$ by \eqref{eq:En}, this implies, $E_n=o_n(1)$. Then, \eqref{eq:critical} and $\eta \equiv 1$ on $B(0,1)$ gives
\begin{gather*}
 \int_{B(0,1)}\abs{\hat{u}_n}^{p^*_s}\dz \le \int_{\rd}\eta^p \abs{\hat{u}_n}^{p^*_s}\dz =o_n(1),
\end{gather*}
which contradicts \eqref{eq:choice} where this quantity equals $\de>0$. Therefore, we must have $\hat{u}\ne 0$.

\noi \textbf{Step 3:} In this step, we pass to the limit in \eqref{eq:rescaled-PS}. We consider two cases.
\medskip

\noi \underline{If \textbf{Type I} holds}: Here \eqref{eq:rescaled-PS} is written for $\overline{u}_n$ instead of $\hat{u}_n$, that is, with $\sigma_n$ replaced by $0$, since $\overline{u}_n=(C_{0,r_n})^{-1}\tilde{u}_n$ and by \eqref{eq:hardy-scale}. Passing to the limit with Lemma \ref{convergence-integrals}-(i) (with $\tilde{\al}=sp$ and $\tilde{\al}=0$) and Lemma \ref{convergence-integrals}-(ii) shows (exactly follows from \eqref{invariant-1}-\eqref{limit-1} with $\al=0$) that $\overline{u}$ weakly solves \eqref{limit-mu}. 
% We set
% \begin{equation*}
%        v:=\overline{u},  \text{ with scale } r_n \text{ and centre } 0.
% \end{equation*}
It remains to see $r_n \ra 0$ or $r_n \ra \infty$. Suppose instead $r_n \ra r_0 \in (0,\infty)$. Since $\overline{u}\ne 0$ there is $R>0$ with $\norm{\overline{u}}_{L^p(B(0,R))}>0$, and the compact embedding $\wps \embd L^p_{loc}(\rd)$ gives (as in \eqref{limitofrn-1}):
\begin{align*}
    0<\norm{\overline{u}}_{L^p(B(0,R))}=\norm{\overline{u}_n}_{L^p(B(0,R))}+o_n(1)=r_n^{-s}\norm{\tilde{u}_n}_{L^p(B(0,r_nR))}+o_n(1).
\end{align*}
As $r_n \ra r_0$, there exists $R_1>0$ with $B(0,r_nR)\subset B(0,R_1)$ for every large $n$, and $\tilde{u}_n \rightharpoonup 0$ in $\wps$ gives $\norm{\tilde{u}_n}_{L^p(B(0,R_1))}\ra 0$ by the same compact embedding. This is a contradiction, so $r_n \ra 0$ or $r_n \ra \infty$. For brevity, we denote $ v: = \hat{u}$.

\medskip
\noi \underline{If \textbf{Type II} holds}: We show $\hat{u}$ weakly solves \eqref{limit-zero}. Fix $\phi \in \wps$. By Lemma \ref{convergence-integrals}-(ii), $\mathcal{A}(\hat{u}_n,\phi)\ra \mathcal{A}(\hat{u},\phi)$, and by Lemma \ref{convergence-integrals}-(i) with $\tilde{\al}=0$, $\int_{\rd}\abs{\hat{u}_n}^{p^*_s-2}\hat{u}_n \phi \ra \int_{\rd}\abs{\hat{u}}^{p^*_s-2}\hat{u}\phi$. For the Hardy term, H\"{o}lder's inequality with the pair $(\p,p)$ gives
\begin{equation}\label{eq:hardy1}
     \left| \int_{\rd}\frac{\abs{\hat{u}_n}^{p-2}\hat{u}_n}{\abs{z+\sigma_n}^{sp}}\phi \dz \right| \le \left( \int_{\rd}\frac{\abs{\hat{u}_n}^{p}}{\abs{z+\sigma_n}^{sp}}\dz\right)^{\frac{1}{\p}} \left( \int_{\rd}\frac{\abs{\phi}^{p}}{\abs{z+\sigma_n}^{sp}}\dz\right)^{\frac{1}{p}}.
\end{equation}
By Lemma \ref{lem:far}-(a), the first factor of \eqref{eq:hardy1} is less than $\left( [\hat{u}_n]_{s,p}^p (\mu_{d,s,p})^{-1} \right)^{1/\p}$, which is bounded. By Lemma \ref{lem:far}-(b), applied to the function $\phi$, the second factor in $o_n(1)$. Hence, the Hardy term cancels out in the limit and $\hat{u}$ weakly solves \eqref{limit-zero}. For brevity, we denote $w:=\hat{u}$, $R_n:=r_n$ and $x_n=y_n$.
% We set \begin{equation*}
%       w:=\hat{u}, \quad R_n:=r_n, \quad x_n:=y_n,
% \end{equation*}
% so that $\abs{x_n}/R_n=\abs{\sigma_n}\ra \infty$.

\noi \textbf{Step 4:} Let $B_n=C_{0,r_n}v$ in \textbf{Type I} and $B_n=C_{x_n,R_n}w$ in \textbf{Type II}, and let $\tilde{\tilde{u}}_n:=\tilde{u}_n-B_n$.
We claim that $\{\tilde{\tilde{u}}_n\}$ is again a (PS) sequence for $I_{\mu,0,0}$ with $\tilde{\tilde{u}}_n \rightharpoonup 0$ in $\wps$, and that its level is
\begin{equation}\label{eq:new-level}
     \eta-I_{\mu,0,f}(\tilde{u})-I_{\mu,0,0}(v) \quad \text{in \textbf{Type I}}, \text{ and }  \eta-I_{\mu,0,f}(\tilde{u})-I_{0,0,0}(w) \quad \text{in \textbf{Type II}}.
\end{equation}
Observe that $\tilde{\tilde{u}}_n \rightharpoonup 0$ in $\wps$, due to the fact that $\tilde{u}_n \rightharpoonup 0$ in $\wps$, and $B_n \rightharpoonup 0$ in $\wps$, which follows applying Proposition \ref{lem:2}, since $\abs{\log (r_n)}\ra \infty$ (from \textbf{Type I}) and $\abs{\log(R_n)}+\abs{x_n}\ra \infty$ (from \textbf{Type II}).
In \textbf{Type I}, the rest of the claim is exactly Step 3 of the proof of Theorem \ref{PS-decomposition} with $\al=0$. Now, we consider \textbf{Type II}, where we incorporate the Hardy term. We define
\begin{align*}
    w_n(z):=(C_{x_n,R_n})^{-1}\tilde{\tilde{u}}_n(z)=\hat{u}_n(z)-w(z), \text{ for } z \in \rd,
\end{align*}
where $\hat{u}_n \rightharpoonup w$ in $\wps$ and $\hat{u}_n \ra w$ a.e. in $\rd$. Applying Lemma \ref{convergence-BL}-((i),(ii)) and then using \eqref{eq:invariance},
\begin{equation}\label{eq:BL-two}
     [\tilde{\tilde{u}}_n]_{s,p}^p=[w_n]_{s,p}^p=[\tilde{u}_n]_{s,p}^p-[w]_{s,p}^p+o_n(1), \quad \int_{\rd}\abs{\tilde{\tilde{u}}_n}^{p^*_s}\dx = \int_{\rd}\abs{\tilde{u}_n}^{p^*_s}\dx-\int_{\rd}\abs{w}^{p^*_s}\dx+o_n(1).
\end{equation}
For the Hardy term, we show
\begin{equation}\label{eq:hardy-same}
      \int_{\rd}\frac{\abs{\tilde{\tilde{u}}_n}^p}{\abs{x}^{sp}}\dx = \int_{\rd}\frac{\abs{\tilde{u}_n}^p}{\abs{x}^{sp}}\dx+o_n(1).
\end{equation}
By \eqref{eq:two-integrals} with $\tilde{\al}=sp$, the difference of the two sides of \eqref{eq:hardy-same} equals (after the change of variables):
\begin{equation*}
     \int_{\rd}\frac{\abs{\hat{u}_n-w}^p-\abs{\hat{u}_n}^p}{\abs{z+\sigma_n}^{sp}}\dz.
\end{equation*}
  Given $\var>0$, the elementary inequality $\left| \abs{a-b}^p-\abs{a}^p \right| \le \var \abs{a}^p+C_{\var}\abs{b}^p$ bounds the modulus of this quantity by
\begin{equation*}
     \var \int_{\rd}\frac{\abs{\hat{u}_n}^p}{\abs{z+\sigma_n}^{sp}}\dz+C_{\var}\int_{\rd}\frac{\abs{w}^p}{\abs{z+\sigma_n}^{sp}}\dz \le \frac{\var}{\mu_{d,s,p}}\,[\hat{u}_n]_{s,p}^p+C_{\var}\int_{\rd}\frac{\abs{w}^p}{\abs{z+\sigma_n}^{sp}}\dz,
\end{equation*}
  where we use Lemma \ref{lem:far}-(a). The second term tends to zero by Lemma \ref{lem:far}-(b). Letting $\var \ra 0$ proves \eqref{eq:hardy-same}. So using \eqref{eq:BL-two} and \eqref{eq:hardy-same}, we get
  \begin{gather*}
         I_{\mu,0,0}(\tilde{\tilde{u}}_n)=I_{\mu,0,0}(\tilde{u}_n)-\left( \frac{1}{p}[w]_{s,p}^p-\frac{1}{p^*_s}\int_{\rd}\abs{w}^{p^*_s}\dx\right)+o_n(1)=\eta-I_{\mu,0,f}(\tilde{u})-I_{0,0,0}(w)+o_n(1),
  \end{gather*}
which is \eqref{eq:new-level} in \textbf{Type II}.  We now show $\norm{I_{\mu,0,0}'(\tilde{\tilde{u}}_n)}_{(\wps)^*} = o_n(1)$. We write
\begin{equation*}
     C_n:=C_{0,r_n} \; \text{ and } \; \sigma_n:=0 \; \text{ in \textbf{Type I}}, \quad C_n:=C_{x_n,R_n} \; \text{ and } \; \sigma_n:=\frac{x_n}{R_n} \; \text{ in \textbf{Type II}},
\end{equation*}
so that in both cases $C_n^{-1}\tilde{u}_n=\hat{u}_n \rightharpoonup W$ and $C_n^{-1}\tilde{\tilde{u}}_n=\hat{u}_n-W$, where $W=v$ in \textbf{Type I} and $W=w$ in \textbf{Type II}. By \eqref{eq:isometry} we have,
\begin{equation}\label{eq:dual2}
\norm{I_{\mu,0,0}'(\tilde{\tilde{u}}_n)}_{(\wps)^*}=\norm{\mathcal{I}_{\sigma_n}'\left( \hat{u}_n-W\right)}_{(\wps)^*},
\end{equation}
so it is enough to show that the right-hand side of \eqref{eq:dual} tends to zero. We now split 
\begin{equation}\label{eq:dual-split}
     \mathcal{I}_{\sigma_n}'\left( \hat{u}_n-W\right)=\left[ \mathcal{I}_{\sigma_n}'\left( \hat{u}_n-W\right)-\mathcal{I}_{\sigma_n}'\left( \hat{u}_n\right)+\mathcal{I}_{\sigma_n}'\left( W\right)\right]+\mathcal{I}_{\sigma_n}'\left( \hat{u}_n\right)-\mathcal{I}_{\sigma_n}'\left( W\right),
\end{equation}
and look at the three parts. The bracket in \eqref{eq:dual-split} tends to $0$ in $(\wps)^*$ by Lemma \ref{lem:dual}-(b) applied to $v_n=\hat{u}_n$ and $v=W$.
The middle term tends to $0$ by \eqref{eq:isometry} applied to $\tilde{u}_n$,
\begin{equation*}
    \norm{\mathcal{I}_{\sigma_n}'\left( \hat{u}_n\right)}_{(\wps)^*}=\norm{I_{\mu,0,0}'\left( \tilde{u}_n\right)}_{(\wps)^*} = o_n(1).
\end{equation*}
The last term is $0$ in \textbf{Type I}, since for $\sigma_n=0$, $\mathcal{I}_{0}=I_{\mu,0,0}$ and $v$ weakly solves \eqref{limit-mu}. In \textbf{Type II}, for every $\psi \in \wps$, we have
\begin{equation*}
      \prescript{}{(\wps)^*}{\langle} \mathcal{I}_{\sigma_n}'(w),\psi{\rangle}_{\wps}=-\mu \int_{\rd}\frac{\abs{w}^{p-2}w}{\abs{z+\sigma_n}^{sp}}\psi \dz,
\end{equation*}
and H\"{o}lder's inequality with the pair $(\p,p)$ together with Lemma \ref{lem:far} give
\begin{equation*}
      \norm{\mathcal{I}_{\sigma_n}'(w)}_{(\wps)^*}\le \frac{\mu}{\mu_{d,s,p}^{\frac1p}}\left( \int_{\rd}\frac{\abs{w}^p}{\abs{z+\sigma_n}^{sp}}\dz\right)^{\frac{1}{\p}} = o_n(1).
\end{equation*}
Hence $\norm{I_{\mu,0,0}'(\tilde{\tilde{u}}_n)}_{(\wps)^*} = o_n(1)$. This completes Step 4.

\medskip

\noi \textbf{Step 5:} We show that the procedure stops after a finite number of steps, and the concentrations separate. In Steps 2-4, from a (PS) sequence of $I_{\mu,0,0}$ at level $\eta$ going weakly to $0$, we construct another (PS) sequence at level $\eta-I_{\mu,0,0}(v)$ or $\eta-I_{0,0,0}(w)$ each step lowers the level by at least $\frac{s}{d} \overline{S}_{\mu}^{d/sp}$ since
\begin{align}\label{lowerbound-II}
    I_{\mu,0,0}(v) \ge \frac{s}{d} \overline{S}_{\mu}^{\frac{d}{sp}} \; \text{ and } \; I_{0,0,0}(w) \ge \frac{s}{d} \overline{S}^{\frac{d}{sp}} \ge \frac{s}{d} \overline{S}_{\mu}^{\frac{d}{sp}}.
\end{align}
Let $\{\tilde{u}_n\}$ be any bounded (PS) sequence for $I_{\mu,0,0}$ at a level $\tilde{\eta}$. Now, we note the following
\begin{equation}\label{eq:level}
     I_{\mu,0,0}(\tilde{u}_n)-\frac{1}{p^*_s}\prescript{}{(\wps)^*}{\langle} I_{\mu,0,0}'(\tilde{u}_n),\tilde{u}_n{\rangle}_{\wps}=\left( \frac1p-\frac{1}{p^*_s}\right)[\tilde{u}_n]_{\mu}^p=\frac{s}{d}[\tilde{u}_n]_{\mu}^p \ge 0.
\end{equation}
Letting $n \ra \infty$ in \eqref{eq:level}, we get
\begin{equation}\label{eq:level-non-neg}
     \tilde{\eta}=\frac{s}{d}\lim_{n \ra \infty}[\tilde{u}_n]_{\mu}^p \ge 0.
\end{equation}
The first level is  $\eta_{(0)}:=\eta-I_{\mu,0,f}(u)$ by \eqref{eq:step-1}, and after $\ell$ removals the level is at most $\eta_{(0)}-\ell \frac{s}{d}\overline{S}_{\mu}^{\frac{d}{sp}}$. By \eqref{eq:level-non-neg}, this is non-negative, so
\begin{equation*}
     \ell \le \frac{d}{s}\,\overline{S}_{\mu}^{-\frac{d}{sp}}\left( \eta-I_{\mu,0,f}(\tilde{u})\right),
\end{equation*}
and the procedure stops after finitely many steps, say $k_1$ of \textbf{Type I} and $k_2$ of \textbf{Type II}, with the remaining (PS) sequence converging strongly to $0$. We index these stages by $\ell=1,\dots,k_1+k_2$. Let $B_n^{\ell}$ be the concentration removed at the $\ell$-th part of Steps 2-4, let $W^{\ell}$ be the associated profile, so that $B_n^{\ell}=C_{0,r_n^i}v_i$ with $W^{\ell}=v_i$ for \textbf{Type I} and $B_n^{\ell}=C_{x_n^j,R_n^j}w_j$ with $W^{\ell}=w_j$ for \textbf{Type II}. We now write 
\begin{equation}\label{eq:stages}
     \tilde{u}_n^{(0)}:=\tilde{u}_n, \quad \tilde{u}_n^{(\ell)}:=\tilde{u}_n^{(\ell-1)}-B_n^{\ell}, \quad \ell=1,\dots,k_1+k_2.
\end{equation}
Each stage passes to a subsequence, but since there are finitely many stages, we do not relabel.
Adding the $k_1+k_2$ identities to \eqref{eq:stages} gives, for every $n$,
\begin{gather*}
    \tilde{u}_n=\sum_{\ell=1}^{k_1+k_2}B_n^{\ell}+\tilde{u}_n^{(k_1+k_2)}.
\end{gather*}
The procedure is stopped at stage $k_1+k_2$, meaning that Step 1 yields $\tilde{u}_n^{(k_1+k_2)}\ra 0$ in $\wps$. 
% Since $\tilde{u}_n=u_n-u$ and the $B_n^{\ell}$ are the $C_{0,r_n^i}v_i$ and the $C_{x_n^j,R_n^j}w_j$, this is (iii), the $o_n(1)$ there being $\tilde{u}_n^{(k_1+k_2)}$.

By \eqref{eq:stages}, the first level is $\eta_{(0)}=\eta-I_{\mu,0,f}(u)$, and by \eqref{eq:new-level} each stage removes exactly the energy of the profile obtained, so we have,
\begin{equation}\label{eq:level-step}
     \eta_{(\ell)}=c_{(\ell-1)}-I_{\mu,0,0}\left( W^{\ell}\right) \; \text{ in \textbf{Type I}}, \quad \eta_{(\ell)}=c_{(\ell-1)}-I_{0,0,0}\left( W^{\ell}\right) \; \text{ in \textbf{Type II}}.
\end{equation}
Finally, $I_{\mu,0,0}$ is continuous on $\wps$ and $I_{\mu,0,0}(0)=0$, so $\tilde{u}_n^{(k_1+k_2)}\ra 0$ in $\wps$ gives
\begin{equation}\label{eq:last-level}
      \eta_{(k_1+k_2)}=\lim_{n \ra \infty}I_{\mu,0,0}\left( \tilde{u}_n^{(k_1+k_2)}\right)=I_{\mu,0,0}(0)=0.
\end{equation}
Adding the $k_1+k_2$ equalities \eqref{eq:level-step} and using \eqref{eq:last-level}, at the end and $\eta_{(0)}=\eta-I_{\mu,0,f}(u)$ at the start, the intermediate levels cancel and we are left with
\begin{equation*}
      0=\eta_{(k_1+k_2)}=\eta-I_{\mu,0,f}(u)-\sum_{i=1}^{k_1}I_{\mu,0,0}(v_i)-\sum_{j=1}^{k_2}I_{0,0,0}(w_j).
\end{equation*}
%which is (iv). Part (i) is proved in Step 1, part (ii) and part (v) in Step 5.

We write $C_n^{\ell}:=C_{y_n^{\ell},\rho_n^{\ell}}$, where $(\rho_n^{\ell},y_n^{\ell})$ are the parameters of stage $\ell$, so that $\rho_n^{\ell}:=r_n^{\ell}$ and $y_n^{\ell}:=0$ in \textbf{Type I}, and $\rho_n^{\ell}:=R_n^{\ell}$ and $y_n^{\ell}:=x_n^{\ell}$ in \textbf{Type II}, and $B_n^{\ell}=C_n^{\ell}W^{\ell}$ in both types. First, we note that
\begin{equation}\label{eq:stage-limit}
     \left( C_n^{\ell}\right)^{-1}\tilde{u}_n^{(\ell-1)}\rightharpoonup W^{\ell}\ne 0 \; \text{ in } \wps.
\end{equation}
Second, using \eqref{eq:stages} we note that
\begin{equation}\label{eq:stage-zero}
     \left( C_n^{\ell}\right)^{-1}\tilde{u}_n^{(\ell)}=\left( C_n^{\ell}\right)^{-1}\tilde{u}_n^{(\ell-1)}-W^{\ell}\rightharpoonup 0 \; \text{ in } \wps.
\end{equation}
For $\ell_1<\ell_2$, we consider
\begin{equation}\label{eq:relative}
     \tau_n^{\ell_1\ell_2}:=\frac{\rho_n^{\ell_2}}{\rho_n^{\ell_1}}, \quad \xi_n^{\ell_1\ell_2}:=\frac{y_n^{\ell_2}-y_n^{\ell_1}}{\rho_n^{\ell_1}}, \quad \text{so that}\quad \left( C_n^{\ell_1}\right)^{-1} \circ C_n^{\ell_2}=C_{\xi_n^{\ell_1\ell_2},\,\tau_n^{\ell_1\ell_2}},
\end{equation}
by \eqref{eq:composition}. Now, \eqref{eq:Lambda} for the pair $\left( \tau_n^{\ell_1\ell_2},\xi_n^{\ell_1\ell_2}\right)$ is
\begin{equation}\label{eq:separation}
     \La_n^{\ell_1\ell_2}=\left| \log \left( \frac{\rho_n^{\ell_2}}{\rho_n^{\ell_1}}\right)\right|+\frac{\abs{y_n^{\ell_2}-y_n^{\ell_1}}}{\rho_n^{\ell_2}},
\end{equation}
We need to prove that $\La_n^{\ell_1\ell_2}\ra \infty$ for all $\ell_1<\ell_2$.
We fix $\ell_1$ and argue by induction on $\ell_2$. Let $\ell_2>\ell_1$ and assume $\La_n^{\ell_1 \ell}\ra \infty$ for every $\ell$ with $\ell_1<\ell<\ell_2$. If $\ell_2=\ell_1+1$, then there is nothing to assume. Suppose $\La_n^{\ell_1\ell_2}\not\ra \infty$. Then up to a subsequence, $\La_n^{\ell_1\ell_2}\le M$ for some $M>0$, so $\tau_n^{\ell_1\ell_2}\in \left[ e^{-M},e^M\right]$ and $\abs{\xi_n^{\ell_1\ell_2}}\le \tau_n^{\ell_1\ell_2}M \le Me^M$, and again up to a further subsequence,
\begin{equation}\label{eq:escape}
    \tau_n^{\ell_1\ell_2}\ra \tau \in (0,\infty), \quad \xi_n^{\ell_1\ell_2}\ra \xi \in \rd.
\end{equation}
We now study $\tilde{u}_n^{(\ell_2-1)}$ of stage $\ell_1$ in two ways, and get two different weak limits.

\underline{ Going forward from stage $\ell_1$}: By \eqref{eq:stages} and \eqref{eq:relative}, 
\begin{gather}\label{eq:forward}
\left( C_n^{\ell_1}\right)^{-1}\tilde{u}_n^{(\ell_2-1)}=\left( C_n^{\ell_1}\right)^{-1}\tilde{u}_n^{(\ell_1)}-\sum_{\ell=\ell_1+1}^{\ell_2-1}C_{\xi_n^{\ell_1\ell},\,\tau_n^{\ell_1\ell}}W^{\ell}.
\end{gather}
The first term weakly goes to zero, from \eqref{eq:stage-zero}. 
Indeed, fix $\ell$ with $\ell_1<\ell<\ell_2$. By the induction hypothesis,
$$\Lambda_n^{\ell_1\ell}=\left|\log \tau_n^{\ell_1\ell}\right|+\frac{|\xi_n^{\ell_1\ell}|}{\tau_n^{\ell_1\ell}} \longrightarrow \infty.$$ Thus condition \textup{(i)} of Proposition \ref{lem:2} is satisfied by the
sequence of translation-dilation parameters $(\xi_n^{\ell_1\ell},\tau_n^{\ell_1\ell})$. Applying the implication \textup{(i)}$\Longrightarrow$\textup{(ii)} of Proposition \ref{lem:2} to the fixed function $W^\ell\in \wps$ gives
$C_{\xi_n^{\ell_1\ell},\tau_n^{\ell_1\ell}}W^\ell \rightharpoonup0 \text{ in }\wps.$ Since the sum in \eqref{eq:forward} contains only finitely many terms, the whole sum converges weakly to zero. Together with \eqref{eq:forward}, this gives us
%$(C_n^{\ell_1})^{-1}\widetilde u_n^{(\ell_2-1)} \rightharpoonup0 \text{ in }\wps$, we have
\begin{equation*}
     \left( C_n^{\ell_1}\right)^{-1}\tilde{u}_n^{(\ell_2-1)}\rightharpoonup 0 \; \text{ in } \wps.
\end{equation*}

\underline{Going backward from stage $\ell_2$}: By \eqref{eq:relative},
\begin{equation}
     \left( C_n^{\ell_1}\right)^{-1}\tilde{u}_n^{(\ell_2-1)}=C_{\xi_n^{\ell_1\ell_2},\,\tau_n^{\ell_1\ell_2}}\left(\left( C_n^{\ell_2}\right)^{-1}\tilde{u}_n^{(\ell_2-1)}\right),
\end{equation}
where the inner sequence weakly goes to $W^{\ell_2}$ by \eqref{eq:stage-limit}. The parameters converge by \eqref{eq:escape}. So we apply Lemma \ref{lem:1}-(b) and obtain
\begin{equation}\label{eq:backward}
    \left( C_n^{\ell_1}\right)^{-1}\tilde{u}_n^{(\ell_2-1)}\rightharpoonup C_{\xi,\tau}W^{\ell_2} \; \text{ in } \wps.
\end{equation}
Now $W^{\ell_2}\ne 0$ and $C_{\xi,\tau}$ is one-to-one, so $C_{\xi,\tau}W^{\ell_2}\ne 0$, and \eqref{eq:forward} contradicts \eqref{eq:backward}. This proves $\La_n^{\ell_1\ell_2}\ra \infty$, and the induction is complete.  
\qed
% If both stages are of \textbf{Type I}, then $y_n^{\ell_1}=y_n^{\ell_2}=0$ and \eqref{eq:separation} is the first part of (vi). If both are of \textbf{Type II}, it is the second part. If they are of different types, the separation holds anyway.

\section{Existence of weak solutions}\label{se-5} 
Having established the global compactness results for $0\leq\al <sp$, we now use them to construct a first positive solution. The argument is variational: we minimize the truncated functional in a region separated from the natural constraint $\Sigma^\alpha$, and then use the profile decomposition to recover strong convergence. For $\al \in [0, sp)$, we consider the following functional 
\begin{align*}
J_{\mu,\al,f}(u) =\frac{1}{p}[u]_{\mu}^{p}-\frac{1}{p_s^*(\al)}\int_{\rd}\frac{(u^+)^{p_s^*(\al)}}{|x|^{\al}}\,{\rm d}x -\prescript{}{(\wps)^*}{\langle}f,u{\rangle}_{\wps}, \; \forall \, u \in \wps.   
\end{align*}
If $u$ is a critical point of $J_{\mu,\al,f}$, then using $f \ge 0$, one can verify that $u$ is non-negative and weakly solves \eqref{MainEq}. 
To establish the existence of a critical point for $J_{\mu,\al,f}$, we consider the following sets
\begin{equation}\label{sigma-sets}
\begin{split}
\Sigma_1^{\al} &\coloneqq \big\{u\in\wps : u=0\mbox{ or }\Psi_{\al}(u)>0\big\}, \quad \Sigma_2^{\al} \coloneqq \big\{u\in\wps : \Psi_{\al}(u)<0\big\},\\
&\quad \quad \quad \Sigma^{\al} \coloneqq \big\{u\in\wps \setminus \{0\}  :\Psi_{\al}(u)=0\big\}.
\end{split}
\end{equation}
where $\Psi_{\al} :\wps (\rd)\to\R$ is defined by
\begin{align}\label{associated-functional}
  \Psi_{\al}(u) \coloneqq [u]_{\mu}^p-\left(\frac{p_s^*(\al)-1}{p-1}\right)\int_{\rd}\frac{|u|^{p_s^*(\al)}}{|x|^{\al}}\dx,  
\end{align}
and set
\begin{align}\label{infimum}
 c_{0}^{\al} \coloneqq \inf_{\Sigma_1^{\al}} J_{\mu,\al,f}(u), \text{ and } c_1^{\al} \coloneqq \inf_{\Sigma^{\al}} J_{\mu,\al,f}(u).   
\end{align}
We also recall the notation
$$\norm{u}_{\al} \coloneqq \left( \int_{\rd}\frac{\abs{u}^{p^*_s(\al)}}{\abs{x}^{\al}}\dx\right)^{\frac{1}{p^*_s(\al)}}, \; \forall \,u \in \wps.$$
The first step is to separate the minimizing region from its boundary. The following energy gap will also be used later in the min-max construction of the second solution.

\begin{lemma}\label{lem:psi-cont}
Let $\al \in [0,sp)$ and let $\{a_n\},\{b_n\}\subset \wps$ be bounded. Then
\begin{equation}\label{eq:psi-cont}
    [a_n-b_n]_{s,p}\ra 0 \Longrightarrow  \Psi_{\al}(a_n)-\Psi_{\al}(b_n)\ra 0.
\end{equation}
\end{lemma}
\begin{proof}
We recall from \eqref{equivalent-norm} that $[\cdot]_{\mu}$ is a norm on $\wps$. Since $\mu>0$ we have $[u]_{\mu}\le [u]_{s,p}$, and \eqref{HS1} gives $\norm{u}_{\al}\le S_{\mu}^{-1/p}[u]_{\mu}\le S_{\mu}^{-1/p}[u]_{s,p}$. The reverse triangle inequality therefore gives,
\begin{equation}\label{eq:psi-cont-1}
\left| [a_n]_{\mu}-[b_n]_{\mu}\right| \le [a_n-b_n]_{\mu}\le [a_n-b_n]_{s,p}\ra 0,
\end{equation}
and, similarly, 
\begin{align}\label{eq:psi-cont-1-1}
   \left| \norm{a_n}_{\al}-\norm{b_n}_{\al}\right| \le \norm{a_n-b_n}_{\al}\le S_{\mu}^{-\frac1p}[a_n-b_n]_{s,p}\ra 0. 
\end{align}
Consider $T:=\left( 1+S_{\mu}^{-1/p}\right)\sup_n \max \left\{ [a_n]_{s,p},[b_n]_{s,p}\right\}<\infty$, so that $[a_n]_{\mu}$, $[b_n]_{\mu}$, $\norm{a_n}_{\al}$, and $\norm{b_n}_{\al}$ lie in $[0,T]$. For $q>1$ the map $t \mapsto t^q$ is Lipschitz on $[0,T]$, with constant $qT^{q-1}$. Applying this with $q=p$ to \eqref{eq:psi-cont-1} and with $q=p^*_s(\al)$ to \eqref{eq:psi-cont-1-1} gives
\begin{equation*}
     \left| [a_n]_{\mu}^p-[b_n]_{\mu}^p\right|\ra 0, \quad \left| \int_{\rd}\frac{\abs{a_n}^{p^*_s(\al)}}{\abs{x}^{\al}}\dx-\int_{\rd}\frac{\abs{b_n}^{p^*_s(\al)}}{\abs{x}^{\al}}\dx\right|\ra 0,
\end{equation*}
and \eqref{eq:psi-cont} follows from the definition of $\Psi_{\al}$.
\end{proof}

We require the following strict inequality for a weak solution to exist. 
\begin{lemma}\label{strict}
Let $\al \in [0, sp)$. If $$\norm{f}_{(\wps)^*} \le \left( \frac{1}{p} -  \frac{p-1}{p^*_s(\al)(p^*_s(\al) -1)}\right) \left( \frac{p-1}{p^*_s(\al)-1} S_{\mu,\al}^{\frac{p^*_s(\al)}{p}} \right)^{\frac{p-1}{p^*_s(\al) - p}},$$
then $c_{0}^{\al} < c_{1}^{\al}$.
\end{lemma}
\begin{proof}
For brevity, we call the R.H.S. $\var_0$. If $u \in \Sigma^{\al}$, then 
    \begin{align*}
        [u]_{\mu}^p = \frac{p^*_s(\al) -1}{p-1} \int_{\rd} \frac{|u|^{p^*_s(\al)}}{|x|^{\al}} \dx,
    \end{align*}
 and \eqref{HS1} yield the following 
 \begin{align*}
     [u]_{\mu}^{p^*_s(\al) - p} \ge \frac{p-1}{p^*_s(\al)-1} S_{\mu,\al}^{\frac{p^*_s(\al)}{p}} \Longrightarrow [u]_{\mu} \ge \left( \frac{p-1}{p^*_s(\al)-1} S_{\mu,\al}^{\frac{p^*_s(\al)}{p}} \right)^{\frac{1}{p^*_s(\al)-p}} := \delta \, ( \text{say}). 
 \end{align*}
 Moreover, for $u \in \Sigma^{\al}$, we get 
 \begin{align*}
     J_{\mu, \al, f} & = \left( \frac{1}{p} - \frac{p-1}{p^*_s(\al)(p^*_s(\al)-1)} \right) [u]_{\mu}^p - \prescript{}{(\wps)^*}{\langle} f, u {\rangle}_{\wps} \\
     & \ge [u]_{\mu} \left( \left( \frac{1}{p} - \frac{p-1}{p^*_s(\al)(p^*_s(\al)-1)} \right) [u]_{\mu}^{p-1} - \norm{f}_{(\wps)^*} \right) \\ 
     & \ge \delta \left( \left( \frac{1}{p} - \frac{p-1}{p^*_s(\al)(p^*_s(\al)-1)} \right) \delta^{p-1} - \norm{f}_{(\wps)^*} \right) = \delta \left( \var_{0} - \norm{f}_{(\wps)^*} \right)>0,
 \end{align*}
 which implies $c_1^{\al} \ge 0$. Next, from 
 \begin{align*}
     &J_{\mu, \al, f}(tu) = \frac{t^p}{p}[u]_{\mu}^p - \frac{t^{p^*_s(\al)}} {p^*_s(\al)} \int_{\rd} \frac{(u^+)^{p^*_s(\al)}}{|x|^{\al}} \dx - t\left<f,u\right>, \text{ and } \\
     &\Psi_{\al}(tu) = t^p[u]_{\mu}^p - \left(\frac{p^*_s(\al) - 1}{p-1}\right) t^{p^*_s(\al)} \int_{\rd} \frac{|u|^{p^*_s(\al)}}{|x|^{\al}} \dx,
 \end{align*}
 we note that $ J_{\mu, \al, f}(tu) <0$ and $\Psi_{\al}(tu)>0$ for every $t <<1$. Therefore, $c_0^{\al}\le J_{\mu, \al, f}(tu)<0 \le c_{1}^{\al}$, as  required. 
\end{proof}

\subsection{Existence of a positive solution: The case $\al>0$}
We first treat the case $0<\alpha<sp$. Theorem \ref{PS-decomposition} will be used to show that the minimizing Palais-Smale sequence can not lose energy through a nontrivial Hardy-Sobolev profile. 
\medskip

\noi \textbf{Proof of Theorem \ref{existence}}: We show that $J_{\mu,\al,f}$ has a critical point $u_{\al} \in \Sigma_1^{\al}$ with $J_{\mu,\al,f}(u_{\al})=c_0^{\al}$. We decompose the proof into a few steps.
		
\noi \textbf{Step~1:} In this step, we see $c_0^{\al}>-\infty$. 
Since $J_{\mu,\al,f}(u)\geq I_{\mu,\al,f}(u)$, it is enough to show that $I_{\mu,\al,f}$ is bounded from below. Note that for all $u \in \Sigma_1^{\al}$,
\begin{align}\label{31-7-3}
  I_{\mu,\al,f}(u)\geq \left( \frac{1}{p} - \frac{p-1}{p^*_s(\al)(p^*_s(\al)-1)} \right) [u]_{\mu}^p - \|f\|_{(\wps)^*} [u]_{\mu}.   
\end{align}
As the RHS is $p^{\text{th}}$ power function in $[u]_{\mu}$ (for $p>1$), $I_{\mu,\al,f}$ is bounded from below.

\noi \textbf{Step~2:} In this step, we show that there exists a bounded nonnegative (PS)-sequence $\{u_n\} \subset \Sigma_1^{\al}$ for $J_{\mu,\al,f}$ at level $c_0^{\al}$. Let $\{u_n\}\subset \Sigma_1^{\al}$ be such that $J_{\mu,\al,f}(u_n)\to c_0^{\al}$.  Applying Ekeland's variational principle on the complete metric space $\overline{\Sigma_1^{\al}}$, we get 
\begin{align}\label{ekeland}
    &c_0^{\al} \le J_{\mu,\al,f}(u_n) \le  c_0^{\al} + \frac{1}{n}, \text{ and } \no\\
    &J_{\mu,\al,f}(u_{n})\leq J_{\mu,\al,f}(v)+\frac1n[v-u_n]_{\mu}, \; \forall \, v\in \overline{\Sigma_1^{\al}}, \, v\neq u_{n}.
\end{align}
Using Taylor's expansion, 
\begin{align*}
   J_{\mu,\al,f}(v)=J_{\mu,\al,f}(u_{n})+ \prescript{}{(\wps)^*}{\langle} J_{\mu,\al,f}'(u_n), v-u_{n} {\rangle}_{\wps}+o([v-u_{n}]_{\mu}).
\end{align*}
For $t>0$ and $\phi \in \wps$, we consider $v=u_n+ t\phi$. From the Taylor's expansion,
\begin{align*}
    \Psi_{\al}(u_n \pm t \phi) = \Psi_{\al}(u_n) \pm t \prescript{}{(\wps)^*}{\langle} \Psi_{\al}'(u_n), \phi {\rangle}_{\wps} +o_t(1), 
\end{align*}
we observe that, if $u_n \in \Sigma^{\al}$ and $\prescript{}{(\wps)^*}{\langle} \Psi_{\al}'(u_n), \phi\rangle_{\wps} =c>0$, then for every $t>0$, $u_n-t\phi \in \Sigma_2^{\al}$. This forces the requirement of  $u_n \in \Sigma^{\al}_1$ for every large $n$. Since, $\norm{f}_{(\wps)^*} \le \var_0$, using Lemma \ref{strict}, there exists $n_0 \in \mathbb{N}$ such that $J_{\mu,\al,f}(u_n) < c_1^{\al}$ for every $n \ge n_0$. This implies that $u_n \in \Sigma_1^{\al}$ for every $n \ge n_0$. Now, for small $t>0$, we choose $v=u_n \pm t \phi \in \overline{\Sigma^{\al}_1}$ in \eqref{ekeland}, to get the following for every $n \ge n_0$:
$$-\frac tn\leq J_{\mu,\al,f}(u_n \pm t \phi)-J_{\mu,\al,f}(u_n)=\pm t \prescript{}{(\wps)^*}{\langle}J_{\mu,\al,f}'(u_n), \phi \rangle_{\wps}+o_t(1).$$
Dividing by $t$ and letting $t\to 0$, we obtain $$-\frac1n \leq \prescript{}{(\wps)^*}{\langle} J_{\mu,\al,f}'(u_n), \phi {\rangle}_{\wps} \le \frac1n,$$ which implies
 $$\norm{J_{\mu,\al,f}'(u_n)}_{(\wps)^*} \leq\frac1n.$$
Thus, $\{u_n\} $, is a (PS) sequence in $\Sigma_1^{\al}$ for $J_{\mu,\al,f}$ at level $c_0^{\al}$.  Moreover, as $J_{\mu,\al,f}(u)\geq I_{\mu,\al,f}(u)$, from \eqref{31-7-3} it follows that $\{u_n\} $ is a bounded sequence. Therefore, up to a subsequence $u_n\rightharpoonup u_{\al}$ in $\wps$ and $u_n\to u_{\al}$ a.e. in $\rd$. Moreover, 
\begin{align*}
   o_n(1)&=\prescript{}{(\wps)^*}\langle J'_{\mu,\al,f}(u_n), (u_n)^-\rangle_{\wps} \\
   &= \mathcal{A}(u_n,(u_n)^-)-\mu \int_{\rd} \frac{|u_n|^{p-2}u_n (u_n)^{-}}{|x|^{sp}} \dx-\prescript{}{(\wps)^*}\langle f, (u_n)^-\rangle_{\wps} \\ & \leq - [(u_n)^-]_{s,p}^p + \mu \int_{\rd} \frac{(u_n^-)^p}{|x|^{sp}} \dx 
   =-[(u_n)^-]_{\mu}^p,
\end{align*}
where the final inequality follows using $f\ge 0$, and an elementary inequality (see \cite[Lemma A.2]{BrPa}):
\begin{align*}
    \abs{a - b}^{p-2}(a-b)(a^{-} - b^{-}) \le -\abs{a^{-} - b^{-}}^{p}, \text{ for } a,b \in \R.
\end{align*}
Therefore, $(u_n)^- \to 0$ in $\wps$ and thus $(u_{\al})^-=0$ a.e. in $\rd$.  Consequently, without loss of generality, we can assume that $\{u_n\}$ is nonnegative.

\noi \textbf{Step 3:}  In this step, we show that $u_n\to u_{\al}$ in $\wps$ and $u_{\al}\in \overline{\Sigma_1^{\al}}$. Applying Theorem \ref{PS-decomposition}, we get
\begin{align}\label{split}
u_n =  u_{\al} +\sum_{i=1}^{k} C_{r_n^i}\tilde{u}_i + o_n(1), \;\text{ in }\wps,    
\end{align}
where $J_{\mu,\al,f}'(u_{\al}) =0$ and $\tilde{u}_i$ is a nonnegative solution of \eqref{homogeneous}. To prove Step~3, we need to show that $k=0$. Arguing by contradiction, suppose that $k\neq 0$ in \eqref{split}. Therefore, for $1\leq i\leq k$ we have
		\begin{align}\label{12-4-3}
		\Psi_{\al}\left(C_{r_n^i}\tilde{u}_i\right)=[\tilde{u}_i]_{\mu}^p-\left(\frac{p_s^*(\al)-1}{p-1}\right)\int_{\rd}\frac{|\tilde{u}_i|^{p_s^*(\al)}}{|x|^{\al}}\dx=\frac{p-p_s^*(\al)}{p-1}[\tilde{u}_i]_{\mu}^p<0.
		\end{align}			
		From Theorem \ref{PS-decomposition}, we also have
		$$c_0^{\al}=\lim_{n\to\infty} J_{\mu,\al,f} (u_n)=  J_{\mu,\al,f}(u_{\al})+\sum_{i=1}^{k}J_{\mu,\al,0}(\tilde u_i).$$
Since $\tilde u_i$ weakly solves \eqref{homogeneous}, $J_{\mu,\al,0}(\tilde u_i)\geq \frac{sp-\al}{p(d-\al)}S_{\mu,\al}^{\frac{d-\al}{sp-\al}}.$
Consequently, $ J_{\mu,\al,f}(u_{\al})<c_0^{\al}$. Therefore, $u_{\al} \not\in \Sigma_1^{\al}$ (as $\inf_{\Sigma_1^{\al}} J_{\mu,\al,f}(u,v)=:c_0^{\al}$) and
\begin{align}\label{12-4-4}
   \Psi_{\al}(u_{\al})\leq 0. 
\end{align}
Next, we evaluate $\Psi_{\al}\left(u_{\al} +\sum_{i=1}^{k}C_{r_n^i}\tilde{u}_i\right)$. We observe that $u_n \in \Sigma_1^{\al}$ implies $\Psi_{\al}(u_n)> 0$. We denote the sum in \eqref{split} by $U_n$. The sequences $\{u_n\}$ and $\{U_n\}$ are bounded in $\wps$ and $[u_n-U_n]_{s,p}\ra 0$ by \eqref{split}. Using Lemma \ref{lem:psi-cont} gives $\Psi_{\al}(u_n)-\Psi_{\al}(U_n)\ra 0$, which implies
\begin{align}\label{J8}
  0\leq \liminf_{n\rightarrow\infty}\Psi_{\al}(u_n)=\liminf_{n\rightarrow\infty} \Psi_{\al}\left(u_{\al} +\sum_{i=1}^{k} C_{r_n^i}\tilde{u}_i \right).  
\end{align}
Note that from Step~2, we already have $u_{\al}\geq 0$ and $\tilde{u}_i\ge 0$ for all $i$. Therefore, 
\begin{align}\label{12-4-1}
\Psi_{\al}\bigg(u_{\al} +\sum_{i=1}^{k}C_{r_n^i}\tilde{u}_i \bigg)\no &=\bigg{[}u_{\al}+\sum_{i=1}^{k}C_{r_n^i}\tilde{u}_i\bigg{]}^p_{\mu} - \left(\frac{p_s^*(\al)-1}{p-1}\right)\int_{\rd}\bigg|u_{\al}+\sum_{i=1}^{k} C_{r_n^i}\tilde{u}_i \bigg|^{p^*_s(\al)}\frac{\dx}{|x|^{\al}}\no\\
&\underbrace{\leq}_{\text{Claim}} \Psi_{\al}(u_{\al})+ \sum_{i=1}^{k}\Psi_{\al} \left( C_{r_n^i}\tilde{u}_i \right)+ o_n(1).
\end{align}
Observe that  
\begin{align*}
    \int_{\rd}\bigg|u_{\al}+\sum_{i=1}^{k} C_{r_n^i}\tilde{u}_i \bigg|^{p^*_s(\al)}\frac{\dx}{|x|^{\al}} \ge \int_{\rd} \frac{|u_{\al}|^{p^*_s(\al)}}{|x|^{\al}} \dx + \sum_{i=1}^{k} \int_{\rd} \frac{|C_{r_n^i}\tilde{u}_i|^{p^*_s(\al)}}{|x|^{\al}} \dx. 
\end{align*}
To prove the claim in \eqref{12-4-1}, we need  
\begin{align}\label{identity}
  \bigg{[}u_{\al}+\sum_{i=1}^{k}C_{r_n^i}\tilde{u}_i\bigg{]}^p_{\mu}  - [u_{\al}]_{\mu}^p - \sum_{i=1}^{k} \bigg{[} C_{r_n^i}\tilde{u}_i \bigg{]}_{\mu}^p = o_n(1). 
\end{align}
We follow the same approach used in \cite[Proposition 3.3]{NS2025}. From Lemma \ref{Bahri}, we recall 
\begin{align}\label{bahri-coron}
\Bigg{|}\bigg{|}\sum_{j=0}^{k}a_j\bigg{|}^{p-1}\sum_{j=0}^{k}a_j - \sum_{j=0}^{k}\abs{a_j}^{p-1}{a_j}\Bigg{|} \le C(p) \sum_{0\leq i\neq j\leq k}\abs{a_j}^{p-1}\abs{a_i},    
\end{align}
with
\begin{align*}
   a_j= C_{r_n^j}\tilde{u}_j(x)-C_{r_n^j}\tilde{u}_j(y),\text{ where, (as notation) }C_{r_n^0}\tilde{u}_0\coloneqq u_{\al}. 
\end{align*}
Then using \eqref{bahri-coron} we have 
\begin{align}\label{ele-1}
   &\Bigg{|}\sum_{j=0}^{k} \left( C_{r_n^j}\tilde{u}_j(x)-C_{r_n^j}\tilde{u}_j(y) \right) \Bigg{|}^p - \sum_{j=0}^{k} \bigg{|}C_{r_n^j}\tilde{u}_j(x)-C_{r_n^j}\tilde{u}_j(y)\bigg{|}^p \no \\
    &= \left| \sum_{j=0}^{k} a_j\right|^p - \sum_{j=0}^{k} \abs{a_j}^p \le \Bigg{|}\bigg{|}\sum_{j=0}^{k}a_j\bigg{|}^{p-1}\sum_{j=0}^{k}a_j - \sum_{j=0}^{k}\abs{a_j}^{p-1}{a_j}\Bigg{|} \le C(p) \sum_{0\leq i\neq j\leq k}\abs{a_j}^{p-1}\abs{a_i} \no \\
    &= C(p) \sum_{0\leq i\neq j\leq k} \left| C_{r_n^j}\tilde{u}_j(x)-C_{r_n^j}\tilde{u}_j(y) \right|^{p-1}\left|C_{r_n^i}\tilde{u}_i(x)-C_{r_n^i}\tilde{u}_i(y)\right|,
\end{align}
for some constant $C(p)>0$. The first inequality in \eqref{ele-1} follows from 
\begin{align*}
    &\abs{l} - \sum_{j=0}^{k} \abs{a_j}^p \le \abs{l} - \abs{k_2} \le \abs{\abs{l} - \abs{k_2}} \le \abs{l-k_2}, \text{ with } \\
    &l= \bigg{|}\sum_{j=0}^{k}a_j\bigg{|}^{p-1}\sum_{j=0}^{k}a_j, \text{ and } k_2=\sum_{j=0}^{k}\abs{a_j}^{p-1}{a_j}.
\end{align*}
Hence using \eqref{ele-1},
\begin{align*}
    &\text{L.H.S of \eqref{split-2}}\\
    &\le C(p)\sum_{0\leq i\neq j\leq k} \;\iint_{\rd\times\rd} \frac{\left| C_{r_n^j}\tilde{u}_j(x)-C_{r_n^j}\tilde{u}_j(y) \right|^{p-1}\left|C_{r_n^i}\tilde{u}_i(x)-C_{r_n^i}\tilde{u}_i(y)\right|}{|x-y|^{d+sp}} \dxy.
\end{align*}
In view of Theorem \ref{PS-decomposition}, 
\begin{align*}
    \left| \log \left( \frac{r_n^i}{r_n^j} \right) \right| \rightarrow \infty, \text{ for } i \neq j, 1 \le i, j \le k.
\end{align*}
Therefore, applying Proposition \ref{weak-bub-II} leads to 
\begin{align*}
    \sum_{0\leq i\neq j\leq k} \;\iint_{\rd\times\rd} \frac{\left|C_{r_n^j}\tilde{u}_j(x)-C_{r_n^j}\tilde{u}_j(y)\right|^{p-1}\left|C_{r_n^i}\tilde{u}_i(x)-C_{r_n^i}\tilde{u}_i(y)\right|}{|x-y|^{d+sp}} \dxy = o_n(1).
\end{align*}
Thus, 
\begin{align}\label{split-2}
  \bigg{[}u_{\al}+\sum_{i=1}^{k}C_{r_n^i}\tilde{u}_i\bigg{]}^p_{s,p}  - [u_{\al}]_{s,p}^p - \sum_{i=1}^{k} \bigg{[} C_{r_n^i}\tilde{u}_i \bigg{]}_{s,p}^p = o_n(1). 
\end{align}
holds. Further, using \eqref{split}, $u_n\rightharpoonup u_{\al}$ in $\wps$, $u_n\to u_{\al}$ a.e. in $\rd$, and applying Lemma \ref{convergence-BL}-(b), we see that  
\begin{align*}
    \displaystyle \int_{\rd} \frac{\abs{u_{\al}+\sum_{i=1}^{k}C_{r_n^i}\tilde{u}_i}^{p^*_s(\al)}}{\abs{x}^{\al}} \dx -  \int_{\rd} \frac{\abs{u_{\al}}^{p^*_s(\al)}}{\abs{x}^{\al}} \dx - \int_{\rd} \frac{\abs{\sum_{i=1}^{k}C_{r_n^i}\tilde{u}_i}^{p^*_s(\al)}}{\abs{x}^{\al}} \dx = o_n(1).
\end{align*}
Therefore, \eqref{identity} holds, and we obtain \eqref{12-4-1}. Combining \eqref{12-4-3} and \eqref{12-4-4} with \eqref{12-4-1} we get a contradiction to \eqref{J8} for large enough $n \in \N$. Therefore, $k=0$ in \eqref{split}, which implies $u_n\to u_{\al}$ in $\wps$ as $n\to\infty$. Consequently, $\Psi_{\al}(u_n)\to \Psi_{\al}(u_{\al})$ as $n\to\infty$, which in turn implies $ u_{\al}\in \overline{\Sigma_1^{\al}}$. Further, applying the strong maximum principle for the $(s,p)$-super harmonic function (see \cite[Theorem 1.2]{DQ}), we get $u_{\al}>0$ a.e. in $\rd$.
\qed

\subsection{Existence of a positive solution: The case $\alpha=0$}  We now turn to $\alpha=0$. The variational construction of the minimizing sequence is the same, but Theorem \ref{PS-decomposition-II} allows two types of profiles. We therefore first establish a splitting property for the Hardy term that is compatible with both Hardy and pure Sobolev bubbles.
\begin{lemma}\label{lem:hardy-split}
Let $\tilde{u}_0,\tilde{u}_1,\dots,\tilde{u}_{k_1},\tilde{U}_0,\tilde{U}_1,\dots,\tilde{U}_{k_2} \in \wps$. Let $\{r_n^i\}$, $\{R_n^j\}\subset \R^+$ and $\{x_n^j\}\subset \rd$ satisfy
\begin{equation*}
      \abs{\log (r_n^i)}\ra \infty, \; \left| \log \left( \frac{r_n^i}{r_n^{i'}}\right)\right|\ra \infty, \;(1 \le i \ne i' \le k_1), \; \text{ and } \; \frac{\abs{x_n^j}}{R_n^j}\ra \infty \;(1\le j \le k_2).
\end{equation*}
We consider $\displaystyle U_n:=u_0+\sum_{i=1}^{k}C_{0,r_n^i}\tilde{u}_i+\sum_{j=1}^{k_2}C_{x_n^j,R_n^j}\tilde{U}_j$. Then the following identity holds:
\begin{align}\label{z0-hardy-split}
    \int_{\rd}\frac{\abs{U_n}^p}{\abs{x}^{sp}}\dx = \int_{\rd}\frac{\abs{u_0}^p}{\abs{x}^{sp}}\dx+\sum_{i=1}^{k_1}\int_{\rd}\frac{\abs{\tilde{u}_i}^p}{\abs{x}^{sp}}\dx+o_n(1).
\end{align}
\end{lemma}

\begin{proof}
 Applying Lemma \ref{Bahri} with $q=p$, $\ell=k_1+k_2$ and, at each fixed $x \in \rd$, with the numbers $u_0(x)$, $C_{0,r_n^i}v_i(x)$ and $C_{x_n^j,R_n^j}\tilde{U}_j(x)$. Dividing by $\abs{x}^{sp}$ and integrating over $\rd$ we obtain 
\begin{equation*}
      \left| \int_{\rd}\frac{\abs{U_n}^p}{\abs{x}^{sp}}\dx-\int_{\rd}\frac{\abs{u_0}^p}{\abs{x}^{sp}}\dx - \sum_{i=1}^{k_1}\int_{\rd}\frac{\abs{C_{0,r_n^i}\tilde{u}_i}^p}{\abs{x}^{sp}}\dx-\sum_{j=1}^{k_2} \int_{\rd}\frac{\abs{C_{x_n^j,R_n^j}\tilde{U}_j}^p}{\abs{x}^{sp}}\dx \right| \le C \sum \mathcal{H}(\cdot,\cdot),
\end{equation*}
where $\mathcal{H}$ is defined in \eqref{H-function}, and the last sum runs over all ordered pairs of two different members of the family
\begin{equation}\label{eq:family}
     \mathcal{F}_n:=\left\{ u_0,\; C_{0,r_n^1}\tilde{u}_1,\dots,C_{0,r_n^{k_1}}\tilde{u}_{k_1},\; C_{x_n^1,R_n^1}\tilde{U}_1,\dots,C_{x_n^{k_2},R_n^{k_2}}\tilde{U}_{k_2} \right\}.
\end{equation}
In the R.H.S, each of these terms tends to $0$. Applying Lemma \ref{lem:3}-(a) if the pair is $(u_0,C_{0,r_n^i}\tilde{u}_i)$, applying Lemma \ref{lem:3}-(b) if the pair is $(C_{0,r_n^i}\tilde{u}_i,C_{0,r_n^{i'}}\tilde{U}_{i'})$ with $i \ne i'$, and applying Lemma \ref{lem:3}-(c) whenever one member of the pair is $C_{x_n^j,R_n^j}\tilde{U}_j$.
Finally, using \eqref{eq:hardy-scale}, $$\int_{\rd}\frac{\abs{C_{0,r_n^i}\tilde{u}_i}^p}{\abs{x}^{sp}}\dx=\int_{\rd}\frac{\abs{\tilde{u}_i}^p}{\abs{x}^{sp}}\dx,$$ and using \eqref{eq:two-integrals} and Lemma \ref{lem:far}-(b), $$\displaystyle\int_{\rd}\frac{\abs{C_{x_n^j,R_n^j}\tilde{U}_j}^p}{\abs{x}^{sp}}\dx=o_n(1).$$ Therefore, the identity \eqref{z0-hardy-split} holds. 
\end{proof}

\noi \textbf{Proof of Theorem \ref{existence-II}:} 
With this Hardy splitting available, the argument used for $0<\alpha<sp$ can now be adapted to the two-profile decomposition
of Theorem \ref{PS-decomposition-II}. We consider the following functionals 
\begin{align*}
    &J_{\mu,0,f}(u)=\frac{1}{p}[u]_{\mu}^p-\frac{1}{p^*_s}\int_{\rd}(u^+)^{p^*_s}\dx-\prescript{}{(\wps)^*}{\langle}f,u{\rangle}_{\wps},  \; \forall \, u \in \wps, \text{ and } \\ & \psi_0(u) \coloneqq [u]_{\mu}^p-\left( \frac{p^*_s-1}{p-1}\right)\int_{\rd}\abs{u}^{p^*_s}\dx, \; \forall \, u \in \wps.
\end{align*}
In view of \eqref{sigma-sets}, let $\Sigma_1^{0},\Sigma_2^{0},\Sigma^{0}$ and $c_0^0,c_1^0$ be defined from $\psi_0$, and $\Sigma_1^{\al},\Sigma_2^{\al},\Sigma^{\al}$ and $c_0^{\al},c_1^{\al}$ be defined from $\Psi_{\al}$. In view of the bound \eqref{31-7-3} with $\al=0$, we get
\begin{equation}\label{eq:below}
      J_{\mu,0,f}(u)\ge I_{\mu,0,f}(u)\ge \left( \frac{1}{p}-\frac{p-1}{p^*_s(p^*_s-1)}\right)[u]_{\mu}^p-\norm{f}_{(\wps)^*}[u]_{\mu}, \quad \forall\, u \in \overline{\Sigma_1^0}.
\end{equation}
The above inequality implies that $J_{\mu,0,f}$ is bounded from below on $\overline{\Sigma_1^0}$.
Further, since  $\norm{f}_{(\wps)^*} \le \var_0$, using Lemma \ref{strict} and Ekeland's variational principle, we can find a bounded non-negative sequence $\{u_n\} \subset \Sigma_1^0$ so that
$$ \norm{J_{\mu,0,f}'(u_n)}_{(\wps)^*}\le \frac1n, \quad \forall\, n \ge n_0.$$
Now, applying Theorem \ref{PS-decomposition-II} we have the following decomposition:
\begin{equation}\label{eq:split}
      u_n=u_0+\sum_{i=1}^{k_1}C_{0,r_n^i}\tilde{u}_i+\sum_{j=1}^{k_2}C_{x_n^j,R_n^j}\tilde{U}_j+o_n(1) \quad \text{in } \wps,
\end{equation}
where $J_{\mu,0,f}'(u_0)=0$, $u_0 \ge 0$, each $\tilde{u}_i$ weakly solves \eqref{limit-mu} and each $\tilde{U}_j$ weakly solves \eqref{limit-zero}. We will show $k_1=k_2=0$. Our proof, again, is by contradiction. 

% We have $u_0 \ge 0$, and each $\tilde{u}_i$ and each $\tilde{U}_j$ is nonnegative. Again, we have $\left(C_{0,r_n^i}\right)^{-1}(u_n)^-\ra 0$ and $\left(C_{x_n^j,R_n^j}\right)^{-1}(u_n)^-\ra 0$ in $\wps$, therefore each $\tilde{u}_i$ and each $\tilde{U}_j$ is a weak limit of nonnegative functions, and the nonnegative cone is convex and closed, hence weakly closed. 

If $v$ weakly solves \eqref{limit-mu} then $[v]_{\mu}^p=\int_{\rd}\abs{v}^{p^*_s}\dx$, and $[C_{r,0}v]_{\mu}=[v]_{\mu}$ hold using \eqref{eq:invariance} and \eqref{eq:hardy-scale}.
So, 
 \begin{equation}\label{eq:psi-typeI}
     \psi_0\left(C_{0,r_n^i}\tilde{u}_i\right)=\frac{p-p^*_s}{p-1}[\tilde{u}_i]_{\mu}^p<0.
 \end{equation}
If $w$ weakly solves \eqref{limit-zero}, then $[w]_{s,p}^p=\int_{\rd}\abs{w}^{p^*_s}\dx$, and by (\eqref{eq:two-integrals}, with $\tilde{\al}=sp$ together with Lemma \ref{lem:far}-(b) give $$[C_{x_n^j,R_n^j}\tilde{U}_j]_{\mu}^p=[\tilde{U}_j]_{s,p}^p+o_n(1).$$ Hence,
\begin{equation}\label{eq:psi-typeII}
     \psi_0\left(C_{x_n^j,R_n^j}\tilde{U}_j\right)=\frac{p-p^*_s}{p-1}[\tilde{U}_j]_{s,p}^p+o_n(1)<0, \quad \text{for } n \text{ large}.
\end{equation}
From Theorem \ref{PS-decomposition-II}, we have 
\begin{equation*}
     c_0^0=\lim_{n \ra \infty}J_{\mu,0,f}(u_n)=J_{\mu,0,f}(u_0)+\sum_{i=1}^{k_1}I_{\mu,0,0}(\tilde{u}_i)+\sum_{j=1}^{k_2}I_{0,0,0}(\tilde{U}_j)\ge J_{\mu,0,f}(u_0)+\frac{s}{d}\overline{S}_{\mu}^{\frac{d}{sp}},
\end{equation*}
since $k_1+k_2 \ge 1$. So $J_{\mu,0,f}(u_0)<c_0^0=\inf_{\Sigma_1^0}J_{\mu,0,f}$. Therefore $u_0 \notin \Sigma_1^0$ and
\begin{equation}\label{eq:psi-u0}
    \psi_0(u_0)\le 0.
\end{equation}
We write $U_n$ for the sum in \eqref{eq:split}, and we have $\psi_0(u_n) > 0$ for every $n \ge n_0$. The sequences $\{u_n\}$ and $\{U_n\}$ are bounded in $\wps$ and $[u_n-U_n]_{s,p}\ra 0$ by \eqref{eq:split}. Using Lemma \ref{lem:psi-cont} with $\al=0$ gives $\psi_0(u_n)-\psi_0(U_n)\ra 0$, and so we have,
\begin{equation}\label{psi-lim-inf}
     0 \le \liminf_{n \ra \infty}\psi_0(u_n)=\liminf_{n \ra \infty}\psi_0(U_n).
\end{equation}
We now claim 
\begin{equation}\label{eq:psi-additive}
     \psi_0(U_n)\le \psi_0(u_0)+\sum_{i=1}^{k}\psi_0 \left( C_{0,r_n^i}\tilde{u}_i\right)+\sum_{j=1}^{k_2}\psi_0 \left( C_{x_n^j,R_n^j}\tilde{U}_j\right)+o_n(1).
\end{equation}
First, since each term is nonnegative and $p^*_s>1$,
\begin{align*}
    \int_{\rd}\abs{U_n}^{p^*_s}\dx \ge \int_{\rd}\abs{u_0}^{p^*_s}\dx+\sum_{i=1}^{k_1}\int_{\rd}\abs{C_{0,r_n^i}\tilde{u}_i}^{p^*_s}\dx+\sum_{j=1}^{k_2}\int_{\rd}\abs{C_{x_n^j,R_n^j}\tilde{U}_j}^{p^*_s}\dx,
\end{align*}
So, we need
\begin{equation}\label{eq:mu-additive}
      [U_n]_{\mu}^p=[u_0]_{\mu}^p+\sum_{i=1}^{k_1}[C_{0,r_n^i}\tilde{u}_i]_{\mu}^p+\sum_{j=1}^{k_2}[C_{x_n^j,R_n^j}\tilde{U}_j]_{\mu}^p+o_n(1).
\end{equation}
For the Gagliardo part of (\ref{eq:mu-additive}), we argue as in \eqref{bahri-coron}-\eqref{split-2}. For the Hardy part of (\ref{eq:mu-additive}), we use Lemma \ref{lem:hardy-split} to have
\begin{equation*}
    \int_{\rd}\frac{\abs{U_n}^p}{\abs{x}^{sp}}\dx=\int_{\rd}\frac{\abs{u_0}^p}{\abs{x}^{sp}}\dx+\sum_{i=1}^{k_1}\int_{\rd}\frac{\abs{C_{0,r_n^i}\tilde{u}_i}^p}{\abs{x}^{sp}}\dx+\sum_{j=1}^{k_2}\int_{\rd}\frac{\abs{C_{x_n^j,R_n^j}\tilde{U}_j}^p}{\abs{x}^{sp}}\dx+o_n(1),
\end{equation*}
since the last sum is itself $o_n(1)$ by (\ref{eq:two-integrals}) and Lemma \ref{lem:far}-(b). By subtracting, \eqref{eq:mu-additive} follows, and \eqref{eq:psi-additive} is proved.
Putting \eqref{eq:psi-typeI}, \eqref{eq:psi-typeII} and \eqref{eq:psi-u0} into \eqref{eq:psi-additive}, we get $\limsup_n \psi_0(U_n)<0$, since $k_1+k_2 \ge 1$ and every term on the right of \eqref{eq:psi-additive} is nonpositive, with at least one of them bounded away from $0$. This contradicts \eqref{psi-lim-inf}. Therefore $k_1=k_2=0$ in \eqref{eq:split}, so $u_n \ra u_0$ in $\wps$. As a result, $\psi_0(u_n)\ra \psi_0(u_0)$, so $u_0 \in \overline{\Sigma_1^0}$, and $J_{\mu,0,f}(u_0)=c_0^0<0$. Finally, the strong maximum principle for $(s,p)$-superharmonic functions (see \cite[Theorem 1.2]{DQ}) gives $u_0>0$ a.e. in $\rd$. \qed

\section{A second positive solution}\label{se-6} The previous section gives a positive solution $u_\alpha$ with energy $c_0^\alpha<0$. We now look for a second critical point at a higher level. The argument has two ingredients. First, the global
compactness theorems show that the Palais-Smale condition is restored below the first bubbling level $c_0^\alpha+\Theta_{\mu,\alpha}$. Second, we construct a path from $u_\alpha$ to the opposite side of the variational constraint, with the maximum energy remaining strictly below this threshold. The resulting min-max level then produces the second solution.
Throughout the section, we use the following notations
\begin{align}\label{eq:q-alpha}
    q \coloneqq p^*_s(\al)=\frac{p(d-\al)}{d-sp}>p, \quad \frac1p-\frac1q=\frac{sp-\al}{p(d-\al)}, \quad \frac{q}{q-p}=\frac{d-\al}{sp-\al},
\end{align}
and we write $S_{\mu}$ for the constant \eqref{best-constant}, which for $\al=0$ is the constant $\overline{S}_{\mu}$. We denote 
\begin{align}\label{eq:Theta}
    \Theta_{\mu,\al}\coloneqq \left( \frac1p-\frac1q\right)S_{\mu}^{\frac{q}{q-p}}=\frac{sp-\al}{p(d-\al)}\,S_{\mu}^{\frac{d-\al}{sp-\al}},
\end{align}
which is the right-hand side of \eqref{fixed amount decrease}. So $\Theta_{\mu,\al}$ is the minimum energy of a nonzero solution of the homogeneous limit problem, for every $\al \in [0,sp)$.

\begin{remark}
For every $\al \in [0,sp)$, let $V_{\al}$ attains $S_{\mu}$, and $V_{\al}>0$ a.e. in $\rd$. We normalize $V_{\al}$ so that
\begin{equation}\label{eq:V-normal}
    [V_{\al}]_{\mu}^p=\int_{\rd}\frac{V_{\al}^{q}}{\abs{x}^{\al}}\dx=:\La.
\end{equation}
Indeed, replacing $V_{\al}$ by $cV_{\al}$ multiplies the two sides of \eqref{eq:V-normal} by $c^p$ and by $c^{q}$, so they become equal for exactly one $c>0$; and if $A$ and $B$ denote the two sides before the normalization, then $c^{q-p}=A/B$ and we have
\begin{equation*}
     \La=c^pA=A^{\frac{q}{q-p}}B^{-\frac{p}{q-p}}=\left( \frac{A}{B^{\frac{p}{q}}}\right)^{\frac{q}{q-p}}=S_{\mu}^{\frac{q}{q-p}},
\end{equation*}
since $AB^{-p/q}=S_{\mu}$ by \eqref{best-constant}. So we do not need to rescale $V_{\al}$. 
   
\end{remark}

\begin{remark}
By \eqref{eq:V-normal}, for every $t \ge 0$,
\begin{equation}\label{eq:V-energy}
      I_{\mu,\al,0}\left( tV_{\al}\right)=\frac{t^p}{p}[V_{\al}]^p_{\mu}-\frac{t^{q}}{q}\int_{\rd}\frac{V_{\al}^{q}}{\abs{x}^{\al}}\dx=\La \left( \frac{t^p}{p}-\frac{t^{q}}{q}\right).
\end{equation}
The function in the bracket has derivative $t^{p-1}-t^{q-1}=t^{p-1}\left( 1-t^{q-p}\right)$, which is positive for $0<t<1$ and negative for $t>1$, because $q>p$. So the function has one maximum, at $t=1$, of value $\Lambda\left(\frac1p-\frac1q\right)$. Therefore, from \eqref{eq:V-energy} with $\La=S_{\mu}^{q/(q-p)}$, we see that
\begin{equation}\label{eq:V-max}
      \max_{t \ge 0}I_{\mu,\al,0}\left( tV_{\al}\right)=\Theta_{\mu,\al}, \; \text{and the maximum is attained only at } t=1.
\end{equation}
This is the exact amount of energy we can gain along a path. 
\end{remark}

\begin{lemma}\label{lem-crit}
Let $u \in \wps$ satisfy $J_{\mu,\al,f}'(u)=0$. Then $u \ge 0$ a.e. in $\rd$ and $J_{\mu,\al,f}(u)\ge c_0^{\al}$.
\end{lemma}
\begin{proof}
The constant sequence $u_n \coloneqq u$ is a (PS) sequence for $J_{\mu,\al,f}$ at the level $J_{\mu,\al,f}(u)$, so $u^-=0$, that is $u \ge 0$.
If $u=0$ then $J_{\mu,\al,f}(0)=0>c_0^{\al}$ and we are done. So let $u \ne 0$. If $\Psi_{\al}(u)\ge 0$ then $u \in \overline{\Sigma_1^{\al}}$ and $J_{\mu,\al,f}(u)\ge c_0^{\al}$. So, we consider the case
\begin{equation}\label{eq:crit-neg}
     \Psi_{\al}(u)<0.
\end{equation}
Define
\begin{equation*}
    \varphi(t):=J_{\mu,\al,f}(tu)=\frac{t^p}{p}A-\frac{t^{q}}{q}B-tC, \text{ for } t>0,
\end{equation*}
with
\begin{align*}
    A \coloneqq [u]_{\mu}^p>0, \, B \coloneqq \int_{\rd}\frac{u^{q}}{\abs{x}^{\al}}\dx>0, \, C \coloneqq \prescript{}{(\wps)^*}{\langle}f,u{\rangle}_{\wps},
\end{align*}
where $u \ge 0$ is used to write $\int_{\rd}\frac{\left( (tu)^+\right)^{q}}{\abs{x}^{\al}}\dx=t^{q}B$, and $A,B>0$ because $u \ne 0$. Then
\begin{equation*}
     \varphi'(t)=h(t)-C, \text{ where } h(t):=t^{p-1}A-t^{q-1}B. 
\end{equation*}
Since $h'(t)=(p-1)t^{p-2}A-(q-1)t^{q-2}B$ vanishes at exactly one positive $t$, namely
\begin{equation}\label{eq:critical-tau}
     \tau \coloneqq \left( \frac{(p-1)A}{(q-1)B}\right)^{\frac{1}{q-p}},
\end{equation}
and $h'>0$ before $\tau$ and $h'<0$ after it, the function $h$ increases on $(0,\tau)$ and decreases on $(\tau,\infty)$.

We now make the following observations:

\noi (i) Using $u \ge 0$ again, $\Psi_{\al}(tu)=t^pA-\frac{q-1}{p-1}t^{q}B$, and for $t>0$ this vanishes exactly when $t^{q-p}=\frac{(p-1)A}{(q-1)B}$, i.e., when $t=\tau$. So $\tau u \in \Sigma^{\al}\subset \overline{\Sigma_1^{\al}}$.

\noi (ii) Assumption \eqref{eq:crit-neg} says $A<\frac{q-1}{p-1}B$, that is $\frac{(p-1)A}{(q-1)B}<1$, and $\tau<1$ follows from \eqref{eq:critical-tau}.

\noi (iii) Since $u$ is a critical point, $\varphi'(1)=0$, i.e., $h(1)=C$. For $t \in (\tau,1)$ we have $h(t)>h(1)=C$, because $h$ decreases on $(\tau,\infty)$ and $\tau<t<1$. Hence $\varphi'>0$ on $(\tau,1)$, so $\varphi$ increases on $[\tau,1]$.

Combining (i), (ii), and (iii), we get
\begin{equation*}
     J_{\mu,\al,f}(u)=\varphi(1)\ge \varphi(\tau)=J_{\mu,\al,f}(\tau u)\ge c_0^{\al}, 
\end{equation*}
as required.
\end{proof}

%The following proposition is a direct application of the global compactness results. It shows that the (PS) condition holds below the first concentration level $c_0^{\al}+\Theta_{\mu,\al}$.
We now combine the lower energy bound for critical points with the global compactness theorems. Since every nontrivial concentration
carries at least the energy $\Theta_{\mu,\alpha}$, compactness is restored below one bubbling energy above $c_0^\alpha$.

\begin{proposition}\label{prop-compact}
Let $\eta<c_0^{\al}+\Theta_{\mu,\al}$ and let $\{u_n\}$ be a (PS) sequence for $J_{\mu,\al,f}$ at the level $\eta$. Then, up to a subsequence, $u_n \ra u$ in $\wps$, where $u$ is a critical point of $J_{\mu,\al,f}$ with $J_{\mu,\al,f}(u)=\eta$.
\end{proposition}
\begin{proof}
The decomposition theorems are about $I_{\mu,\al,f}$, so we first pass to $\{u_n^+\}$. Since $\{u_n^+\}$ is a (PS) sequence for $I_{\mu,\al,f}$ at the same level $\eta$, and $u_n-u_n^+ \ra 0$ in $\wps$. We apply Theorem \ref{PS-decomposition} if $\al>0$, and Theorem \ref{PS-decomposition-II} if $\al=0$. We denote $u$ the weak limit and $N \in \N \cup \{0\}$ the total number of concentrations produced, that is $N=k$ if $\al>0$ and $N=k_1+k_2$ if $\al=0$. In both cases, 
\begin{align}\label{eq:compact-sum}
    \eta=J_{\mu,\al,f}(u)+\sum_{\ell=1}^{N}E_{\ell},
\end{align}
where $E_{\ell}$ is the energy of the $\ell$-th concentration, and where we use that $u$ is nonnegative, being a weak limit of a nonnegative functional, so that $I_{\mu,\al,f}$ and $J_{\mu,\al,f}$ have the same value and the same derivative at $u$.
We now apply two lower bounds on \eqref{eq:compact-sum}. First, $u$ is a critical point of $J_{\mu,\al,f}$, so $J_{\mu,\al,f}(u)\ge c_0^{\al}$ using Lemma \ref{lem-crit}. Second, every $E_{\ell}$ is at least $\Theta_{\mu,\al}$: for $\al>0$ this is \eqref{fixed amount decrease}, and for $\al=0$, we have the bounds \eqref{lowerbound-II}, for all $i,j$, which covers the concentrations of both types. Hence $$\eta \ge c_0^{\al}+N\,\Theta_{\mu,\al}, \text{ for } \al \in [0,sp).$$ If $N \ge 1$, then $\eta \ge c_0^{\al}+\Theta_{\mu,\al}$ contradicts with the hypothesis. So $N=0$, and the decomposition gives $u_n^+ \ra u$ in $\wps$, and since $u_n-u_n^+ \ra 0$, we get $u_n \ra u$ in $\wps$. Finally, from the continuity, $J_{\mu,\al,f}(u)=\eta$ holds.
\end{proof}

\subsection{A path along which the energy stays below the threshold}\label{sec2-path}
It remains to construct a path whose energy stays inside the compactness range of Proposition \ref{prop-compact}. The usual additive path is not suitable for general $p\neq2$; instead, we use the nonlinear path induced by the $\ell^p$ geometry. From Theorem \ref{existence} and Theorem \ref{existence-II} we recall that 
\begin{align}\label{eq:u0}
    u_{\al}>0, \quad J_{\mu,\al,f}'(u_{\al})=0, \quad J_{\mu,\al,f}(u_{\al})=c_0^{\al}<0, \quad \Psi_{\al}(u_{\al})>0.
\end{align}
For $t \ge 0$, we consider the path
\begin{equation}\label{eq:path}
   \gamma(t) \coloneqq \left( u_{\al}^p+t^pV_{\al}^p\right)^{\frac1p}. 
\end{equation}

The next lemma collects the three properties of this path that are needed for the min-max argument: control of the $\wps$ energy, a strict gain in the critical term, and a strict bound below the first bubbling threshold.
\begin{lemma}\label{lem:path}
The following holds:
\begin{enumerate}
    \item[\rm (a)] $\gamma(t)\in \wps$ and $\displaystyle [\gamma(t)]_{\mu}^p \le [u_{\al}]_{\mu}^p+t^p[V_{\al}]_{\mu}^p$ for every $t \ge 0$.
    \item[\rm (b)] The integral $\displaystyle G(t) \coloneqq \int_{\rd}\frac{\gamma(t)^{q}-u_{\al}^{q}-t^{q}V_{\al}^{q}}{\abs{x}^{\al}}\dx$ is finite, and $G(t)>0$ for every $t>0$.
    \item[\rm (c)] For every $t \ge 0$,
    \begin{gather}\label{eq:path-energy}
        J_{\mu,\al,f}\left( \gamma(t)\right) \le J_{\mu,\al,f}(u_{\al})+I_{\mu,\al,0}\left( tV_{\al}\right)-\frac{1}{q}G(t).
    \end{gather}
    \item[\rm (d)] $\gamma:[0,\infty)\ra \wps$ is continuous, and $J_{\mu,\al,f}(\gamma(t))\ra -\infty$ and $\Psi_{\al}(\gamma(t))\ra -\infty$ as $t \ra \infty$.
    \item[\rm (e)] $\displaystyle \sup_{t \ge 0}J_{\mu,\al,f}\left( \gamma(t)\right)<c_0^{\al}+\Theta_{\mu,\al}$.
\end{enumerate}
\end{lemma}
   \begin{proof}
\noi (a) We fix $t \ge 0$ and $x,y \in \rd$ and write $X\coloneqq\left( u_{\al}(x),tV_{\al}(x)\right)$, $Y\coloneqq\left( u_{\al}(y),tV_{\al}(y)\right)$, two points of $\R^2$, and let $\norm{\cdot}$ be the $\ell^p$ norm on $\R^2$. By \eqref{eq:path}, $\gamma(t)(x)=\norm{X}$ and $\gamma(t)(y)=\norm{Y}$, so the reverse triangle inequality $\left| \norm{X}-\norm{Y}\right| \le \norm{X-Y}$ gives us,
\begin{equation*}
      \abs{\gamma(t)(x)-\gamma(t)(y)}\le \left( \abs{u_{\al}(x)-u_{\al}(y)}^p+t^p\abs{V_{\al}(x)-V_{\al}(y)}^p\right)^{\frac1p}.
\end{equation*}
Raising to the power $p$, dividing by $\abs{x-y}^{d+sp}$ and integrating over $\rd \times \rd$,
\begin{equation}\label{eq:path-gagliardo}
       [\gamma(t)]_{s,p}^p \le [u_{\al}]_{s,p}^p+t^p[V_{\al}]_{s,p}^p<\infty.
\end{equation}       
This is the step where the additive path is lost, and there is no cross term at all. Next, \eqref{eq:path} says $\gamma(t)^p=u_{\al}^p+t^pV_{\al}^p$ pointwise, so dividing by $\abs{x}^{sp}$ and integrating gives an equality,
\begin{equation}\label{eq:path-hardy}
    \int_{\rd}\frac{\gamma(t)^p}{\abs{x}^{sp}}\dx=\int_{\rd}\frac{u_{\al}^p}{\abs{x}^{sp}}\dx+t^p\int_{\rd}\frac{V_{\al}^p}{\abs{x}^{sp}}\dx.
\end{equation}
Now multiply \eqref{eq:path-hardy} by $\mu$ and subtract it from \eqref{eq:path-gagliardo}. On the left we get $[\gamma(t)]^p_{\mu}$, and on the right we get $[u_{\al}]^p_{\mu}+t^p[V_{\al}]^p_{\mu}$. Since the subtracted quantity is an equality, we have the inequality. This is (a), and \eqref{eq:path-gagliardo} also shows $\gamma(t)\in \wps$.

\noi (b) Set $\theta\coloneqq\frac{q}{p}>1$. For $A,B>0$ we consider the inequality $(A+B)^{\theta}>A^{\theta}+B^{\theta}$ holds, since $\sigma \mapsto \sigma^{\theta}$ is strictly superadditive on $(0,\infty)$ when $\theta>1$. Take $A:= u_{\al}(x)^p$ and $B:=t^pV_{\al}(x)^p$, both positive almost everywhere when $t>0$ because $u_{\al}>0$ and $V_{\al}>0$. Then $A^{\theta}=u_{\al}^{q}$ and $B^{\theta}=t^{q}V_{\al}^{q}$, so we have
\begin{equation*}
     \gamma(t)^{q}=(A+B)^{\theta}>u_{\al}^{q}+t^{q}V_{\al}^{q}, \text{ a.e. in } \rd.
\end{equation*}
Multiplying by the positive weight $\abs{x}^{-\al}$ preserves the sign, and $G(t)>0$ holds. Further, $G(t)$ is finite because $\left( a^p+b^p\right)^{1/p}\le a+b$ for $a,b \ge 0$, so $\gamma(t)\le u_{\al}+tV_{\al}$ and the integral of $\gamma(t)^q\abs{x}^{-\al}$ is finite by \eqref{HS1}.

\noi (c) Recall the functional $$J_{\mu,\al,f}(v)=\frac1p[v]^p_{\mu}-\frac{1}{q}\int_{\rd}\frac{\left( v^+\right)^{q}}{\abs{x}^{\al}}\dx-\prescript{}{(\wps)^*}{\langle}f,v{\rangle}_{\wps}, \; \forall \, v \in \wps,$$ and $\gamma(t)\ge 0$, so the middle term is $\frac{1}{q}\int_{\rd}\frac{\gamma(t)^{q}}{\abs{x}^{\al}}\dx$. Now we bound the three terms.

The first term is bounded above by (a). The second term is an identity, by the definition of $G$:
\begin{equation*}
    \int_{\rd}\frac{\gamma(t)^{q}}{\abs{x}^{\al}}\dx = \int_{\rd}\frac{u_{\al}^{q}}{\abs{x}^{\al}}\dx+t^{q}\int_{\rd}\frac{V_{\al}^{q}}{\abs{x}^{\al}}\dx+G(t),
\end{equation*}
For the third term, $\gamma(t)\ge u_{\al}$ and $f \ge 0$ give $\prescript{}{(\wps)^*}{\langle}f,\gamma(t){\rangle}_{\wps}\ge \prescript{}{(\wps)^*}{\langle}f,u_{\al}{\rangle}_{\wps}$, and we have
\begin{align*}
    J_{\mu,\al,f}\left( \gamma(t)\right)\le &\;\frac1p \left( [u_{\al}]^p_{\mu}+t^p[V_{\al}]^p_{\mu}\right)-\frac{1}{q}\left( \int_{\rd}\frac{u_{\al}^{q}}{\abs{x}^{\al}}\dx+t^{q}\int_{\rd}\frac{V_{\al}^{q}}{\abs{x}^{\al}}\dx+G(t)\right) \\
    &-\prescript{}{(\wps)^*}{\langle}f,u_{\al}{\rangle}_{\wps},
\end{align*}
and we get \eqref{eq:path-energy}.

\noi (d) Let $t_n \ra t$ in $[0,\infty)$. From \eqref{eq:path}, $\gamma(t_n)(x)\ra \gamma(t)(x)$ for every $x$, so $D\gamma(t_n)\ra D\gamma(t)$ at every point of $\rd \times \rd$. Also, by the pointwise bound proved in (a) and by $\abs{\sigma_1-\sigma_2}^p \le 2^{p-1}\left( \abs{\sigma_1}^p+\abs{\sigma_2}^p\right)$,
\begin{gather*}
     \abs{D\gamma(t_n)-D\gamma(t)}^p \le 2^{p}\abs{Du_{\al}}^p+2^{p}\left( \sup_n t_n^p+t^p\right)\abs{DV_{\al}}^p,
\end{gather*}
and R.H.S is integrable since $u_{\al},V_{\al}\in \wps$ and $\sup_n t_n<\infty$. By dominated convergence $[\gamma(t_n)-\gamma(t)]_{s,p}\ra 0$, which is the continuity of $\gamma$. For large $t$, \eqref{eq:path-energy} 
and $G \ge 0$ give
\begin{equation*}
    J_{\mu,\al,f}\left( \gamma(t)\right)\le c_0^{\al}+\frac{t^p}{p}[V_{\al}]^p_{\mu}-\frac{t^{q}}{q}\La \ra -\infty,
\end{equation*}
since $q>p$. Similarly $\gamma(t)^{q}\ge t^{q}V_{\al}^{q}$ and (a) give
\begin{equation*}
    \Psi_{\al}\left( \gamma(t)\right)\le [u_{\al}]^p_{\mu}+t^p[V_{\al}]^p_{\mu}-\frac{q-1}{p-1}\,t^{q}\La \ra -\infty.
\end{equation*}

\noi (e) We fix $t$. If $t>0$, then \eqref{eq:path-energy}, the bound $I_{\mu,\al,0}(tV_{\al})\le \Theta_{\mu,\al}$, and $G(t)>0$ give
\begin{equation}\label{eq:path-strict}
      J_{\mu,\al,f}\left( \gamma(t)\right)\le c_0^{\al}+\Theta_{\mu,\al}-\frac{1}{q}G(t)<c_0^{\al}+\Theta_{\mu,\al}.
\end{equation}
If $t=0$, then $J_{\mu,\al,f}(\gamma(0))=J_{\mu,\al,f}(u_{\al})=c_0^{\al}<c_0^{\al}+\Theta_{\mu,\al}$. So the strict inequality holds at every point $t \ge 0$. By (d), the map $t \mapsto J_{\mu,\al,f}(\gamma(t))$ is continuous on $[0,\infty)$ and tends to $-\infty$, so it is bounded above and attains its maximum at some $t_* \in [0,\infty)$. So, we have
\begin{equation*}
     \sup_{t \ge 0}J_{\mu,\al,f}\left( \gamma(t)\right)=J_{\mu,\al,f}\left( \gamma(t_*)\right)<c_0^{\al}+\Theta_{\mu,\al}.
\end{equation*}
This finishes the proof.
   \end{proof}

\subsection{The min-max level}\label{sec2:minmax} The path constructed in the previous subsection connects the first
solution to a point on the opposite side of $\Sigma^\alpha$ while remaining below the first bubbling threshold. We now use this path
to define the min-max level and show that it lies simultaneously above the separating level $c_1^\alpha$ and below
$c_0^\alpha+\Theta_{\mu,\alpha}$. By Lemma \ref{lem:path}-(d), we can fix $T>0$ with
\begin{align}\label{eq:endpoint}
    \Psi_{\al}\left( \gamma(T)\right)<0 \; \text{ and } \; J_{\mu,\al,f}\left( \gamma(T)\right)<J_{\mu,\al,f}(u_{\al}),
\end{align}
and we take $\gamma(T)$ as the far end of our paths. Set
\begin{align}\label{eq:Gamma}
    \Gamma \coloneqq \left\{ \beta \in \mathcal{C}\left( [0,1],\wps\right): \beta(0)=u_{\al},\; \beta(1)=\gamma(T)\right\}, \; \text{ and } \; \eta \coloneqq \inf_{\beta \in \Gamma}\max_{t \in [0,1]}J_{\mu,\al,f}(\beta(t)).
\end{align}

The following lemma separates two energy levels $c_0^{\al}$ and $c_1^{\al}$. 
\begin{lemma}\label{lem:mountain}
$c_1^{\al}\le \eta<c_0^{\al}+\Theta_{\mu,\al}$, and $c_1^{\al}>c_0^{\al}$.
\end{lemma}
\begin{proof}
The map $t \mapsto \gamma(tT)$ is in $\Gamma$, so $\eta$ is at most its maximum energy, which is at most $\sup_{t \ge 0}J_{\mu,\al,f}(\gamma(t))$. Part (e) of Lemma \ref{lem:path} bounds this by $c_0^{\al}+\Theta_{\mu,\al}$, strictly. 
    
   Let $\beta \in \Gamma$. By \eqref{eq:u0} we have $\Psi_{\al}(\beta(0))>0$ and by \eqref{eq:endpoint} we have $\Psi_{\al}(\beta(1))<0$, so
    \begin{equation*}
         t_1 \coloneqq \inf \left\{ t \in [0,1]: \Psi_{\al}(\beta(t))<0\right\}
    \end{equation*}
    is well defined and lies in $(0,1)$. The map $t \mapsto \Psi_{\al}(\beta(t))$ is continuous, nonnegative on $[0,t_1)$ (by the definition of $t_1$), and negative at points arbitrarily close to $t_1$ from the right. So $\Psi_{\al}(\beta(t_1))=0$.

To show $\beta(t_1)\in \Sigma^{\al}$, it remains to check $\beta(t_1)\ne 0$, since $\Sigma^{\al}$ excludes the origin. We show that every $w$ with $\Psi_{\al}(w)<0$ is far from $0$. Such $w$ satisfies $$\int_{\rd}\frac{\abs{w}^{q}}{\abs{x}^{\al}}\dx>\frac{p-1}{q-1}[w]_{\mu}^p,$$ and hence $w \ne 0$, so $[w]_{\mu}>0$. Raising the first inequality to the power $\frac{p}{q}$ and using \eqref{HS1},
\begin{equation*}
     [w]_{\mu}^p \ge S_{\mu}\left( \int_{\rd}\frac{\abs{w}^{q}}{\abs{x}^{\al}}\dx\right)^{\frac{p}{q}}>S_{\mu}\left( \frac{p-1}{q-1}\right)^{\frac{p}{q}}\left( [w]_{\mu}^p\right)^{\frac{p}{q}}.
\end{equation*}
Dividing by $\left( [w]^p_{\mu}\right)^{p/q}$ and using $1-\frac{p}{q}=\frac{sp-\al}{d-\al}$, we get
\begin{equation}\label{eq:below-bar}
      [w]_{\mu}>\rho_0:=\left[ S_{\mu}\left( \frac{p-1}{q-1}\right)^{\frac{p}{q}}\right]^{\frac{d-\al}{p(sp-\al)}}>0,
\end{equation}
 and $\rho_0$ depends only on $d$, $s$, $p$, $\al$ and $\mu$. Now take $t \ra t_1^+$ along points where $\Psi_{\al}(\beta(t))<0$: each of them satisfies $[\beta(t)]_{\mu}>\rho_0$, so $[\beta(t_1)]_{\mu}\ge \rho_0>0$ and $\beta(t_1)\ne 0$.

Hence $\beta(t_1)\in \Sigma^{\al}$ and $\max_{t \in [0,1]}J_{\mu,\al,f}(\beta(t))\ge J_{\mu,\al,f}(\beta(t_1))\ge c_1^{\al}$. Taking the infimum over $\beta \in \Gamma$ gives $\eta \ge c_1^{\al}$. Finally, $c_1^{\al}>c_0^{\al}$ is Lemma \ref{strict}.   
\end{proof}

We have now established both ingredients required for the min-max argument: the level lies strictly above the energy of the first solution and strictly below the first concentration threshold. We can therefore complete the proof of the multiplicity result.
\medskip

\noi \textbf{Proof of Theorem \ref{Thm:second}:}
The first solution $u_{\al}$ is given by Theorem \ref{existence} if $\al>0$ and by Theorem \ref{existence-II} if $\al=0$; in both cases, the hypothesis is \ref{eq:threshold2}.
For the second one, we check the geometry of the mountain pass theorem on $\Gamma$. The two ends of the paths have energy $J_{\mu,\al,f}(u_{\al})=c_0^{\al}$ and $J_{\mu,\al,f}(\gamma(T))<c_0^{\al}$, by (\ref{eq:u0}) and (\ref{eq:endpoint}), so the larger of the two is $c_0^{\al}$; and $c_0^{\al}<c_1^{\al}\le \eta$ by Lemma \ref{lem:mountain}. So the min-max level is strictly above both ends. By Lemma \ref{lem:mountain} we also have $\eta<c_0^{\al}+\Theta_{\mu,\al}$, which is exactly the range where we apply Proposition \ref{prop-compact}. It gives, up to a subsequence, $u_n \ra v_{\al}$ in $\wps$ with $J_{\mu,\al,f}'(v_{\al})=0$ and $J_{\mu,\al,f}(v_{\al})=\eta$. It is different from $u_\al$ because $\eta \ge c_1^{\al}>c_0^{\al}=J_{\mu,\al,f}(u_{\al})$. It is not the zero function, because $J_{\mu,\al,f}'(0)=-f \ne 0$. It is nonnegative by Lemma \ref{lem-crit}, so $v_{\al}$ weakly solves \eqref{MainEq}. Since $v_{\al} \ge 0$, $\mu>0$ and $f \ge 0$, every term on the right of \eqref{MainEq} is nonnegative, so $v_{\al}$ is a nonnegative weak supersolution of $(-\Delta_p)^su=0$ which is not identically $0$. Using the strong maximum principle, we get $v_{\al}>0$ a.e. in $\rd$.
\qed \medskip

\noi \textbf{Acknowledgment:}
Nirjan Biswas acknowledges the support of the National Board for Higher Mathematics Postdoctoral Fellowship (0204/16(9)/2024/RD-II/6761).
\bibliographystyle{abbrvnat}

\end{document}